\documentclass[12pt,a4paper,reqno]{amsart}
\usepackage{amsmath}
\usepackage{amsfonts}
\usepackage{amssymb}
\usepackage{bm} 

\newcommand\R{{\mathbf{R}}}
\newcommand\C{{\mathbf{C}}}

\newcommand\D{{\mathbf{D}}}

\newcommand\e{{\mathbf{e}}}
\renewcommand\H{{\mathbf{H}}}

\renewcommand\S{{\mathcal{S}}}

\newcommand\bigO{{\mathcal{O}}}
\renewcommand\L{{\mathcal{L}}}
\newcommand\T{{\mathbf{T}}}
\newcommand\E{{\mathrm{E}}}

\newcommand\Sym{{\operatorname{Sym}}}
\newcommand\Hom{{\operatorname{Hom}}}
\newcommand\Gram{{\operatorname{Gram}}}
\newcommand\Dil{{\operatorname{Dil}}}
\newcommand\Trans{{\operatorname{Trans}}}

\newcommand\Rev{{\operatorname{Rev}}}
\newcommand\Rot{{\operatorname{Rot}}}
\newcommand\Frame{{\operatorname{Frame}}}
\newcommand\Energy{\dot{\mathcal{H}^1}}

\newcommand\loc{{\operatorname{loc}}}
\newcommand\eps{{\varepsilon}}

\theoremstyle{plain}
  \newtheorem{theorem}[subsection]{Theorem}
  
  \newtheorem{claim}[subsection]{Claim}
  \newtheorem{proposition}[subsection]{Proposition}
  \newtheorem{lemma}[subsection]{Lemma}
  \newtheorem{corollary}[subsection]{Corollary}

\theoremstyle{remark}
  \newtheorem{remark}[subsection]{Remark}

\theoremstyle{definition}
  \newtheorem{definition}[subsection]{Definition}

\include{psfig}

\begin{document}

\title[Global regularity of wave maps IV]{Global regularity of wave maps IV.  Absence of stationary or self-similar solutions in the energy class}
\author{Terence Tao}
\address{Department of Mathematics, UCLA, Los Angeles CA 90095-1555}
\email{ tao@@math.ucla.edu}

\vspace{-0.3in}
\begin{abstract}
Using the harmonic map heat flow, we construct an energy class $\Energy$ for wave maps $\phi$ from two-dimensional Minkowski space $\R^{1+2}$ to hyperbolic spaces $\H^m$, and then show (conditionally on a large data well-posedness claim for such wave maps) that no stationary, travelling, self-similar, or degenerate wave maps exist in this energy class.  These results form three of the five claims required in \cite{tao:heatwave} to prove global regularity for such wave maps.  (The conditional claim of large data well-posedness is one of the remaining claims required in \cite{tao:heatwave}.)
\end{abstract}

\maketitle

\section{Introduction}

\subsection{The energy space}

This paper is a technical component of a larger program \cite{tao:heatwave} to establish large data global regularity for the initial value problem for two-dimensional wave maps into hyperbolic spaces. 
A significant portion of this paper will, however, not concern wave maps \emph{per se}, but instead focus on the more mundane issue of constructing an energy space $\Energy$ to hold the initial data for such wave maps, and establishing the basic properties of that space.

To explain this, we quickly recall some notation from \cite{tao:heatwave}.  Fix $m \geq 1$; we allow all implied constants to depend on $m$.  Let
$\H = (\H^m,h) \equiv SO(m,1) / SO(m)$ be the $m$-dimensional \emph{hyperbolic space}, i.e. the simply-connected $m$-dimensional Riemannian manifold of constant negative sectional curvature $-1$.  We define \emph{classical data} to be a pair $\Phi = (\phi_0,\phi_1)$, where $\phi_0: \R^2 \to \H$ is a smooth map which differs from some constant $\phi_0(\infty)$ by a Schwartz function (embedding $\H$ in $\R^{1+m}$ to define the Schwartz space), and $\phi_1: \R^2 \to T \H$ is a Schwartz function such that $\phi_1(x)$ lies in the tangent plane $T_{\phi_0(x)} \H$ of $\H$ at $\phi_0(x)$ for every $x \in \R^2$, and let $\S$ be the space of all classical data; this can be given the structure of a topological space by using the Schwartz topology.  With regards to wave maps, one should interpret $\phi_0$ and $\phi_1$ as being the initial position and initial velocity respectively of a (classical) wave map at some time.  We observe the four symmetries
\begin{align}
\Trans_{x_0}: (\phi_0(x), \phi_1(x)) &\mapsto (\phi_0(x-x_0), \phi_1(x-x_0)) \label{space-trans-data}\\
\Rev: (\phi_0(x), \phi_1(x)) &\mapsto (\phi_0(x), -\phi_1(x)) \label{time-reverse-data}\\
\Rot_U: (\phi_0(x), \phi_1(x)) &\mapsto (U\phi_0(x), dU(\phi_0(x))(\phi_1(x))) \label{rotate-data}  \\
\Dil_\lambda: (\phi_0(x), \phi_1(x)) &\mapsto (\phi_0(\frac{x}{\lambda}), \frac{1}{\lambda} \phi_1(\frac{x}{\lambda})) \label{scaling-data}
\end{align}
of spatial translation, time reversal, target rotation, and dilation that act continuously on $\S$, where $x_0 \in \R^2$, $U \in SO(m,1)$, and $\lambda > 0$.

Given any classical initial data $\Phi = (\phi_0,\phi_1)$, one can form the \emph{stress-energy tensor} $\T_{\alpha \beta} = \T(\Phi)_{\alpha \beta}$ for $\alpha,\beta = 0,1,2$ by the formula
\begin{equation}\label{stress-def}
 \T_{\alpha \beta} = \Gram_{\alpha \beta} - \frac{1}{2} g_{\alpha \beta} \operatorname{tr}(\Gram)
\end{equation}
where $g_{\alpha \beta}$ is the Minkowski metric $dg^2 = -dt^2 + dx_1^2 + dx_2^2$ (with the usual raising, lowering, and summation conventions), $\operatorname{tr}(\Gram) := g^{\alpha \beta} \Gram_{\alpha\beta}$, and $\Gram$ is the \emph{Gram matrix}
$$ \Gram_{\alpha \beta} := \langle \partial_\alpha \phi_0, \partial_\beta \phi_0 \rangle_{\phi_0^* h}$$
with the convention that $\partial_0 \phi_0 := \phi_1$, and $\phi_0^* h \in \Gamma( \phi_0^*( \Sym^2 T^* \H) )$ is the pullback of the metric $h$ by $\phi_0$.  Note that one can also recover the Gram matrix from the stress-energy tensor by the formula
\begin{equation}\label{destress}
\langle \partial_\alpha \phi_0, \partial_\beta \phi_0 \rangle_{\phi_0^* h} = \T_{\alpha \beta} -
g_{\alpha \beta} \operatorname{tr}(\T).
\end{equation}
We also define the energy
\begin{equation}\label{energy-def}
 \E(\Phi) := \int_{\R^2} \T(\Phi)_{00}\ dx = \frac{1}{2} \int_{\R^2} |\nabla \phi_0|_{\phi_0^* h}^2 + |\phi_1|_{\phi_0^* h}^2\ dx.
 \end{equation}
The stress-energy tensor and the Gram matrix can be thought of as maps from $\S$ to $L^1( \R^2 \to \Sym(\R^{1+2})$.  It is not hard to see that these maps are continuous (since the topology on $\S$ is so strong).  Similarly, the energy functional $\E$ can be viewed as a continuous map from $\S$ to $[0,+\infty)$.

Our first main objective is a somewhat technical one, defining an energy space $\Energy$ that can be viewed as a completion of the classical data space $\S$ (once one quotients out by the rotation symmetry \eqref{rotate-data}), which respects the above symmetries, and for which the stress-energy tensor, Gram matrix, and energy can still be meaningfully defined.  More precisely, we will show

\begin{theorem}[Energy space]\label{energy-claim}  There exists a complete metric space $\Energy$ with a continuous map $\iota: \S \to \Energy$, that obeys the following properties:
\begin{itemize}
\item[(i)] $\iota(\S)$ is dense in $\Energy$.
\item[(ii)] $\iota$ is invariant under the action \eqref{rotate-data} of the rotation group $SO(m,1)$, thus $\iota(\Rot_U \Phi) = \iota(\Phi)$ for all $\Phi \in \S$. Conversely, if $\iota(\Phi) = \iota(\Psi)$, then $\Psi = \Rot_U(\Phi)$ for some $U \in SO(m,1)$.
\item[(iii)] The actions \eqref{space-trans-data}, \eqref{time-reverse-data}, \eqref{scaling-data} on $\S$ extend to a continuous isometric action on $\Energy$ (after quotienting out by rotations as in (ii)).
\item[(iv)] The Gram map $\Gram: \S \to L^1(\R^2 \to \operatorname{Sym}^2(\R^{1+2}))$ extends to a continuous map $\Gram: \Energy \to L^1(\R^2 \to \operatorname{Sym}^2(\R^{1+2}))$ (again after quotienting out by rotations as in (ii)).  In particular, the same is true for the stress-energy tensor $\T$ (by \eqref{stress-def}), and we have a continuous energy functional $\E: \Energy \to [0,+\infty)$.
\item[(v)] If $\Phi \in \Energy$ has zero energy, thus $\E(\Phi)=0$, then $\Phi$ is constant (or more precisely, $\Phi = \iota(p,0)$ for any constant $p \in \H$). 
\end{itemize}
\end{theorem}

This result is the first of five claims required in the first paper \cite{tao:heatwave} in this program to establish large data global regularity of wave maps.  At present, the space $\Energy$ in the above theorem is only described abstractly; the concrete construction of this space will be important, however, for establishing the other four claims of the paper.

\begin{remark} Suppose we replaced the hyperbolic space target $\H$ by a Euclidean space target $\R^m$.  In this case, the energy space is simply the standard space $\dot H^1(\R^2 \to \R^m) \times L^2(\R^2 \to \R^m)$, with the Hilbert space structure given by the energy functional
$$ \E(\phi_0,\phi_1) := 
\frac{1}{2} \int_{\R^2} |\nabla \phi_0|^2 + |\phi_1|^2\ dx,$$
and with $\iota$ being the identity embedding.  The analogue of rotations \eqref{rotate-data} is that of translations $(\phi_0,\phi_1) \mapsto (\phi_0+c,\phi_1)$ by constants $c \in \R^m$; note that such translations do not affect the $\dot H^1$ norm of $\phi_0$.  All the claims of Theorem \ref{energy-claim} are then easily verified from the standard theory of Sobolev and Lebesgue spaces. The reader is encouraged to view the space $\Energy$ as a nonlinear counterpart to the standard Euclidean energy space.  
\end{remark}

\begin{remark}\label{perspectives}  In the Euclidean space setting, there are at least five ways in which one can view an element $\Phi$ of the energy space (or more generally, of other low-regularity function spaces, such as Sobolev spaces):
\begin{enumerate}
\item (Cauchy perspective) $\Phi$ is a  formal limit of (an equivalence class of) Cauchy sequences of classical data with respect to a suitable norm or metric.  
\item (Lebesgue perspective) $\Phi$ is a pair of functions $(\phi_0,\phi_1)$ such that $\phi_1$ and the (weak) derivative of $\phi_0$ are defined pointwise almost everywhere and are square integrable.  
\item (Schwartz perspective) $\Phi$ is a linear functional on the space of test functions, which is continuous if the test functions are assigned a suitable dual (negative regularity) Sobolev norm.
\item (Fourier perspective) $\Phi$ is a function whose Fourier transform obeys suitable weighted square-integrability estimates.  
\item (Littlewood-Paley perspective) $\Phi$ is a function whose Littlewood-Paley resolution (defined using the heat extension, harmonic extension, wavelet transform, or Littlewood-Paley projections) obeys suitable weighted square-integrability estimates.
\end{enumerate}
These five perspectives are of course well known to be equivalent in the case of Euclidean domains and targets, thanks to the general theory of Sobolev space.  However, in the case of hyperbolic targets, it seems difficult to make the second, third, and fourth perspectives work well; for instance, in the Lebesgue perspective it is difficult to decide what it means for a sequence $(\phi_0^{(n)},\phi_1^{(n)})$ to converge to a limit $(\phi_0,\phi_1)$ because there is no canonical way to define \emph{differences} $\phi_0^{(n)} - \phi_0$ and $\phi_1^{(n)} - \phi_1$ (also, the notion of a weak derivative becomes problematic).  We were also unable to discover a usable analogue of the notion of testing a function taking values in $\H$ or $T\H$ against a test function, or of taking a Fourier transform of such functions.  Hence we shall rely entirely on the first and fifth perspectives, using the non-linear Littlewood-Paley resolution arising from the harmonic map heat flow to define distances on classical data, and then taking metric completions.
\end{remark}

The proof of Theorem \ref{energy-claim} will occupy Sections \ref{caloric-sec}-\ref{energy-sec}.  The energy space $\Energy$ will be constructed using the harmonic map heat flow
\begin{equation}\label{heatflow}
 \partial_s \phi_0 = (\phi_0^* \nabla)_i \partial_i \phi_0
\end{equation}
(where we sum Roman indices $i,j$ over $1,2$), as well as its linearisation\footnote{We will explain our notation in later sections.  We are using the variable $s$ to denote the heat-temporal variable as we wish to reserve $t$ for the wave-temporal variable.  The heat flow \eqref{heatflow}, \eqref{heatflow-linearised} is essentially the gradient flow for the energy functional $\E$; we will exploit this fact via various useful energy identities and inequalities for this heat flow.}
\begin{equation}\label{heatflow-linearised}
\partial_s \phi_1 = (\phi_0^* \nabla)_i (\phi_0^* \nabla)_i \phi_1 - (\phi_1 \wedge \partial_i \phi_0) \partial_i \phi_0
\end{equation}
to achieve a ``nonlinear Littlewood-Paley resolution'' of the position $\phi_0$ and velocity $\phi_1$ respectively.  To motivate this, let us first return to the Euclidean setting, in which $\phi_0, \phi_1$ are just smooth maps from $\R^2$ to $\R^m$, with $\phi_0$ constant outside of a compact set and $\phi_1$ vanishing outside of a compact set for simplicity.  We can extend the former function to the upper half-plane $\R^+ \times \R^2 := \{ (s,x): s \geq 0, x \in \R^2\}$ by solving the heat equation
$$ \partial_s \phi_0 = \Delta \phi_0.$$
As is well known, there exist a unique smooth bounded extension of $\phi_0$ to this space.  (We can also extend $\phi_1$ to this space, though we will not need it here.)  We then recall the standard \emph{energy identity}
\begin{equation}\label{gf}
\E( \phi_0,\phi_1) = 
\int_0^\infty \int_{\R^2} |\partial_s \phi_0|^2 \ dx ds + \frac{1}{2} \int_{\R^2} |\phi_1(0,x)|^2\ dx
\end{equation}
which can be easily verified by either the Fourier transform, functional calculus, or by an integration by parts.  This energy identity can be viewed as an integrated version of the instantaneous energy identity
$$ \partial_s \int_{\R^2} |\nabla \phi_0|^2\ dx = - 2 \int_{\R^2} |\partial_s \phi_0|^2\ dx$$
The energy identity \eqref{gf}, together with linearity, provides an isometric embedding $\iota: \dot H^1(\R^2 \to \R^m) \times L^2(\R^2 \to \R^m) \mapsto {\mathcal L}$ of the energy space (using $\E$ to define a Hilbert space structure) into the \emph{Littlewood-Paley space}
\begin{equation}\label{ldef}
 {\mathcal L} := L^2( \R^+ \times \R^2 \to \R^m, dx ds ) \times L^2( \R^2 \to \R^m, \frac{1}{2} dx )
\end{equation}
given by the formula
$$ \iota: ( \phi_0, \phi_1 ) \mapsto ( \partial_s \phi_0, \phi_1(0,\cdot) ).$$
Thus one can identify the energy space with a certain subspace of functions on the upper half-space (and the plane $\R^2$) which are square-integrable with respect to an explicit measure.

It turns out that one can do something similar with maps into hyperbolic space $\H$.  Any classical data $(\phi_0,\phi_1)$ can be extended from $\R^2$ to $\R^+ \times \R^2$ via the equations \eqref{heatflow}, \eqref{heatflow-linearised}, thanks to the work of Eells and Sampson \cite{eells}; we shall reprove these facts here for the convenience of the reader.  Note that the negative curvature of the target manifold $\H$ is essential here, as it prevents the heat flow from developing singularities, or from asymptotically approaching a non-constant harmonic map.  The analogue of \eqref{gf} is the energy identity
\begin{equation}\label{gf-2}
\E(\phi_0,\phi_1) =
\int_0^\infty \int_{\R^2} |\partial_s \phi_0|_{\phi_0^* h}^2\ ds dx + \frac{1}{2} \int_{\R^2} |\phi_1|_{\phi_0^* h}^2\ dx
\end{equation}
(where the indices $i,j$ are summed over $i,j=1,2$).  We will prove this formula in Lemma \ref{energy-ident}.

The above formula suggests that there should be an analogue of the embedding $\iota$ into the Littlewood-Paley space ${\mathcal L}$, though we no longer expect $\iota$ to be exactly an isometry.  This turns out to indeed be the case; the key point is that we can interpret $\partial_s \phi_0$ and $\phi_1$ as lying in $\R^m$ (rather than in the tangent space $T_{\phi_0} \H$) by use of a canonical orthonormal frame (or \emph{gauge}) $e: \R^+ \times \R^2 \to \operatorname{Frame}( \phi_0^*(T\H) )$ for $T \H$ (or more precisely for the pullback bundle $\phi_0^*(T \H)$), namely the \emph{caloric gauge} from \cite{tao:forges}, defined by requiring $e$ to be parallel along the heat-temporal vector field $\partial_s$ and equal to a constant frame $e(\infty): \R^m \to T_{\phi_0(\infty)} \H$ at $s=\infty$.  This gauge is unique up to rotation symmetry (which is related to the rotation ambiguity in Theorem \ref{energy-claim}(ii)) and can be used to define an analogue of the Littlewood-Paley embedding $\iota$.  We shall then construct the energy space $\Energy$ by using $\iota$ and ${\mathcal L}$ to define a metric structure on $\S$ (quotiented out by rotations), and then taking metric completions.  The various claims from Theorem \ref{energy-claim} will then follow from the parabolic regularity, energy, and stability theory of the harmonic map heat flow, which we shall develop at length in this paper (this theory will also be used in the other papers \cite{tao:heatwave3}, \cite{tao:heatwave4}, \cite{tao:heatwave5} in this program).

\subsection{No light-speed travelling waves in the energy class}

Recall that the linear wave equation
$$ \partial^\alpha \partial_\alpha \phi = 0,$$
where $\phi: \R^{1+2} \to \R^m$ is smooth,
admits light-speed travelling wave solutions of the form
$$ \phi(t,x) = \phi_0( x - t v )$$
for any unit vector $v \in \R^2$, $|v|=1$, provided that $\phi_0$ is constant along all directions orthogonal to $v$.  The initial data $(\phi_0,\phi_1) := (\phi(0), \partial_t \phi_0)$ for such waves is then \emph{degenerate} in the sense that
$$|\phi_1 + v \cdot \nabla \phi_0|^2, |w \cdot \nabla \phi_0|^2 \equiv 0$$
whenever $w \in \R^2$ is orthogonal to $v$.  

On the other hand, it is easy to see (for instance via Plancherel's theorem) that no such waves exist in the energy class $\dot H^1(\R^2 \to \R^m) \times L^2(\R^2 \to \R^m)$, other than the constant waves $\phi_0 \equiv \hbox{const}$, $\phi_1 \equiv 0$, which have zero energy.  Our next result is to establish the analogous claim for hyperbolic space targets:

\begin{theorem}[No non-trivial shift-invariant finite energy data]\label{nondeg}  Let $v \in \R^2$ be such that $|v|=1$, and let $\Phi = (\phi_0,\phi_1) \in \Energy$ be such that\footnote{These expressions for energy class solutions are of course defined using the Gram tensor or stress-energy tensor using Theorem \ref{energy-claim}(iv), and thus exist as elements of $L^1(\R^2)$; in particular, they are only defined up to almost everywhere equivalence.} $|\phi_1 + v \cdot \nabla \phi_0|_{\phi_0^* h}^2, |w \cdot \nabla \phi_0|_{\phi_0^* h}^2 \equiv 0$ whenever $w \in \R^2$ is orthogonal to $v$.  Then $\Phi$ has zero energy.
\end{theorem}

This theorem will be proven in Section \ref{nondeg-sec}; again our main tool will be the harmonic map heat flow.  It is another of the five claims required in \cite{tao:heatwave} to establish the global regularity of wave maps.

\subsection{Wave maps}

Define a \emph{classical wave map} to be a pair $(\phi,I)$, where $I$ is a time interval and $\phi: I \times \R^2 \to \H$ is a smooth map which differs from a constant $\phi(\infty) \in \H$ by a Schwartz function in space, which obeys the equation
\begin{equation}\label{cov}
 (\phi^*\nabla)^\alpha \partial_\alpha \phi = 0.
\end{equation}
Observe that for any time $t \in I$, the data $\phi[t] := (\phi(t), \partial_t \phi(t))$ lies in $\S$, and indeed $\phi$ can be viewed as a smooth curve $\phi: I \to \S$.  We refer to \cite{kman.barrett}, \cite{kman.selberg:survey}, \cite{shatah-struwe}, \cite{struwe.barrett}, \cite{tataru:survey}, \cite[Chapter 6]{tao:cbms}, \cite{rod}, \cite{krieger:survey} for surveys of the initial value problem for wave maps, which is of course the primary concern of this project.

In \cite{tao:heatwave3} we shall establish the following local well-posedness result (which, incidentally, is another of the five claims required in \cite{tao:heatwave}):

\begin{claim}[Large data local-wellposedness in the energy space]\label{lwp-claim}  For every time $t_0 \in \R$ and every initial data $\Phi_0 \in \Energy$ there exists a \emph{maximal lifespan} $I \subset \R$, and a \emph{maximal Cauchy development} $\phi: t \mapsto \phi[t]$ from $I \to \Energy$, which obeys the following properties:
\begin{itemize}
\item[(i)] (Local existence) $I$ is an open interval containing $t_0$.
\item[(ii)] (Strong solution) $\phi: I \to \Energy$ is continuous.
\item[(iii)] (Persistence of regularity) If $\Phi_0 = \iota( \tilde \Phi_0 )$ for some classical data $\tilde \Phi_0$, then there exists a classical wave map $(\tilde \phi,I)$ with initial data $\tilde \phi[t_0] = \tilde \Phi_0$ such that $\phi[t] = \iota(\tilde \phi[t])$ for all $t \in I$.
\item[(iv)] (Continuous dependence)  If $\Phi_{0,n}$ is a sequence of data in $\Energy$ converging to a limit $\Phi_{0,\infty}$, and $\phi_n: I_n \to \Energy$ and $\phi_\infty: I_\infty \to \Energy$ are the associated maximal Cauchy developments on the associated maximal lifespans, then for every compact subinterval $K$ of $I_\infty$, we have $K \subset I_n$ for all sufficiently large $n$, and $\phi_n$ converges uniformly to $\phi$ on $K$ in the $\Energy$ topology.
\item[(v)] (Maximality)  If $t_* \in \R$ is a finite endpoint of $I$, then $\phi(t)$ has no convergent subsequence in $\Energy$ as $t \to t_*$.
\end{itemize}
\end{claim}

It should not be surprising that this result will be proven using the theory of the harmonic map heat flow, since this flow is used to construct the energy space.  It will also rely heavily on (slight extensions of) the delicate spacetime function spaces and estimates from \cite{tataru:wave2}, \cite{tao:wavemap2}.  We refer to maximal Cauchy developments, and any restriction of such developments to a smaller time interval, as \emph{energy class solutions}.

We will not prove Claim \ref{lwp-claim} here.  However, we shall use this claim to rule out two special types of energy class solutions which would otherwise cause great difficulty for the global regularity problem, namely travelling and self-similar solutions:

\begin{definition}[Travelling and self-similar solutions]\label{travel}  An energy class solution $\phi: I \to \Energy$ is said to be \emph{travelling} with velocity $v \in \R^2$ if 
\begin{equation}\label{static-def}
|\partial_t \phi + v \cdot \partial_x \phi|_{\phi^* h}^2 \equiv 0
\end{equation}
throughout $I \times \R^2$, where the quantity in \eqref{static-def} is of course defined via the Gram tensor (or stress-energy tensor).  Similarly, an energy class solution $\phi: I \to \Energy$ is said to be \emph{self-similar} if $\T = 0$ outside of the light cone $\{ (t,x): |x| \leq |t| \}$, and if
\begin{equation}\label{selfsimilar-def}
|t \partial_t \phi + x \cdot \partial_x \phi|_{\phi^* h}^2 \equiv 0
\end{equation}
throughout $I \times \R^2$.
\end{definition}

Our main result here is as follows.

\begin{theorem}[No non-trivial self-similar or travelling energy class solutions]\label{selfsim} Assume Claim \ref{lwp-claim} holds.  Then:  
\begin{itemize}
\item[(i)] The only energy class solutions $\phi: \R \to \Energy$ which are travelling with some velocity $v$ with $|v| < 1$ are the constant (i.e. zero-energy) solutions.
\item[(ii)] The only energy class solutions $\phi: (-\infty,0) \to \Energy$ which are self-similar are the constant solutions.
\end{itemize}
\end{theorem}

This result is yet another one of the five claims required\footnote{In \cite{tao:heatwave} an additional hypothesis was assumed that the solutions were almost periodic, but this hypothesis turns out to not be needed in our arguments.} in \cite{tao:heatwave}.  Thus the results in this paper (which constitute the ``elliptic'' and ``parabolic'' portions of the project) reduce the task of establishing large data global regularity for wave maps to just two claims, the local well-posedness claim (Claim \ref{lwp-claim}) and a further claim regarding the existence of non-trivial almost periodic maximal Cauchy developments in the event that global regularity breaks down (see \cite[Claim 1.16]{tao:heatwave}).  These last two ``hyperbolic'' claims will be the objective of the papers \cite{tao:heatwave3} and \cite{tao:heatwave4}, \cite{tao:heatwave5} respectively.

As remarked in \cite{tao:heatwave}, the claims in Theorem \ref{selfsim} are well known in the context of classical wave maps; the main difficulty is to show that the proofs in that case are in some sense stable with respect to perturbations in the energy class.  We prove part (i) of this theorem in Section \ref{travel-sec} and part (ii) in Section \ref{self-sec}.

\subsection{Organisation of the paper}

In Section \ref{notation-sec} we set out our basic notation on function spaces, the heat equation, and asymptotic notation, and record some standard parabolic regularity estimates for the heat equation as well as the Gagliardo-Nirenberg inequality; these inequalities will be used repeatedly throughout the paper.  In Section 3 we study the harmonic map heat flow and the caloric gauge from a \emph{qualitative} viewpoint - focusing on the existence theory and qualitative asymptotics, and relying heavily on the hypothesis that the initial data is classical.  Here we will also develop our basic geometric formalism for understanding maps into $\H$, in particular the use of differentiated fields $\psi_\alpha$ and connection fields $A_\alpha$ with respect to an orthonormal frame $e$.  The main objective of the section is to establish the existence and qualitative asymptotics of a caloric gauge for an arbitrary classical field, including fields that vary with respect to a time parameter $t$.  In Section 4 we then develop the \emph{quantitative} theory of these flows and gauges, in which the estimates are only allowed to depend on the energy of the data, rather than on smoother norms.  The estimates here can be viewed as nonlinear counterparts to the parabolic regularity estimates for the linear heat equation developed in Section \ref{linear-heat}.

In Section \ref{energy-sec} we use the above estimates to construct the energy space and establish Theorem \ref{energy-claim}.  The one delicate task here is to show that the Gram matrix operator $\Gram$ extends continuously from $\C$ to $\Energy$, which requres one to understand how the Gram matrix can be reconstructed from the Littlewood-Paley resolution given by the harmonic map heat flow.  Then, in Section \ref{nondeg-sec}, we establish Theorem \ref{nondeg}; the basic strategy here is to localise to a fixed ``frequency'' (or more precisely, to a fixed range of the heat-temporal parameter $s$) and exploit a one-dimensional Poincar\'e inequality in the direction orthogonal to $v$.

In the second half of the paper, we apply the above theory to wave maps, with the objective of establishing Theorem \ref{selfsim}.  We begin by establishing some basic estimates in Section \ref{heatwave-sec} on the heat flow when applied to a wave map with bounded energy, in particular obtaining some crucial boundedness and uniform continuity estimates on \emph{second time derivatives} of this flow, which we obtain by carefully measuring the extent to which the wave map equation and heat flow equation fail to commute.

Morally, travelling and self-similar wave maps should arise from harmonic maps (after applying a Lorentz transformation or a conformal transformation).  It is therefore necessary to rule out the existence of non-trivial harmonic maps into hyperbolic space $\H^m$.  This is easy for classical harmonic maps, but for applications to energy class wave maps we will need a robust version of this observation, in which the tension field $(\phi^* \nabla)_i \partial_i \phi$ is only assumed to be small in a rough norm, rather than vanshing completely.  Fortunately, the harmonic map heat flow machinery developed earlier can establish the results we need (after first applying the necessary change of variables to arrive at the point where one has an approximate harmonic map).  We establish these results in Section \ref{weakmap}.

In Section \ref{travel-sec} we use the above machinery to rule out non-trivial travelling wave maps in the energy class (the first part of Theorem \ref{selfsim}).  The basic idea is to use the heat flow (and the estimates in Sectino \ref{heatwave-sec}) to regularise the energy class wave map (or more precisely, a classical approximant to such maps) in order to gain enough regularity that one can justify the formal observation that travelling wave maps arise from harmonic maps, which can then be ruled out by the theory in Section \ref{weakmap}.

In Section \ref{self-sec} we use a similar strategy to rule out non-trivial self-similar wave maps in the energy class, thus establishing the second part of Theorem \ref{selfsim}.  Here there are some additional technical issues caused by some mild singularities at the light cone, requiring an additional stress-energy analysis (related to the holomorphicity of the Hopf differential for two-dimensional harmonic maps) to establish enough regularity near the boundary of the cone to ignore the singularity.

\subsection{Acknowledgements}

This project was started in 2001, while the author was a Clay Prize Fellow.  The author thanks Andrew Hassell and the Australian National University for their hospitality when a substantial portion of this work was initially conducted, and to Ben Andrews and Andrew Hassell for a crash course in Riemannian geometry and manifold embedding, and in particular to Ben Andrews for explaining the harmonic map heat flow.  The author also thanks Mark Keel for background material on wave maps, Daniel Tataru for sharing some valuable insights on multilinear estimates and function spaces, and to Igor Rodnianski and Jacob Sterbenz for valuable discussions.  The author is supported by NSF grant DMS-0649473 and a grant from the Macarthur Foundation.

\section{Notation and basic estimates}\label{notation-sec}

\subsection{Asymptotic notation}

We use $X = O(Y)$ or $X \lesssim Y$ to denote the estimate $|X| \leq CY$ for some absolute constant $C > 0$.  If we wish to permit $C$ to depend on some parameters, we shall denote this by subscripts, e.g. $X = O_k(Y)$ or $X \lesssim_k Y$ denotes the estimate $|X| \leq C_k Y$ where $C_k > 0$ depends on $k$.  On the other hand, we always allow the implied constants to depend on the dimension $m$ of the target hyperbolic space $\H^m$, which is fixed throughout the paper.

Now suppose we have an additional parameter $n$, with $X$ and $Y$ depending on $n$.  We write $X = o_{n \to \infty}(Y)$ to denote the statement that $|X| \leq c(n) Y$ for some $c(n)$ depending only on $n$ such that $c(n) \to 0$ as $n \to \infty$.  Similarly for $n$ replaced by other parameters (or $\infty$ replaced by a different limit).  Again, if $c(n)$ needs to depend on another parameter, such as $k$ (with $c(n) \to 0$ as $n \to \infty$ for each \emph{fixed} $k$), we denote this by subscripts, thus $X = o_{n \to \infty;k}(Y)$.  Conversely, if $c(n)$ does not depend on a parameter $k$, we say that the statement $X = o_{n \to \infty}(Y)$ holds \emph{uniformly} in $k$.

Note that parameters can be other mathematical objects than numbers.  For instance, the statement that a function $u: \R^2 \to \R$ is Schwartz is equivalent to the assertion that one has a bound of the form $|\partial_x^k u(x)| \lesssim_{j,k,u} \langle x \rangle^{-j}$ for all $j,k \geq 0$ and $x \in \R^2$, where $\langle x \rangle := (1+|x|^2)^{1/2}$.

\subsection{Function spaces}

We use the usual $L^p_x(\R^2)$ spaces, as well as the norm
$$ \| u \|_{C^k_x(\R^2)} := \sup_{0 \leq j \leq k} \sup_{x \in \R^2} |\partial_x^j u(x)|$$
and seminorm
$$ \| u \|_{\dot C^k_x(\R^2)} := \sup_{x \in \R^2} |\partial_x^k u(x)|$$
for the Banach space $C^k_x(\R^2)$ of $k$-times continuously differentiable functions, where $k=0,1,2,\ldots$ and $\partial_x = (\frac{\partial}{\partial x_1}, \frac{\partial}{\partial x_2})$ is the gradient operator.  (We will reserve the symbol $\nabla$ for the Levi-Civita connection on $H$.)

When analysing stationary or self-similar wave maps, it will be convenient to also use the norm
\begin{equation}\label{morrey}
\| f \|_{L^1_\loc(\R^2)} := \sup_{x_0 \in \R^2} \int_{|x-x_0| \leq 1} |f(x)|\ dx,
\end{equation}
as this norm is weak enough to be controlled by both $L^1_x(\R^2)$ and $L^2_x(\R^2)$.

We will rely frequently on various special cases of the \emph{Gagliardo-Nirenberg inequality}, such as
\begin{align}
\| \partial^k_x u \|_{L^p_x(\R^2)} &\lesssim_{p,k} \|u\|_{L^p_x(\R^2)}^{1/2} \|\partial_x^{2k} u\|_{L^p_x(\R^2)}^{1/2} \label{gag-1} \\
\| u \|_{L^\infty_x(\R^2)} &\lesssim \|u\|_{L^2_x(\R^2)}^{1/2} \|\partial_x^2 u\|_{L^2_x(\R^2)}^{1/2} \label{gag-2} \\
\| u \|_{L^2_x(\R^2)} &\lesssim \|u\|_{L^1_x(\R^2)}^{1/2} \|\partial_x^2 u\|_{L^1_x(\R^2)}^{1/2} \label{gag-6} \\
\| u \|_{L^\infty_x(\R^2)} &\lesssim
\| u \|_{L^2_x(\R^2)}^{1/3} \| \partial_x u \|_{L^4_x(\R^2)}^{2/3}\label{gag-3} \\
\| u \|_{L^4_x(\R^2)} &\lesssim
\| u \|_{L^2_x(\R^2)}^{1/2} \| \partial_x u \|_{L^2_x(\R^2)}^{1/2}\label{gag-4}
\end{align}
valid for all scalar or vector-valued Schwartz functions $u$ on $\R^2$ (we allow the constants here to depend on the dimension of the range of $u$) and all $k=0,1,2,\ldots$ and $1 \leq p \leq \infty$.  Such inequalities are standard in the literature, see e.g. \cite[Appendix A]{tao:cbms}.  

\begin{remark} We rely primarily on Gagliardo-Nirenberg inequalities rather than Sobolev inequalities in this paper due to the (well-known) failure of the endpoint Sobolev embeddings $\dot H^1_x(\R^2) \not \subset L^\infty_x(\R^2)$ and $\dot W^{1,1}_x(\R^2) \not \subset L^2_x(\R^2)$.
\end{remark}

\subsection{The linear heat equation}\label{linear-heat}

Throughout the paper we use $\Delta := \frac{\partial^2}{\partial x_1^2} + \frac{\partial^2}{\partial x_2^2}$ to denote the (spatial) Laplacian on $\R^2$. We use $e^{s\Delta}$ for $s > 0$ to denote the free heat propagator
\begin{equation}\label{heat-prop}
e^{s\Delta} u(x) := \frac{1}{4\pi s} \int_{\R^2} e^{-|x-y|^2/4s} u(y)\ dy.
\end{equation}
From Young's inequality we easily establish the parabolic regularity estimate
\begin{equation}\label{heat-lp}
\| \partial_x^k e^{s\Delta} u \|_{L^q_x(\R^2)} \lesssim_{p,q,k} s^{\frac{1}{q}-\frac{1}{p}-\frac{k}{2}} \| e^{s\Delta} u \|_{L^p_x(\R^2)} 
\end{equation}
valid for all $s > 0$, $k \geq 0$, and $1 \leq p \leq q \leq \infty$.  In particular we have
\begin{equation}\label{strict-reg}
\| e^{s\Delta} u \|_{\dot C^1_x(\R^2)} \lesssim s^{-1/2} \| u \|_{C^0_x(\R^2)} 
\end{equation}
and
\begin{equation}\label{nonstrict-reg}
\| e^{s\Delta} u \|_{C^1_x(\R^2)} \lesssim (1 + s^{-1/2}) \| u \|_{C^0_x(\R^2)}. 
\end{equation}
Since $L^1_\loc$ is translation invariant, we also see from \eqref{heat-prop} and Minkowski's inequality that
\begin{equation}\label{sdel}
\| e^{s\Delta} f\|_{L^1_\loc(\R^2)} \lesssim \|f\|_{L^1_\loc(\R^2)}
\end{equation}
for all $s > 0$.

We have the following variant of \eqref{heat-lp}:

\begin{lemma}[Parabolic Strichartz estimate]\label{ubang} For any $u \in L^2_x(\R^2)$ and $2 < p \leq \infty$ we have
$$ \int_0^\infty s^{-2/p} \| e^{s\Delta} u \|_{L^p_x(\R^2)}^2\ ds \lesssim_p \|u\|_{L^2_x(\R^2)}.$$
\end{lemma}

\begin{remark} Note that a direct application of \eqref{heat-lp} would almost establish this claim except for a logarithmic divergence in the $s$ integral.  This ability to remove the logarithmic divergence is crucial for technical reasons at various points in this paper.
\end{remark}

\begin{proof} We use the $TT^*$ method.  By duality it suffices to show that
$$ \| \int_0^\infty s^{-1/p} e^{s\Delta} F(s)\ ds \|_{L^2_x(\R^2)}^2 \lesssim_p \int_0^\infty \|F(s)\|_{L^{p'}_x(\R^2)}^2\ ds$$
for all test functions $F$, where $p' := p/(p-1)$ is the dual exponent.  The left-hand side can be expanded as
$$ \int_0^\infty \int_0^\infty s^{-1/p} s^{-1/p'} \langle e^{(s+s')\Delta} F(s), F(s') \rangle_x\ ds ds'.$$
Applying \eqref{heat-lp} and writing $f(s) := \|F(s)\|_{L^{p'}_x(\R^2)}$, it suffices to show that
$$ \int_0^\infty \int_0^\infty \frac{ds ds'}{(s+s')^{1-2/p} s^{1/p} (s')^{1/p}} f(s) f(s')\ ds ds'
\lesssim_p \int_0^\infty f(s)^2\ ds.$$
By symmetry we can reduce to the region where $s' \leq s$.  If one decomposes into the dyadic ranges $2^{-n} s \leq s' \leq 2^{-n+1} s$, we can bound the left-hand side by
$$ \lesssim \sum_{n=1}^\infty 2^{n/p} 
\int_0^\infty \int_{2^{-n} s \leq s' \leq 2^{-n+1} s} \frac{ f(s') f(s) }{s}\ ds' ds.$$
By Schur's test, the summand is $O( 2^{n/p} 2^{-n/2} \int_0^\infty f(s)^2\ ds )$, and the claim follows.
\end{proof}

We recall \emph{Duhamel's formula}
\begin{equation}\label{duh}
 u(s_1) = e^{(s_1-s_0)\Delta} u(s_0) + \int_{s_0}^{s_1} e^{(s_1-s)\Delta} (\partial_s u - \Delta u)(s)\ ds
\end{equation}
for any continuous map $s \mapsto u(s)$ from the interval $[s_0,s_1]$ to the space of tempered distributions on $\R^2$, which can be either scalar or vector valued.  From this and \eqref{heat-prop} we immediately obtain

\begin{corollary}[Comparison principle]\label{compar}  For each $s \in [s_0,s_1]$, let $u(s)$ be a non-negative tempered distribution on $\R^2$ varying continuously in $s$ (in the tempered distributional topology) and such that $\partial_s u \leq \Delta u$ in the sense of distributions.  Then $u(s_1) \leq e^{(s_1-s_0)\Delta} u(s_0)$.
\end{corollary}

\begin{remark}
This corollary will be particularly useful when combined with Lemma \ref{ubang}, \eqref{heat-prop}, or the $k=0$ case of \eqref{heat-lp}.
\end{remark}

\section{The harmonic map heat flow and the caloric gauge I.  Qualitative theory}\label{caloric-sec}

In this section, we study the qualitative properties of the harmonic map heat flow from classical data, and recall how this flow is used to define the \emph{caloric} gauge from \cite{tao:forges}.  The theory here is essentially already contained in the work of Eells and Sampson \cite{eells}, but for the convenience of the reader we give the full details here.

Throughout this section, all functions and vector fields are assumed to be smooth.  Our estimates here will be \emph{qualitative} in the sense that they will depend on smooth norms of the given data; in later sections we shall be much more interested in \emph{quantitative} estimates, which only depend on low-regularity quantities such as the total energy of the data.

\subsection{The geometry of hyperbolic space}

For the purposes of this qualitative analysis it is convenient to represent the hyperbolic space $\H = (\H^m,h)$ concretely as the upper unit hyperboloid
\begin{equation}\label{hyperdef}
 \H := \{ (t,x) \in \R^{1+m}: t = +\sqrt{1 + |x|^2} \} \subset \R^{1+m}
\end{equation}
(with the metric $dh^2$ induced from the Minkowski metric $dg^2 = -dt^2 + dx^2$ on $\R^{1+m}$).  In this case, the Levi-Civita connection can be written explicitly in coordinates as
\begin{equation}\label{nxy}
 \nabla_X Y(p) = \partial_X Y(p) - \langle X(p), Y(p) \rangle_h p
\end{equation}
for all vector fields $X, Y$ on $\H$ (thus $X(p) \in \R^{1+m}$ and $\langle X(p),p\rangle_{\R^{1+m}} = 0$ for all $p \in \H$) and all $p \in \H$.  We observe the \emph{zero-torsion property}
\begin{equation}\label{zerotor}
 \nabla_X \partial_Y f - \nabla_Y \partial_X f - \partial_{[X,Y]} f = 0
\end{equation}
for all $f: \H \to \R$ and vector fields $X,Y$, as well as the \emph{constant negative curvature property}
\begin{equation}\label{curv}
 \nabla_X \nabla Y Z - \nabla_Y \nabla_X Z - \nabla_{[X,Y]} Z = - (X \wedge Y) Z
\end{equation}
for all vector fields $X,Y,Z$, where $X \wedge Y \in \Gamma( \Hom( T\H \to T\H ) )$ is the anti-symmetric rank (1,1) tensor defined by the formula
$$ (X \wedge Y) Z := X \langle Y, Z \rangle_h - Y \langle X, Z \rangle_h.$$
We also observe the \emph{Leibniz rule}
\begin{equation}\label{leibnitz}
\partial_X h( Y, Z ) = h( \nabla_X Y, Z ) + h( Y, \nabla_X Z )
\end{equation}
for all vector fields $X,Y,Z$.

Now let $\phi: \R^d \to \H$ be a smooth map from a vector space $\R^d$ into hyperbolic space.  Then the tangent bundle $T\H$ over $\H$ pulls back under $\phi$ to a vector bundle $\phi^*(T\H)$; a section $\psi$ of this bundle is thus an assignment of a tangent vector $\psi(x) \in T_{\phi(x)} \H$ to every $x \in \R^d$; using the representation \eqref{hyperdef}, one can view $\psi$ as a map $\psi: \R^d \to \R^{1+m}$ such that $\langle \psi, \phi\rangle_{\R^{1+m}}=0$.  The Levi-Civita connection $\nabla$ on $T\H$ then pulls back to a connection $\phi^* \nabla$ on $\phi^*(T\H)$; using the standard coordinate vector fields $\partial_1,\ldots,\partial_d$ on $\R^d$, and using the representation \eqref{hyperdef}, the connection $\phi^* \nabla$ can be expressed explicitly in coordinates (using \eqref{nxy}) as
\begin{equation}\label{phicor}
(\phi^*\nabla)_i \psi(x) = \partial_i \psi(x) - \langle \psi(x), \partial_i \phi(x) \rangle_{\phi^* h} \phi(x).
\end{equation}
From \eqref{zerotor} (or \eqref{phicor}) we have the \emph{zero torsion property}
\begin{equation}\label{zerotor-phi}
 (\phi^*\nabla)_i \partial_j \phi = (\phi^*\nabla)_j \partial_i \phi
\end{equation}
while from \eqref{curv} (or \eqref{phicor}) we have the \emph{constant negative curvature property}
\begin{equation}\label{curv-phi}
(\phi^*\nabla)_i (\phi^*\nabla)_j \psi - (\phi^*\nabla)_j (\phi^*\nabla)_i \psi 
= - (\partial_i \phi \wedge \partial_j \phi) \psi
\end{equation}
for any section $\psi$ of $\phi^* T\H$, where $\partial_i \phi \wedge \partial_j \phi \in \Gamma( \Hom( \phi^* T\H \to \phi^* T\H ) )$ is the anti-symmetric rank $(1,1)$ tensor defined by the formula
$$(\partial_i \phi \wedge \partial_j \phi) \psi = \partial_i \phi \langle \partial_j \phi, \psi \rangle_{\phi^* h} - (\partial_j \phi \wedge \partial_i \phi) \psi.$$
Finally, from \eqref{leibnitz} (or \eqref{phicor}) we have the \emph{Leibniz rule}
\begin{equation}\label{leibnitz-phi}
\partial_i \langle \psi, \psi' \rangle_{\phi^* h} = \langle (\phi^* \nabla)_i \psi, \psi' \rangle_{\phi^* h} + \langle (\phi^* \nabla)_i \psi, \psi' \rangle_{\phi^* h}
\end{equation}
for any $\psi, \psi'$.  

\begin{remark} Of course, all the above discussion continues to hold if the vector space domain $\R^d$ is replaced by an open subset of a vector space, such as a spacetime slab $(t_0,t_1) \times \R^2 \subset \R^{1+2}$.
\end{remark}

\begin{remark} The formalism here is completely covariant with respect to the target $\H$.  Later on in this section (and for most of the paper) we shall instead work with respect to a gauge (an orthonormal frame), in which the derivative fields $\partial_j \phi$ are now replaced with fields $\psi_j$ taking values in $\R^m$, and the connection $(\phi^* \nabla)_j$ is represented by a matrix field $A_j$ taking values in $\mathfrak{so}(m)$.
\end{remark}

\subsection{Harmonic map heat flow}

A function $\phi: \R^+ \times \R^2 \to \H$ is said to be a \emph{harmonic map heat flow}, or \emph{heat flow} for short, if it obeys the equation
\begin{equation}\label{psphi}
\partial_s \phi = (\phi^* \nabla)_i \partial_i \phi
\end{equation}
where we parameterise $\R^+ \times \R^2$ by $(s,x_1,x_2)$, and $i$ ranges over $1,2$ with the usual summation conventions.  Using the representation \eqref{hyperdef} and \eqref{phicor}, we can express the harmonic map heat flow equation in coordinates as
\begin{equation}\label{pars}
 \partial_s \phi = \Delta \phi - |\partial_x \phi|_{\phi^* h}^2 \phi 
\end{equation}
where of course $\partial_x \phi := (\partial_1 \phi, \partial_2 \phi)$, and thus
$$ |\partial_x \phi|_{\phi^* h}^2 := \langle \partial_i \phi, \partial_i \phi \rangle_{\phi^* h}.$$

For each $k \geq 1$, define the \emph{energy densities} $\e_k$ of a heat flow by the formula
\begin{equation}\label{energy-dens-def}
\begin{split}
 \e_k &:= |(\phi^* \nabla)_x^{k-1} \partial_x \phi|_{\phi^* h}^2 \\
 &:= \left\langle (\phi^* \nabla)_{i_1} \ldots  (\phi^* \nabla)_{i_{k-1}} \partial_{i_k} \phi, (\phi^* \nabla)_{i_1} \ldots  (\phi^* \nabla)_{i_{k-1}} \partial_{i_k} \phi \right\rangle_{\phi^* h}
 \end{split}
\end{equation}
for $k \geq 1$, where $i_1,\ldots,i_k$ are summed over $1,2$ as usual.  The following estimates are crucial to us:

\begin{lemma}[Bochner-Weitzenb\"ock type identities]
Let $\phi$ be a heat flow.  Then we have
\begin{equation}\label{bochner}
\partial_s \e_1 = \Delta \e_1 - 2 \e_2 - |\partial_x \phi \wedge \partial_x \phi|_{\phi^* h}^2
\end{equation}
where the expression $|\partial_x \phi \wedge \partial_x \phi|_{\phi^* h}^2$ is the Hilbert-Schmidt norm of the operator $\nabla_i \phi \wedge \nabla_j \phi$ using the inner product $\phi^* h$, summed over $i,j$.  More generally, we have
\begin{equation}\label{psem}
\partial_s \e_k = \Delta \e_k - 2 \e_{k+1} + \sum_{a,b,c \geq 1: a+b+c=k+2} O_{k}( \e_a^{1/2} \e_b^{1/2} \e_c^{1/2} \e_k^{1/2} )
\end{equation}
for all $k \geq 1$.  
\end{lemma}

\begin{proof}
We begin with \eqref{bochner}.  From \eqref{zerotor-phi}, \eqref{leibnitz-phi} we have
$$ \partial_s \e_1 = 2 \langle (\phi^* \nabla)_i \partial_s \phi, \partial_i \phi \rangle_{\phi^* h}.$$
On the other hand, from \eqref{psphi}, \eqref{curv-phi}, \eqref{zerotor-phi} we have
$$(\phi^* \nabla)_i \partial_s \phi = (\phi^* \nabla)_j (\phi^* \nabla)_j \partial_i \phi - (\partial_i \phi \wedge \partial_j \phi) \partial_j \phi.$$
Also, from \eqref{leibnitz-phi} we have
$$ \Delta \e_1 = 2\e_2 + 2 \left\langle (\phi^* \nabla)_j (\phi^* \nabla)_j \partial_i \phi , \partial_i \phi \right\rangle_{\phi^* h}.$$
Putting these estimates together yield \eqref{bochner}.

Now we turn to \eqref{psem}.  From \eqref{leibnitz-phi} we have
$$ \partial_s \e_k = 2 \left\langle (\phi^* \nabla)_s (\phi^* \nabla)_x^{k-1} \partial_x \phi, (\phi^* \nabla)_x^{k-1} \partial_x \phi \right\rangle_{h^* \phi}.$$
From many applications of \eqref{curv-phi}, \eqref{zerotor-phi}, \eqref{psphi}, and the triangle inequality we have
$$ \left| (\phi^* \nabla)_s (\phi^* \nabla)_x^{k-1} \partial_x \phi -
(\phi^* \nabla)_x^k \partial_s \phi \right|_{h^* \phi} \lesssim_k \sum_{a,b,c \geq 1: a+b+c=k+2} \e_a^{1/2} \e_b^{1/2} \e_c^{1/2}.$$
Using \eqref{psphi} we can of course write $(\phi^* \nabla)_x^{m} \partial_s \phi  = (\phi^* \nabla)_x^{m} (\phi^* \nabla)_j \partial_j \phi$.  By many applications of \eqref{curv-phi}, \eqref{zerotor-phi} we have
\begin{align*}
\biggl| (\phi^* \nabla)_x^k (\phi^* \nabla)_j \partial_j \phi
 &- (\phi^* \nabla)_j (\phi^* \nabla)_j (\phi^* \nabla)_x^{k-1} \partial_x \phi\biggr|_{h^* \phi} \\
 & \quad \lesssim_k \sum_{a,b,c \geq 1: a+b+c=k+2} \e_a^{1/2} \e_b^{1/2} \e_c^{1/2}.
\end{align*}
Finally, we observe from \eqref{leibnitz-phi} that
$$ \Delta \e_m = 2\e_{m+1} + 2 \left\langle (\phi^* \nabla)_j (\phi^* \nabla)_j (\phi^* \nabla)_x^{m-1} \partial_x \phi, 
(\phi^* \nabla)_x^{m-1} \partial_x \phi \right\rangle_{h^* \phi}.$$
Putting all these estimates together (and using Cauchy-Schwarz) one obtains \eqref{psem} as desired.
\end{proof}

\begin{corollary}\label{boch}  Let $\phi$ be a heat flow.  Then we have the inequalities
\begin{equation}\label{bochner-3}
\partial_s \sqrt{\e_1} \leq \Delta \sqrt{\e_1}
\end{equation}
and more generally
\begin{equation}\label{psem-sqrt}
\partial_s \sqrt{\e_k} \leq \Delta \sqrt{\e_k} + \sum_{a,b,c \geq 1: a+b+c=k+2} O_{k}( \e_a^{1/2} \e_b^{1/2} \e_c^{1/2} )
\end{equation}
for all $k \geq 1$, where the expressions here are interpreted in a distributional sense.
\end{corollary}

\begin{proof}  We shall work formally; one can justify the arguments here rigorously by replacing $\sqrt{\e_k}$ by $(\eps^2 + \e_k)^{1/2}$ and taking distributional limits as $\eps \to 0$.  We leave the details to the interested reader.

Writing $\partial_s \e_k = 2 \sqrt{\e_k} \partial_s \sqrt{\e_k}$ and $\Delta \e_k = 2 \sqrt{\e_k} \Delta \sqrt{\e_k} + 2 |\partial_x \sqrt{\e_k}|^2$, we see from \eqref{bochner} (discarding the final negative term on the right-hand side) and \eqref{psem-sqrt} that it suffices to show the \emph{diamagnetic inequality}
\begin{equation}\label{diamagnetic}
|\partial_x \sqrt{\e_k}| \leq \sqrt{\e_{k+1}}.
\end{equation}
But from \eqref{leibnitz-phi} and Cauchy-Schwarz we have
$$ |\partial_x \e_k| \leq 2 \sqrt{\e_k} \sqrt{\e_{k+1}}$$
and the claim follows.
\end{proof}

As a first  application of these estimates, we have the following global existence and asymptotics of harmonic map heat flow from classical initial data, essentially due to Eells and Sampson \cite{eells}:

\begin{proposition}[Global existence and qualitative decay of heat flows]\label{Decay}\cite{eells}  Let $\phi(0): \R^2 \to \H$ be smooth and differing from $\phi(\infty) \in \H$ by a Schwartz function.  Then there exists a unique smooth heat flow extension $\phi: \R^+ \times \R^2 \to \H$ with all derivatives uniformly bounded.  Furthermore, if we identify $\H$ with a subset of $\R^{1+m}$ using \eqref{hyperdef}, thus $\phi = (\phi^0,\ldots,\phi^m)$, then the components $\phi(s)^i - \phi(\infty)^i$ and all of its derivatives are rapidly decreasing in space for each $s$, and  we have the estimates
\begin{equation}\label{manytime}
 |\partial_s^j \partial_x^k(\phi^i - \phi(\infty)^i)(s,x)| \lesssim_{j,k,\phi} \langle s \rangle^{-(1+k+2j)/2}
 \end{equation}
for all $(s,x) \in \R^+ \times \R^2$, $0 \leq i \leq m$ and $j,k \geq 0$, where $\langle s \rangle := (1+|s|^2)^{1/2}$.  In particular, as $s \to \infty$, $\phi(s)$ converges in the $C^\infty(\R^2 \to \R^{1+m})$ topology to $\phi(\infty)$.
\end{proposition}

\begin{proof}  We first dispose of the uniqueness claim.  If $\phi, \tilde \phi$ are two heat flows with the same initial data, and with $\phi, \tilde \phi, \partial_x \phi, \partial_x \tilde \phi$ uniformly bounded, then from \eqref{pars} we see that the difference $u := \tilde \phi - \phi$ is smooth, bounded, has bounded first derivative, and obeys a heat equation of the form $\partial_s u = \Delta u + O_{\phi,\tilde \phi}(|u|) + O_{\phi,\tilde \phi}(|\partial_x u|)$.  Since $u$ vanishes at time zero, and is bounded with bounded first derivative, a routine application of the maximum principle shows that $u$ is identically zero, yielding uniqueness.

Next we establish local existence.  We can rewrite \eqref{pars} using the Duhamel formula \eqref{duh} as
\begin{equation}\label{duh-local}
\phi(s) = e^{s\Delta} \phi(0) - \int_0^s e^{(s-s')\Delta} (|\partial_x \phi|_{\phi^* h}^2 \phi)(s')\ ds'.
\end{equation}
As $\phi$ is smooth and differs from constant by a Schwartz function, the linear solution $e^{s\Delta} \phi$ has all derivatives uniformly bounded.  Using \eqref{nonstrict-reg}, one can iterate \eqref{duh-local} in the space $C^0_s C^1_x( [0,T] \times \R^2)$ for some sufficiently small $T > 0$ to establish a local solution in this space by the Picard iteration method; differentiating \eqref{duh-local} repeatedly and using further parabolic regularity estimates we can readily establish that $\phi$ is in fact smooth with all derivatives bounded in such an interval.  One can iterate the local theory in the usual manner and conclude that the heat flow can be continued in time and is smooth so long as the first derivatives $\partial_x \phi$ of $\phi$ remain bounded.  By working in weighted spaces such as $\| \langle x \rangle^k u \|_{C^0_s C^m_x([0,T] \times \R^2)}$, one can also ensure that $\phi(s)-\phi(\infty)$ is rapidly decreasing in space for each fixed $s$ for which the solution exists; we omit the standard details.

To estimate the first derivatives of $\phi$, we use \eqref{bochner-3}.  Applying Corollary \ref{compar} and \eqref{heat-prop}, we conclude that
$$ \sqrt{\e_1}(s,x) \leq \frac{1}{4\pi s} \int_{\R^2} e^{-|x-y|^2/4s} \sqrt{\e_1}(0,y)\ dy$$
for any $s > 0$ and $x \in \R^2$.  Since $\sqrt{\e_1}$ is bounded and compactly supported at time zero, we conclude in particular the pointwise estimate
\begin{equation}\label{joke}
 |\partial_x \phi(s,x)|_{\phi^* h} \lesssim_\phi \langle s \rangle^{-1} e^{-|x|^2/8s}
 \end{equation}
(say) on first derivatives.  This \emph{a priori} bound on first derivatives, combined with the above local existence theory, ensures a global smooth solution for the heat flow.  Integrating this (and using the fact that $\phi(s,x) - \phi(\infty)$ tends to zero as $x \to \infty$) we also obtain the pointwise bound
\begin{equation}\label{zero}
\phi(s,x) = \phi(\infty) + O_\phi(\langle s\rangle^{-1/2}).
\end{equation}
for all $(s,x) \in \R^+ \times \R^2$, where we are interpreting $\phi$ as taking vlaues in $\R^{1+m}$.  (The case when $s = O(1)$ can be handled directly from the local theory.)  This gives the $j=k=0$ case of \eqref{manytime}.

Having bounded the zeroth and first derivatives of $\phi$, we now turn to the higher derivatives.  Specifically, we claim a pointwise bound of the form
\begin{equation}\label{sqrtem}
\sqrt{\e_k}(s,x) \leq C_{k,\phi} \langle s \rangle^{-(k+1)/2}
\end{equation}
for $k \geq 1$.  We establish this bound by induction on $k$. The case $k = 1$ already follows from \eqref{joke}, so suppose that $k \geq 2$ and that the claim has already been proven for smaller values of $k$.  Then by repeatedly applying \eqref{sqrtem}, we see from \eqref{psem} that
$$ \partial_s \e_k = \Delta \e_k - 2 \e_{k+1} + O_{k,\phi}( \langle s \rangle^{-2} \e_k + \langle s \rangle^{-(k+5)/2} \e_k^{1/2} )$$
and thus by Cauchy-Schwarz (and discarding the $\e_{k+1}$ term)
$$ \partial_s \e_k \leq \Delta \e_k + O_{k,\phi}( \langle s \rangle^{-2} \e_k + \langle s \rangle^{-k-3} ).$$
Also, from \eqref{psem} with $k$ replaced by $k-1$ we see from \eqref{sqrtem} that
$$ \partial_s \e_{k-1} = \Delta \e_{k-1} - 2 \e_k + O_{k,\phi}( \langle s \rangle^{-k-2} ).$$
For any $s_0 > 1$, we thus have
$$ \partial_s ( (s-s_0) \e_k + C_{k,\phi} \e_{k-1} ) \leq \Delta ((s-s_0) \e_k + C_{k,\phi} \e_{m-1} ) + O_{k,\phi}( \langle s_0 \rangle^{-k-2} )$$
for all $s_0 \leq s \leq 2s_0$, if $C_{k,\phi}$ is sufficiently large.  By induction hypothesis, the quantity $(s-s_0) \e_k + C_{k,\phi} \e_{k-1}$ has size $O_{k,\phi}( \langle s_0 \rangle^{-k-1} )$ at time $s_0$, and hence by the maximum principle is similarly bounded at times $s_0 \leq s \leq 2s_0$.  From this we easily conclude \eqref{sqrtem} for $k$ for times $s \geq 2$; the case $0 \leq s < 2$ can be handled by the local theory.  This closes the induction and establishes \eqref{sqrtem} for all $k$.

From \eqref{sqrtem} (and the boundedness of $\phi$) we know that
\begin{equation}\label{phix}
|(\phi^* \nabla)_x^{k-1} \partial_x \phi| \lesssim_{k,\phi} \langle s \rangle^{-(k+1)/2} 
\end{equation}
for all $k \ge 1$.  By repeated use of \eqref{phicor} and induction we thus establish \eqref{manytime} for $j=0$ and all $k \geq 1$; the case $j=k=0$ was already established from \eqref{zero}.  The case $j > 0$ then follows by repeated use of \eqref{pars} to convert time derivatives to spatial derivatives.  Finally, since $\phi - \phi(\infty)$ is rapidly decreasing and has all derivatives bounded for every fixed $s > 0$, we see that all derivatives are also rapidly decreasing (as can be seen either by Taylor's theorem with remainder, or by using the Gagliardo-Nirenberg inequality \eqref{gag-1} applied to localisations of $\phi - \phi(\infty)$), and so $\phi-\phi(\infty)$ is Schwartz in space.
\end{proof}

\subsection{The caloric gauge}

Let $\phi: \R^+ \times \R^2 \to \H$ be a harmonic map heat flow with classical initial data that equals $\phi(\infty)$ at spatial infinity (and thus at temporal infinity also, thanks to Proposition \ref{Decay}). We now recall the \emph{caloric gauge} from \cite{tao:forges} which places a canonical orthonormal frame on $\phi$.

Given any point $p \in \H$, define an \emph{orthonormal frame} at $p$ to be any orthogonal orientation-preserving map $e: \R^m \to T_p \H$ from $\R^m$ to the tangent space at $p$ (with the metric $h(p)$, of course), and let $\Frame(T_p \H)$ denote the space of such frames; note that this space has an obvious transitive action of the special orthogonal group $SO(m)$.  We then define the \emph{orthonormal frame bundle} $\Frame( \phi^* T\H)$ of $\phi$ to be the space of all pairs $((s,x), e)$ where $(s,x) \in \R^+ \times \R^2$ and $e \in \Frame(T_{\phi(s,x)}\H)$; this is a smooth vector bundle over $\R^+ \times \R^2$.  We then define an \emph{orthonormal frame} $e \in \Gamma(\Frame(\phi^* T\H))$ for $\phi$ to be a section of this bundle, i.e. a smooth assignment $e(s,x) \in \Frame(T_{\phi(s,x)}\H)$ of an orthonormal frame at $\phi(s,x)$ to every point $(s,x) \in \R^+ \times \R^2$.

Each orthonormal frame $e \in \Gamma(\Frame(\phi^* T\H)) \phi$ provides an orthogonal, orientation-preserving identification between the vector bundle $\phi^* T \H$ (with the metric $\phi^* h$) and the trivial bundle $(\R^+ \times \R^2) \times \R^m$ (with the Euclidean metric on $\R^m$), thus sections $\Psi \in \Gamma(\phi^* T\H)$ can be pulled back to functions $e^* \Psi: \R^+ \times \R^2 \to \R^m$ by the formula $e^* \Psi := e^{-1} \circ \Psi$.  The connection $\phi^* \nabla$ on $\phi^* T\H$ can similarly be pulled back to a connection $D$ on the trivial bundle $(\R^+ \times \R^2) \times \R^m$, defined by
\begin{equation}\label{D-def}
 D_i := \partial_i + A_i
 \end{equation}
where $A_i \in \mathfrak{so}(m)$ is the skew-adjoint $m \times m$ matrix field is given by the formula
\begin{equation}\label{Adef}
(A_i)_{ab} = \langle (\phi^* \nabla)_i e_a, e_b \rangle_{\phi^* h}
\end{equation}
where $e_1,\ldots,e_m$ are the images of the standard orthonormal basis for $\R^m$ under $e$.  Of course one similarly has a covariant derivative $D_s = \partial_s + A_s$ in the $s$ direction, defined similarly.

We shall rely frequently on the following fact (cf. Corollary \ref{boch}):

\begin{lemma}[Diamagnetic inequalities]\label{dilemma}  If $\varphi: \R^+ \times \R^2 \to \R^m$ is any smooth function, then
$$ \left|\partial_i |\varphi|\right| \leq |D_i \varphi|$$
and
$$ (\partial_s - \Delta) |\varphi| \leq |(D_s - D_i D_i)\varphi + \eta|$$
in the distributional sense, where $\eta$ is any quantity such that $\eta \cdot \varphi \geq 0$.
\end{lemma}

\begin{proof}  To prove the first inequality, observe that
$$ |\partial_i |\varphi|^2| = |2 \varphi \cdot D_i \varphi| \leq 2 |\varphi| |D_i \varphi|,$$
and the claim then formally follows from the product rule and dividing by $|\varphi|$.  This can be made rigorous by replacing $|\varphi|$ with $(\eps^2 + |\varphi|^2)^{1/2}$ and then taking distributional limits as $\eps \to 0$.

To prove the second inequality, we similarly observe that
\begin{align*}
(\partial_s - \Delta) |\varphi|^2 &= 2 \varphi \cdot (D_s \varphi - D_i D_i \varphi + \eta) - 2 \varphi \cdot \eta - 2 |D_x \varphi|^2 \\
&\leq 2 |\varphi| |D_s \varphi - D_i D_i \varphi + \eta| - 2 |\partial_x |\varphi||^2
\end{align*}
where we have (formally) used the first inequality.  Since we formally have
$$ (\partial_s - \Delta) |\varphi|^2 = 2 |\varphi| ( \partial_s |\varphi| - \Delta |\varphi| ) - 2 |\partial_x |\varphi||^2,$$
the claim then follows, after again replacing $|\varphi|$ with $(\eps^2 + |\varphi|^2)^{1/2}$ to make the arguments rigorous.
\end{proof}

We define the \emph{derivative fields} $\psi_j: \R^+ \times \R^2 \to \R^m$ by the formula 
\begin{equation}\label{psij-def}
\psi_j := e^* \partial_j \phi,
\end{equation}
and similarly define 
\begin{equation}\label{psis-def}
\psi_s := e^* \partial_s \phi.
\end{equation}
We write $\psi_x := (\psi_1, \psi_2)$ and $A_x := (A_1,A_2)$.
The zero-torsion property \eqref{zerotor-phi}, when viewed in the orthonormal frame $e$, becomes the assertion that
\begin{equation}\label{zerotor-frame}
D_i \psi_j = D_j \psi_i
\end{equation}
or equivalently
\begin{equation}\label{zerotor-frame2}
\partial_i \psi_j -\partial_j \psi_i = A_j \psi_i - A_i \psi_j
\end{equation}
while the negative curvature property \eqref{curv-phi} becomes
\begin{equation}\label{curv-frame}
[D_i, D_j] = \partial_i A_j - \partial_j A_i + [A_i,A_j] = - \psi_i \wedge \psi_j
\end{equation}
where $\psi_i \wedge \psi_j$ is the anti-symmetric matrix field
$$ \psi_i \wedge \psi_j := \psi_i \psi_j^* - \psi_j \psi_i^*$$
or in other words 
$$ \psi_i \wedge \psi_j: v \mapsto \psi_i (\psi_j \cdot v) - \psi_j (\psi_i \cdot v).$$
The Leibniz rule \eqref{leibnitz-phi} becomes
\begin{equation}\label{leibnitz-frame}
\partial_i ( \varphi \cdot \varphi' ) = (D_i \varphi) \cdot \varphi' + \varphi \cdot (D_i \cdot \varphi')
\end{equation}
and is equivalent to the antisymmetry of $A$.  Of course, one has analogues of \eqref{zerotor-frame}, \eqref{curv-frame}, \eqref{leibnitz-frame} if $i$ or $j$ is replaced by the $s$ subscript.  Finally, the heat flow equation \eqref{psphi}, when viewed in the frame $e$, becomes
\begin{equation}\label{ps-frame}
\psi_s = D_i \psi_i.
\end{equation}

We observe the gauge symmetry
\begin{equation}\label{gauge}
\begin{split}
\phi \mapsto \phi; &\quad e \mapsto e U; \quad \psi_i \mapsto U^{-1} \psi_i; \\
D_i \mapsto U^{-1} D_i U; \quad &A_i \mapsto U^{-1} \partial_i U + U^{-1} A_i U
\end{split}
\end{equation}
for any choice of \emph{gauge transform} $U: \R^+ \times \R^2 \to SO(m)$, with similar transformations when the $i$ subscript is replaced by $s$.  Geometrically, this transform corresponds to rotating the orthonormal frame $e$ by $U$, leaving the underlying heat flow $\phi$ unchanged.

\begin{definition}[Caloric gauge]  Let $\phi: \R^2 \to \H$ be a smooth function differing from $\phi(\infty)$ by a Schwartz function, and let $\phi: \R^+ \times \R^2 \to \H$ be its heat flow extension (as given by Proposition \ref{Decay}).  We say that a gauge $e$ is a \emph{caloric gauge} for $\phi$ with \emph{boundary frame} $e(\infty) \in \Frame(T_{\phi(\infty)}\H)$ if we have
\begin{equation}\label{ass}
A_s = 0
\end{equation}
throughout $\R^+ \times \R^2$, and if we have 
\begin{equation}\label{esx}
\lim_{s \to \infty} e(s,x) = e(\infty)
\end{equation}
for all $x \in \R^2$.  
\end{definition}

For future reference we record some basic evolution equations in the caloric gauge.

\begin{lemma}\label{evolution}  Let $\phi: \R^+ \times \R^2 \to \H$ be a heat flow with classical initial data, let $e$ be a caloric gauge for $\phi$, and let $\psi_x, \psi_s, A_x$ be the associated derivative fields and connection fields.  Then we have the evolution equations
\begin{align}
\partial_s \psi_x &= D_x \psi_s = \partial_x \psi_s + A_x \psi_s \label{psi-evolve}\\
\partial_s A_x &= - \psi_s \wedge \psi_x\label{sax} \\
\partial_s \psi_x &= D_i D_i \psi_x - (\psi_x \wedge \psi_i) \psi_i \label{psix-heat} \\
\partial_s \psi_s &= D_i D_i \psi_s - (\psi_s \wedge \psi_i) \psi_i. \label{psis-heat}
\end{align}
and the inequalities
\begin{align}
\partial_s |\psi_x| &\leq \Delta |\psi_x| \label{psix-delta} \\
\partial_s |\psi_s| &\leq \Delta |\psi_s| \label{psis-delta}
\end{align}
holding in the distributional sense. In particular, from Corollary \ref{compar} we have the pointwise bounds
\begin{align}
|\psi_x(s)| &\leq e^{s\Delta} |\psi_x(0)| \label{psix-compar} \\
|\psi_s(s)| &\leq e^{s\Delta} |\psi_s(0)|. \label{psis-compar}
\end{align}
These identities and inequalities do not require the gauge to obey \eqref{esx}.  
\end{lemma}

\begin{proof}  The identity \eqref{psi-evolve} follows from \eqref{zerotor-frame}, \eqref{D-def} and \eqref{ass}, while \eqref{sax} follows from \eqref{curv-frame} and \eqref{ass}.  To prove \eqref{psix-heat}, we use \eqref{zerotor-frame}, \eqref{ass}, \eqref{curv-frame}, \eqref{ps-frame} to compute
\begin{align*}
\partial_s \psi_x &= D_s \psi_x \\
&= D_x \psi_s \\
&= D_x D_i \psi_i \\
&= D_i D_x \psi_i + (\psi_x \wedge \psi_i) \psi_i \\
&= D_i D_i \psi_x + (\psi_x \wedge \psi_i) \psi_i 
\end{align*}
as desired.  The proof of \eqref{psis-heat} is similar.  The inequalities \eqref{psix-delta}, \eqref{psis-delta} then follow from \eqref{psix-heat}, \eqref{psis-heat} by Lemma \ref{dilemma}, noting that $((v \wedge w) w) \cdot v = \frac{1}{2} |v \wedge w|^2$ is non-negative for any $v,w \in \R^m$.
\end{proof}

A fundamental fact is that caloric gauges exist and are unique once one specifies $e(\infty)$ (cf. \cite[Theorem 2.9]{tao:forges}):

\begin{theorem}[Existence and uniqueness of caloric gauge]\label{caloric-thm}  Let $\phi: \R^+ \times \R^2 \to \H$ be a heat flow with classical initial data that equals $\phi(\infty) \in \H$ at infinity, and let $e(\infty) \in \Frame( T_{\phi(\infty)} \H)$ be an orthonormal frame at $\phi(\infty)$.  Then there exists a unique caloric gauge $e \in \Gamma( \Frame( \phi^* T\H ) )$ for $\phi$ with boundary frame $e(\infty)$.  Furthermore, for each fixed $s$, $\psi_i$ and $\psi_s$ are Schwartz functions in space, and one has the qualitative decay estimates
\begin{align}
|\partial_s^j \partial_x^k \psi_x| &\lesssim_{j,k,\phi} \langle s \rangle^{-(2+k+2j)/2}\label{phi-1}\\
|\partial_s^j \partial_x^k \psi_s| &\lesssim_{j,k,\phi} \langle s \rangle^{-(3+k+2j)/2} \label{phi-2}\\
|\partial_s^j \partial_x^k A_x| &\lesssim_{j,k,\phi} \langle s \rangle^{-(3+k+2j)/2}\label{phi-3}
\end{align}
on $\R^+ \times \R^2$ for all $j,k \geq 0$ where we abbreviate $\psi_x := (\psi_1,\psi_2)$ and $A_x := (A_1,A_2)$.
\end{theorem}

\begin{proof} Let us first prove uniqueness.   If $e$ is a caloric gauge, we see from \eqref{ass} that
\begin{equation}\label{phin}
(\phi^* \nabla)_s e_a =0 
\end{equation}
for all $a=1,\ldots,m$.  In particular, if $e'$ is another caloric gauge with the same boundary frame $e(\infty)$, we see that $|e_a-e'_a|$ is constant in $s$.  Since $e_a-e'_a$ vanishes in the limit $s \to \infty$, the uniqueness follows.

Now we establish existence.  We place an arbitrary Schwartz orthonormal frame $e(0,x) \in \Frame( T_{\phi(0,x)} \H )$ on the initial data $\phi(0,\cdot)$ (such a frame exists since the spatial domain $\R^2$ is contractible and $\phi(0,\cdot)$ differs from a constant by a Schwartz function).  We then evolve this frame in $s$ using \eqref{phin}; the fact that $\phi^* \nabla$ respects the metric on $\phi^* T\H$ shows that $e$ remains an orthonormal frame as $s$ increases, and from the Picard existence theorem we see that $e$ can be defined globally.  The smoothness of $\phi$ and $e(0,\cdot)$ easily implies that $e$ is also smooth, and so $A$ and $\psi$ are also smooth.

Recall from the proof of Proposition \ref{Decay} (or from \eqref{manytime}) that we have \eqref{phix}.  Viewing this estimate in the frame $e$, we conclude that
\begin{equation}\label{dxpsi}
|D_x^{k-1} \psi_x| \lesssim_{k,\phi} \langle s \rangle^{-(k+1)/2} 
\end{equation}
for all $k \geq 1$.  Since $\psi_s = D_i \psi_i$, we also have
\begin{equation}\label{dxpsis}
|D_x^{k-1} \psi_s| \lesssim_{k,\phi} \langle s \rangle^{-(k+2)/2}
\end{equation}
for all $k \geq 1$.
On the other hand, from differentiating \eqref{sax} we obtain
$$ |\partial_s D_x^k A_x| \lesssim_k  \sum_{0 \leq j \leq k} |D_x^j \psi_s| |D_x^{k-j} \psi_x|$$
for any $k \geq 0$, where the covariant derivative $D_x$ acts on matrix fields $B$ by the formula 
\begin{equation}\label{matrix-cov}
D_x B := \partial_x B + [A,B], 
\end{equation}
thus in particular we have the Leibniz rule
\begin{equation}\label{wedge-leibnitz}
D_x ( u \wedge v ) = (D_x u) \wedge v + u \wedge (D_x v)
\end{equation}
 for any vector fields $u,v$.  Applying \eqref{dxpsi} we conclude that 
\begin{equation}\label{dax}
|\partial_s D_x^k A_x| \lesssim_{k,\phi} \langle s \rangle^{-(k+5)/2};
\end{equation}
since $D_x^k A_x$ is initially bounded at time zero, we conclude that $D_x^k A_x = O_{k,\phi}(1)$ throughout $\R^+ \times \R^2$ for all $k \geq 0$.    Using the definition of $D_x$ on matrix fields, we thus also conclude that 
\begin{equation}\label{nax}
|\partial_x^k A_x| \lesssim_{k,\phi} 1
\end{equation}
for all $k \geq 0$, and thus by \eqref{dxpsi}, \eqref{dxpsis} we obtain
\begin{equation}\label{psax}
 |\partial_x^k \psi_x| \lesssim_{k,\phi} \langle s \rangle^{-1}; \quad |\partial_x^k \psi_s| \lesssim_{k,\phi} \langle s \rangle^{-3/2} 
 \end{equation}
for all $k \geq 0$.  Applying this to \eqref{sax}, we see that $A_x(s)$ converges in $C^\infty_x(\R^2 \to \mathfrak{so}(m))$ to some limit $A_x(\infty)$ as $s \to \infty$.

Since $A_x(s)$ is bounded in $C^\infty(\R^2 \to \mathfrak{so}(m))$, and $\phi(s)$ is bounded uniformly in $s$, we see from \eqref{Adef} that $e(s,\cdot)$ is locally bounded in $C^\infty$ uniformly in $s$, where we use the embedding \eqref{hyperdef} to view $e(s,x)$ as a linear transformation from $\R^m$ to $\R^{1+m}$.  Also, from \eqref{manytime} we have $\partial_s \phi = O_\phi( \langle s \rangle^{-3/2} )$, which when combined with \eqref{phin} shows that $e(s,\cdot)$ converges uniformly as $s \to \infty$ to a limit $e(\infty,\cdot)$.  Using the local $C^\infty$ bounds, we see that this convergence is also in $C^\infty_\loc$, and so $e(\infty,\cdot)$ is smooth.  In particular, we can take limits in \eqref{Adef} and conclude that $A_x(\infty,\cdot)$ are the connection coefficients for $e(\infty,\cdot)$.

Now we can apply a gauge transformation \eqref{gauge} by some smooth gauge transform $U(s,x) = U(x)$ independent of $x$ to normalise $e(\infty,\cdot) = e(\infty)$, thus creating a caloric gauge.  Note that in this gauge $A_x(s)$ converges in $C^\infty_{\operatorname{loc}}(\R^2 \to \mathfrak{so}(m))$ to zero as $s \to \infty$.  Also from \eqref{phix} we see that $\psi_x(s)$ is locally bounded in $C^\infty(\R^2 \to \R^m)$ uniformly in $s$.

The final task is to establish the bounds \eqref{phi-1}, \eqref{phi-2}, \eqref{phi-3} and to establish that $\psi_s, \psi_x$ are Schwartz.  Note that all the previous bounds on $e, A, \psi$ need not apply any more, because we have taken a gauge transform.  Nevertheless, we can largely recover these bounds from the caloric gauge condition.  First of all, the bounds \eqref{dxpsi}, \eqref{dxpsis}, \eqref{sax} continue to hold in this gauge, and so \eqref{dax} does also.  Using the boundary condition at $s=+\infty$ now instead of $s=0$, we now conclude a stronger decay estimate on $A$, namely that
$$ |D_x^k A_x| \lesssim_{k,\phi} \langle s \rangle^{-(k+3)/2}$$
for all $k \geq 0$. Writing $\partial_x = D_x - A_x$ we thus conclude that
$$ |\partial_x^k A_x| \lesssim_{k,\phi} \langle s \rangle^{-(k+3)/2} $$
for $k \geq 0$.  Inserting this back into \eqref{dxpsi}, \eqref{dxpsis} we now have
\begin{equation}\label{sasx}
|\partial_x^k \psi_x| \lesssim_{k,\phi} \langle s \rangle^{-(k+2)/2} 
\end{equation}
and
\begin{equation}\label{sass}
 |\partial_x^k \psi_s| \lesssim_{k,\phi} \langle s \rangle^{-(k+3)/2} 
\end{equation}
for $k \geq 0$.   This gives \eqref{phi-1}, \eqref{phi-2}, \eqref{phi-3} in the case $j=0$.  To handle the $j > 0$ case, we use the evolution equations from Lemma \ref{evolution}.
From repeated use of these evolution equations and induction on $j$ we obtain the $j>0$ cases of \eqref{phi-1}, \eqref{phi-2}, \eqref{phi-3} from the $j=0$ case.

Finally, we need to verify that $\psi_x$ and $\psi_s$ are Schwartz for each fixed $s$.  We already know from \eqref{sasx}, \eqref{sass} that all derivatives of $\psi_x, \psi_s$ are bounded.  Also, from Proposition \ref{Decay} we know that $\partial_x \phi$, $\partial_s \phi$ are rapidly decreasing, and hence $\psi_x, \psi_s$ are too.  Using Taylor's theorem with remainder (or \eqref{gag-1} and localisation) one can then verify that the derivatives of $\psi_x, \psi_s$ must also be rapidly decreasing, and the claim follows.
\end{proof}

\begin{remark}\label{rotor}
Suppose that $\phi: \R^+ \times \R^2 \to \H$ is a heat flow with classical initial data that equals $\phi(\infty)$ at infinity, let $e(\infty)$ be a frame for $\phi(\infty)$, and let $e$ be the corresponding caloric gauge given by Theorem \ref{caloric-thm}.  Observe that if $U \in SO(m,1)$ is any Lorentz transform (which thus acts on $\H$ using the representation \eqref{hyperdef}), then $U \circ \phi$ is another heat flow with classical initial data that equals $U(\phi(\infty))$ at infinity, and $U_* e$ is the caloric gauge with frame $U_* e(\infty) = dU(\phi(\infty)) \circ e(\infty)$ at infinity.  In particular, restricting $U$ to the stabiliser of $\phi(\infty)$ (which is isomorphic to the rotation group $SO(m)$), we see that the choice of frame $e(\infty)$ at infinity only affects the caloric gauge up to rotation.
\end{remark}

\begin{remark} While the fields $\psi_s$, $\psi_x$ are Schwartz, the same does not appear to be true of the connection field $A_x$; the estimates above can be used to establish a bound of the form $A_x = O( \min( \langle s \rangle^{-3/2}, \langle x \rangle^{-3} )$, which is a substantial amount of spatial decay (in particular, making $A_x$ absolutely integrable) but is not an infinite amount (although higher derivatives of $A_x$ will exhibit better decay).  Fortunately we will not need an infinite amount of decay on $A_x$ in our arguments.
\end{remark}

\subsection{Smooth deformations of heat flows}

Until now, we have constructed heat flows $\phi: \R^+ \times \R^2 \to \H$ and caloric gauges $e \in \Gamma(\Frame( \phi^* T\H ) )$ for a single (static) initial data $\phi: \R^2 \to \H$.  However, for applications to wave maps, it is necessary to construct these flows and gauges for a \emph{time-varying} (dynamic) field $\phi: \R^{1+2} \to \H$, thus obtaining a dynamic family of heat flows $\phi: \R^+ \times \R^{1+2} \to \H$ from the product spacetime $\R^+ \times \R^{1+2} := \{ (s,t,x): s \geq 0; t \in \R; x \in \R^2 \}$ to $\H$, creating a dynamic caloric gauge $e \in \Gamma(\Frame( \phi^* T\H ) )$ together with differentiated fields $\psi_x, \psi_t, \psi_s$ and connection fields $A_x, A_t$.  We now record the basic qualitative properties of such a construction.

\begin{theorem}[Dynamic caloric gauges]\label{dynamic-caloric}  Let $I$ be a time interval with non-empty interior, and let $\phi: I \times \R^2 \to \H$ be a smooth map which differs from constant $\phi(\infty)$ by a Schwartz function in space, and let $e(\infty) \in \Frame(T_{\phi(\infty)}\H)$ be a frame for $\phi(\infty)$.  Then $\phi$ extends smoothly to a dynamic heat flow $\phi: \R^+ \times I \times \R^2 \to \H$ that converges in $C^\infty_{\loc}(I \times \R^2)$ to $\phi(\infty)$ as $s \to \infty$, and there exists a unique smooth frame $e \in \Gamma(\Frame(\phi^* T\H))$ such that $e(t)$ is a caloric gauge for $\phi(t)$ which equals $e(\infty)$ at infinity for each $t \in I$.  
All derivatives of $\phi - \phi(\infty)$ in the variables $t,x,s$ are Schwartz in $x$ for each fixed $t,s$.

Furthermore, the time differentiated field $\psi_t$ obeys the linear parabolic equation
\begin{equation}\label{dst}
 \partial_s \psi_t = D_i D_i \psi_t + (\psi_t \wedge \psi_i) \psi_i,
\end{equation}
while the time connection $A_t$ obeys the ODE
\begin{equation}\label{ast}
\partial_s A_t = - \psi_s \wedge \psi_t
\end{equation}
with $\lim_{s \to \infty} A_t = 0$ for each $t,x$.

\end{theorem}

\begin{proof}  The uniqueness follows from the uniqueness in Theorem \ref{caloric-thm}, so we now turn to existence.  Without loss of generality we can take $I$ to be compact.  For each fixed time $t$, we can extend the initial data $\phi(t,\cdot): \R^2 \to \H$ to a smooth heat flow $\phi(\cdot,t,\cdot): \R^+ \times \R^2 \to \H$ thanks to Proposition \ref{Decay}.  An inspection of the proof of that proposition shows that the global heat flow was obtained by gluing together local heat flows obtained by a Picard iteration method.  As the expression \eqref{duh-local} being iterated in that method is a smooth function of the unknown $\phi$ (in fact, it is real analytic), it follows by standard arguments that the solution $\phi$ depends smoothly (and even real analytically) on the initial data, as measured in the $C^1(\R^2)$ topology.  An inspection of the proof of Proposition \ref{Decay} reveals that all the constants that depend on the initial data $\phi(t,\cdot)$ in fact only depend on a bound on finitely many derivatives of that data, as well as a bound on the support of that data, and so can be made uniform in $t$ since $I$ is compact.  From this we see that the map $t \mapsto \phi(\cdot,t,\cdot)$ is locally smooth in smooth topologies, and so we can glue together all the heat flows $\phi(\cdot,t,\cdot)$ to create a smooth dynamic heat flow $\phi: \R^+ \times I \times \R^2 \to \H$.

Unfortunately, as we are gluing together infinitely many local flows to create the global flow, this argument does not directly establish uniform smoothness of $\phi$ in the limit $s \to \infty$.  However, this can be remedied by constructing the heat flow on an infinite heat-temporal interval $[s_0,+\infty)$ directly by an iteration method.  The starting point is the the variant
\begin{equation}\label{duh-2}
\phi(s,t) = e^{(s-s_0)\Delta} \phi(s_0,t) - \int_{s_0}^s e^{(s-s')\Delta} (|\partial_x \phi|_{\phi^* h}^2 \phi)(s',t)\ ds'
\end{equation}
of \eqref{duh-local} for any $s_0 \geq 0$, which of course follows from \eqref{duh}.  Now we take $s_0$ to be large and consider the norm $X$ defined by
$$ \|u\|_X := \sup_{s \in [s_0,+\infty)} s \| u(s) \|_{\dot C^1(\R^2)}.$$
Observe from \eqref{manytime} that $\phi(\cdot,t,\cdot)$ has norm $O_{\phi}(1)$ in this space, regardless of the values of $s_0$ or $t$; from \eqref{duh-2}, \eqref{strict-reg}, and the triangle inequality this implies that the linear evolution $e^{(s-s_0)\Delta} \phi(s_0,t)$ does also, if $s_0$ is sufficiently large depending on $\phi$.  For similar reasons, if $s_0$ is sufficiently large depending on $\phi$, the Picard iteration map associated to \eqref{duh-2} is a contraction on the ball of $X$ of radius $O_\phi(1)$ centred at the origin.  Because of this, we can construct the solution $\phi(\cdot,t,\cdot)$ on $[s_0,+\infty) \times \R^2$ from the initial data $\phi(s_0,t,\cdot)$ by a single iteration scheme, uniformly in $t$; and so $\phi$ varies smoothly in $t$ in the $X$ topology, in particular we have
$$ |\partial_t^j \partial_x \phi|\lesssim_{\phi,j} \langle s \rangle^{-1} $$
for all $j \geq 0$.  By differentiating \eqref{duh-2} and using higher order analogues of \eqref{strict-reg}, one can similarly obtain
$$ |\partial_t^j \partial_x^m \phi| \lesssim_{\phi,j,m} \langle s \rangle^{-(m+1)/2} $$
for all $m \geq 1$.  In particular, from \eqref{pars} we have
\begin{equation}\label{sthok}
 |\partial_t^j \partial_x^m \partial_s \phi| \lesssim_{\phi,j,m}  \langle s \rangle^{-(m+3)/2} 
\end{equation}
for all $j,m \geq 0$.  Since we also have $\phi = \phi(\infty) + O_\phi(s^{-1/2})$ by \eqref{manytime}, it is not hard to conclude that $\phi(s)$ converges in $C^\infty_\loc(I \times \R^2)$ to $\phi(\infty)$ as claimed.  A modification of the argument also shows that $\phi-\phi(\infty)$ and its derivatives are Schwartz in $x$ for each $t,s$.

We now repeat the construction of the caloric gauge in Theorem \ref{caloric-thm}, but now taking the dynamic variable $t$ into account. Namely, we begin as before by selecting an arbitrary smooth orthonormal frame $e(0,t,x) \in \Frame( T_{\phi(0,t,x)} \H )$ on the initial data $\phi(0,\cdot,\cdot)$, which differs from $e(\infty)$ by a Schwartz function in spac.  We then extend this in $s$ using \eqref{phin} as before; using \eqref{sthok} and Picard iteration we see that $e$ is smooth in all variables and extends smoothly to a limit $e(\infty,t,x)$ as $s \to \infty$.  We then repeat the arguments in Theorem \ref{caloric-thm} and establish a smooth caloric gauge $e \in \Gamma(\Frame( \phi^* T\H ) )$ as required.  Finally, \eqref{dst}, \eqref{ast} follow from the arguments used to prove Lemma \ref{evolution}. The convergence of $A_t$ to zero as $s \to \infty$ follows from the smoothness of $e$ all the way up to $s=\infty$ and the fact that $e(\infty,t,x) = e(\infty)$ is constant in $t$.
\end{proof}

\section{The harmonic map heat flow and the caloric gauge II.  Quantitative estimates}\label{littlewood-sec}

In the previous section we established various qualitative properties of heat flows with classical data (especially when viewed in the caloric gauge), in which the implied constants were allowed to depend on smooth norms of the initial data.  Now we turn to the more quantitative theory, in which we still work with classical heat flows (in order to easily justify all integration by parts, etc.) but the constants are only allowed to depend on the energy of the heat flow rather than on any higher regularity norms.  Such estimates are of course essential if we are to use the heat flow to construct an energy space.  Some similar computations in the context of the global regularity problem for wave maps, but using an extrinsic regularisation method in place of the harmonic map heat flow, also appear in \cite{tataru:wave3}.  Whereas the qualitative decay estimates in the previous section were ``subcritical'' (favourable with respect to scaling), the quantitative estimates here will be ``critical'' (scale-invariant), basically because of the scale-invariant nature of the energy functional.

Throughout this section, the reader may find it useful to keep the (non-rigorous) dimensional analysis heuristics
\begin{align*}
\partial_x, D_x &\approx s^{-1/2}; \\
\partial_s &\approx s^{-1};\\
\|u\|_{L^\infty_x(\R^2)} &\lessapprox s^{-1/2} \|u\|_{L^2_x(\R^2)};\\
\sup_{s > 0} |f(s)| &\lessapprox \int_0^\infty |f(s)|\frac{ds}{s};\\
\| \psi_x \|_{L^2_x(\R^2)}, \| A_x \|_{L^2_x(\R^2)} &\lessapprox 1;\\
\psi_s &\approx \partial_x \psi_x; \\
\psi_x(s) &\approx e^{s\Delta} \psi_x(0)
\end{align*}
in mind.

We begin with some parabolic regularity estimates which, roughly speaking, assert that the covariant derivatives of a heat flow with bounded energy enjoy the same decay estimates as their linear counterparts (i.e. solutions to the free heat equation with finite energy initial data).

\begin{proposition}[Covariant parabolic regularity]\label{corpar} Let $\phi: \R^+ \times \R^2 \to \H$ be a heat flow with classical initial data, and with energy bound
\begin{equation}\label{en}
 \int_{\R^2} \e_1(0,x)\ dx \leq E
\end{equation}
for some $E > 0$, where the energy densities $\e_k$ were defined in \eqref{energy-dens-def}.  Then one has the bounds
\begin{align}
|\int_0^\infty \int_{\R^2} s^{k-1} \e_{k+1}(s,x)\ dx ds| &\lesssim_{E,k} 1 \label{l2-integ} \\
\sup_{s > 0} s^{k-1} \int_{\R^2} \e_{k}(s,x)\ dx &\lesssim_{E,k} 1 
\label{l2-const} \\
\sup_{s > 0, x \in \R^2} s^{k} \e_k(s,x) &\lesssim_{E,k} 1 
\label{linfty-const}
\end{align}
for all $k \geq 1$.
\end{proposition}

The estimates here should be compared with those in Section \eqref{linear-heat}.  Note, in contrast to the qualitative estimates in Proposition \ref{Decay}, the bounds here depend only on the energy rather than on $\phi$ itself.

\begin{proof}  For simplicity of notation we allow all implied constants to depend on $E$ and $k$.

We induct on $k$, beginning with the base case $k=1$.  From \eqref{bochner-3} and Corollary \ref{compar} we have the pointwise estimate
\begin{equation}\label{point}
 \sqrt{\e_1(s,x)} \leq (e^{s\Delta} \sqrt{\e_1(0)})(x)
 \end{equation}
for $s \geq 0$.
The claims \eqref{l2-const}, \eqref{linfty-const} for $k=0$ then follow from \eqref{en}, \eqref{heat-lp}.  Next, from integrating \eqref{bochner} in space we obtain the energy identity
$$ \partial_s \int_{\R^2} \e_1\ dx \leq - 2 \int_{\R^2} \e_2\ dx.$$
Integrating this in $s$, we establish \eqref{l2-integ} for $k=0$ from \eqref{l2-const} for $k=0$.

Now suppose that $k \geq 1$ and that the claims have already been established for smaller values of $k$.  Integrating \eqref{psem} we have
\begin{align*}
\partial_s \int_{\R^2} \e_k\ dx &= - 2 \int_{\R^2} \e_{k+1}\ dx \\
&\quad + 
O\left( \sum_{a,b,c \geq 1: a+b+c=k+2} \int_{\R^2} \e_a^{1/2} \e_b^{1/2} \e_c^{1/2} \e_k^{1/2}\ dx \right).
\end{align*}
By symmetry we may take $a \geq b \geq c$, which forces $a \geq 2$ and $b,c < k$.  Applying the inductive hypothesis \eqref{linfty-const} we conclude
$$ 
\partial_s \int_{\R^2} \e_k\ dx = - 2 \int_{\R^2} \e_{k+1}\ dx + 
O\left( \sum_{2 \leq a \leq k} \int_{\R^2} s^{-(k+2-a)/2} \e_a^{1/2} \e_k^{1/2}\ dx \right).$$
Writing $E_a(s) := \int_{\R^2} \e_a\ ds$, we conclude from Cauchy-Schwarz that
$$ 2 E_{k+1} + \partial_s E_k = O\left( \sum_{a=2}^k s^{-(k+2-a)/2} E_a^{1/2} E_k^{1/2} \right)$$
and hence by the arithmetic mean-geometric mean inequality
\begin{equation}\label{2m}
 2 E_{k+1} + \partial_s E_k = O\left( \sum_{a=2}^k s^{a-k-1} E_a \right).
\end{equation}
Suppose we integrate this against $\psi(s/s_0)$, where $s_0 > 0$ and $\psi$ is a non-negative smooth cutoff function supported on $[1/2,2]$.  We conclude that
$$ \int_0^\infty \psi(s/s_0) E_{k+1} \leq O\left( \sum_{a=2}^k s_0^{a-k-1} \int_{s_0/2}^{2s_0} E_a\ ds \right).$$
Summing this dyadically in $s_0$ and using the inductive hypothesis \eqref{l2-integ}, we establish \eqref{l2-integ} for $k$ as required.  If we now return to \eqref{2m} and integrate this on an interval $[s_0,s_1]$ with $s_0 \leq s_1 \leq 2s_0$, we conclude using \eqref{l2-integ} and Cauchy-Schwarz that
$$ E_k(s_1) - E_k(s_0) = O( s_0^{1-k} ).$$
Combining this with \eqref{l2-integ} for $k-1$, we quickly obtain \eqref{l2-const} for $k$ as required.

The final task is to establish \eqref{linfty-const} for $k$.  We start here with \eqref{psem-sqrt}.  Using Duhamel's formula \eqref{duh} and \eqref{heat-lp}, we see that
\begin{align*}
\| \sqrt{\e_k}(s_0) \|_{L^\infty_x(\R^2)} &\lesssim s_0^{-1/2} \| \sqrt{\e_k}(s_0/2) \|_{L^2_x(\R^2)}\\
&\quad  + \int_{s_0/2}^{s_0} \sum_{a,b,c \geq 1: a+b+c=k+2} (s_0-s)^{-1/2} \| \e_a^{1/2} \e_b^{1/2} \e_c^{1/2} \|_{L^2_x(\R^2)}\ ds
\end{align*}
for all $s_0 > 0$.
Applying \eqref{l2-const} (which has been established up to $m$) and \eqref{linfty-const} (which has been established up to $k-1$) we conclude
$$
\| \sqrt{\e_k}(s_0) \|_{L^\infty_x(\R^2)} \lesssim s_0^{-1/2} s_0^{(1-k)/2}
+ \int_{s_0/2}^{s_0} \sum_{a,b,c \geq 1: a+b+c=k+2} (s_0-s)^{-1/2} s_0^{(-k-1)/2}\ ds$$
and \eqref{linfty-const} for $k$ follows.
\end{proof}

We also have the following variant estimate:

\begin{lemma}[Integrated $L^\infty$ parabolic regularity]\label{infty-lemma} Let the notation and assumptions be as in the previous proposition.  Then we also have
$$ \int_0^\infty s^{k-1} \sup_{x \in \R^2} \e_k(s,x)\ ds \lesssim_{E,k} 1$$
for all $k \geq 1$.
\end{lemma}

\begin{proof}  From the Gagliardo-Nirenberg inequality \eqref{gag-3} we have
$$ \| \sqrt{\e_k}(s) \|_{L^\infty_x(\R^2)} \lesssim
\| \sqrt{\e_k}(s) \|_{L^2_x(\R^2)}^{1/3} \| \partial_x \sqrt{\e_k}(s) \|_{L^4_x(\R^2)}^{2/3}$$
and hence by the diamagnetic inequality \eqref{diamagnetic}
$$ \| \sqrt{\e_k}(s) \|_{L^\infty_x(\R^2)} \lesssim
\| \sqrt{\e_k}(s) \|_{L^2_x(\R^2)}^{1/3} \| \sqrt{\e_{k+1}}(s) \|_{L^4_x(\R^2)}^{2/3}.$$
Another application of the Gagliardo-Nirenberg inequality \eqref{gag-4} gives
$$ \| \sqrt{\e_{k+1}}(s) \|_{L^4_x(\R^2)} \lesssim
\| \sqrt{\e_{k+1}}(s) \|_{L^2_x(\R^2)}^{1/2} \| \nabla \sqrt{\e_{k+1}}(s) \|_{L^2_x(\R^2)}^{1/2}$$
and thus by the diamagnetic inequality \eqref{diamagnetic} we conclude that
$$
 \| \sqrt{\e_k}(s) \|_{L^\infty_x(\R^2)} \lesssim
\prod_{j=0,1,2} \| \sqrt{\e_{k+j}}(s) \|_{L^2_x(\R^2)}^{1/3}.
$$
For $k \geq 2$, the claim now follows from \eqref{l2-integ}, \eqref{l2-const}, and H\"older's inequality.  

For $k=1$, the above argument does not quite work (it would require the $k=0$ case of \eqref{l2-integ}, which fails); nevertheless, the claim follows immediately in this case from \eqref{point}, \eqref{en}, and Lemma \ref{ubang}.
\end{proof}

Now we work in the caloric gauge and control the connection $A_x$ and its derivatives.

\begin{proposition}[Connection bounds]\label{abound}  Let $\phi$ be a heat flow with classical initial data obeying the energy bound \eqref{en}, let $e$ be a caloric gauge for $\phi$, and let $A$ be the connection coefficients.  Then we have the pointwise bounds
\begin{align}
\| \partial_x^k A_x(s) \|_{L^\infty_x(\R^2)} &\lesssim_{k,E} s^{-(k+1)/2} \label{ax-infty} \\
\| \partial_x^k A_x(s) \|_{L^2_x(\R^2)} &\lesssim_{k,E} s^{-k/2} \label{ax-2} 
\end{align}
for all $k \geq 0$ and $s > 0$, as well as the integrated estimates
\begin{align}
\int_0^\infty s^{(k-1)/2} \| \partial_x^{k+1} A_x(s) \|_{L^2_x(\R^2)}\ ds &\lesssim_{k,E} 1 \label{aint-2} \\
\int_0^\infty s^{(k-1)/2} \| \partial_x^k A_x(s) \|_{L^\infty_x(\R^2)}\ ds &\lesssim_{k,E} 1 \label{aint-infty}.
\end{align}
for all $k \geq 0$.
\end{proposition}

\begin{remark}\label{al1} It is also possible to obtain the estimate $\|\partial_x^{k+1} A_x\|_{L^1_x(\R^2)} \lesssim_{k,E} s^{-k/2}$ for $k \geq 0$, but we will not need this estimate here and so will omit the proof.  (The $k=0$ case of this estimate is in fact somewhat delicate, requiring some non-trivial paraproduct estimates from harmonic analysis.)
\end{remark}

\begin{proof} As before we omit the dependence of the implied constants on $m$ and $E$. From \eqref{sax} we have the integral formula
\begin{equation}\label{sax-integrated}
 A_x(s) = \int_s^\infty \psi_s(s') \wedge \psi_x(s')\ ds'.
\end{equation}
Repeatedly differentating this covariantly using \eqref{matrix-cov}, \eqref{wedge-leibnitz} we obtain
\begin{equation}\label{dmax}
 |D_x^k A_x(s)| \lesssim \int_s^\infty \sum_{j=0}^k \sqrt{\e_{j+2}(s')}  \sqrt{\e_{k-j+1}(s')}\ ds'.
\end{equation}
Applying \eqref{linfty-const} we obtain
$$ \| D_x^k A_x(s) \|_{L^\infty_x(\R^2)} = O( s^{-(k+1)/2} )$$
for every $k \geq 0$, and the claim \eqref{ax-infty} then follows from \eqref{matrix-cov} and an inductive argument.  In a similar spirit, from \eqref{dmax} and the Minkowski and H\"older inequalities, we have
\begin{equation}\label{dxax}
 \| D_x^k A_x(s) \|_{L^2_x(\R^2)} \lesssim \int_s^\infty \sum_{j=0}^k \left\| \sqrt{\e_{j+2}(s')} \right\|_{L^2_x(\R^2)}
\left\| \sqrt{\e_{k-j+1}(s')} \right\|_{L^\infty_x(\R^2)}\ ds'.
\end{equation}
Using \eqref{l2-integ}, Lemma \ref{infty-lemma} and Cauchy-Schwarz we obtain \eqref{ax-2} with the ordinary derivatives $\partial_x$ replaced by covariant ones, but by using \eqref{matrix-cov}, \eqref{ax-infty} one can recover the ordinary derivatives.  

Next, applying \eqref{dxax} with $k$ replaced by $k+1$ and then using the arithmetic mean-geometric mean inequality we have
\begin{align*}
 \| D_x^{k+1} A_x(s) \|_{L^2_x(\R^2)} &\lesssim \sum_{j=0}^{k+1} \int_s^\infty (s')^{j-1/2} \left\| \sqrt{\e_{j+2}(s')} \right\|_{L^2_x(\R^2)}^2 \\
 &\quad \quad + s^{1/2-j} \left\| \sqrt{\e_{k-j+2}(s')} \right\|_{L^\infty_x(\R^2)}\ ds'.
\end{align*}
The claim \eqref{aint-2} (with ordinary derivatives replaced by covariant ones) then follows from \eqref{l2-integ}, Lemma \ref{infty-lemma}, and Fubini's theorem, and one can then recover the ordinary derivatives using \eqref{matrix-cov}, \eqref{ax-infty} as before.  

Finally, we turn to \eqref{aint-infty}.  If $k \geq 1$, we can apply the Gagliardo-Nirenberg inequality \eqref{gag-2} to deduce \eqref{aint-infty} from \eqref{aint-2}, so it suffices to check the case $k=0$.  From \eqref{dmax} and Fubini's theorem we can bound the left-hand side of \eqref{aint-infty} by
$$ \int_0^\infty (s')^{1/2} \| \sqrt{e_2}(s') \|_{L^\infty_x(\R^2)} \| \sqrt{e_1}(s') \|_{L^\infty_x(\R^2)}\ ds'$$
and the claim follows from Lemma \ref{infty-lemma} and Cauchy-Schwarz.
\end{proof}

\begin{remark}  The $k=0$ case of \eqref{ax-2} asserts in particular that the \emph{Coulomb functional} $\int_{\R^2} |A_x(0,x)|^2\ dx$ of the gauge is bounded by $O_E(1)$.  In particular, if one works not in the caloric gauge, but in the \emph{minimal Coulomb gauge} of H\'elein \cite{helein}, defined as the orthonormal frame that minimises the Coulomb functional, we thus conclude that the Coulomb functional is $O_E(1)$ in this gauge also.  Intriguingly, this fact seems to be rather difficult to deduce directly from the Coulomb gauge condition $\partial_i A_i$, even when the energy $E$ is small (so that the Coulomb gauge becomes unique).  Thus we see that the caloric gauge can be used to deduce some non-trivial facts about other gauges as well.
\end{remark}

As a corollary we may replace the covariant derivatives in Proposition \ref{corpar} and Lemma \ref{infty-lemma} with ordinary derivatives, thus obtaining the same bounds for $\psi_x$ that we have just established\footnote{As a rule of thumb, $A_x$ seems to always obey at least as good bounds as $\psi_x$, although there are also additional estimates for $A_x$, such as the $L^1$ estimates mentioned in Remark \ref{al1}, which are not obeyed by $\psi_x$.} for $A_x$:

\begin{corollary}\label{corbound}  With the assumptions and notation in Proposition \ref{abound}, we have
\begin{align}
\int_0^\infty s^{k-1} \| \partial_x^k \psi_x \|_{L^2_x(\R^2)}^2 ds &\lesssim_{E,k} 1 \label{l2-integ-ord} \\
\sup_{s > 0} s^{(k-1)/2} \| \partial_x^{k-1} \psi_x \|_{L^2_x(\R^2)} &\lesssim_{E,k} 1 
\label{l2-const-ord} \\
\int_0^\infty s^{k-1} \| \partial_x^{k-1} \psi_x \|_{L^\infty_x(\R^2)}^2\ ds &\lesssim_{E,k} 1\label{linfty-integ-ord} \\
\sup_{s > 0} s^{k/2} \| \partial_x^{k-1} \psi_x \|_{L^\infty_x(\R^2)} &\lesssim_{E,k} 1 
\label{linfty-const-ord}
\end{align}
for all $k \geq 1$.  Similar estimates hold if one replaces $\partial_x \psi_x$ with $\psi_s$, $\partial_x^2$ with $\partial_s$, and/or $\partial_x$ with $D_x$.
\end{corollary}

\begin{proof} Write $\partial_x = D_x - A_x$ and apply Proposition \ref{corpar}, Lemma \ref{infty-lemma}, and \eqref{ax-infty}.
\end{proof}

\subsection{The covariant heat equation}

From \eqref{psix-heat}, \eqref{psis-heat}, \eqref{dst} we see that the derivative fields $\psi_x, \psi_s, \psi_t$ in the caloric gauge are all solutions to the covariant heat equation
\begin{equation}\label{uheat}
\partial_s u = D_i D_i u - (u \wedge \psi_i) \psi_i.
\end{equation}

We will repeatedly use the following parabolic estimates for such solutions.

\begin{lemma}[Covariant parabolic regularity]\label{uheat-est}   Let the assumptions and notation be as in Proposition \ref{abound}.  Let $u(0): \R^2 \to \R^m$ be Schwartz.  Then there exists a unique smooth solution $u: \R^+ \times \R^2 \to \R^m$ to \eqref{uheat} with initial data $u(0)$ such that $u(s)$ is Schwartz for each $s$.  Furthermore we have the pointwise estimate
\begin{equation}\label{upoint}
|u(s)| \leq e^{s\Delta} |u(0)|
\end{equation}
the energy inequality
\begin{equation}\label{energy-ineq}
\partial_s \int_{\R^2} |u(s,x)|^2\ dx \leq - 2 \int_{\R^2} |D_x u(s,x)|^2\ dx
\end{equation}
and the parabolic estimates
\begin{align}
\sup_{s > 0} s^{k/2} \| \partial_x^k u(s) \|_{L^2_x(\R^2)} \lesssim_{E,k} \|u(0)\|_{L^2_x(\R^2)} \label{u-l2-fixed}\\
\sup_{s > 0} s^{(k+1)/2} \| \partial_x^k u(s) \|_{L^\infty_x(\R^2)} \lesssim_{E,k} \|u(0)\|_{L^2_x(\R^2)} \label{u-linfty-fixed}\\
\int_0^\infty s^k \| \partial_x^{k+1} u(s) \|_{L^2_x(\R^2)}^2\ ds \lesssim_{E,k} \|u(0)\|_{L^2_x(\R^2)}^2 \label{u-l2-integ}\\
\int_0^\infty s^k \| \partial_x^k u(s) \|_{L^\infty_x(\R^2)}^2\ ds \lesssim_{E,k} \|u(0)\|_{L^2_x(\R^2)}^2 \label{u-linfty-integ}
\end{align}
for all $k \geq 0$.  We also have the variant estimate
\begin{equation}\label{ulp}
\int_0^\infty s^{-2/p} \| u(s) \|_{L^p_x(\R^2)}^2\ ds \lesssim_p \|u(0)\|_{L^2_x(\R^2)}^2 
\end{equation}
for all $2 < p \leq \infty$.
\end{lemma}

\begin{proof}  This will largely be a reprise of Proposition \ref{corpar}, though the bounds already obtained on $\psi_x$ and its derivatives will make our task slightly easier.

From \eqref{uheat} and Lemma \ref{dilemma} we have
$$ \partial_s |u| \leq \Delta |u|$$
and so \eqref{upoint} follows from Corollary \ref{compar}.  Similarly, from \eqref{uheat} we have
$$ \partial_s |u|^2 = \Delta |u|^2 - 2 |D_x u|^2 - |u \wedge \psi_x|^2$$
from which \eqref{energy-ineq} is immediate.

We now prove \eqref{u-l2-fixed}, \eqref{u-l2-integ}.  From \eqref{ax-infty} it suffices to establish the covariant version of these estimates, in which $\partial_x$ is relpaced by $D_x$.  We do this by induction on $k$.  The case $k=0$ follows from \eqref{energy-ineq} and the fundamental theorem of calculus, so assume that $k \geq 1$ and the claim has already been proven for smaller values of $k$.  We now suppress dependence of constants on $E, k$. 
Repeated application of the covariant Leibniz rule and \eqref{curv-frame} gives the equation
\begin{equation}\label{dxu}
 \partial_s D_x^k u = D_j D_j D_x^k u + \sum_{k_1+k_2+k_3=k-1} O( |D_x^{k_1} \psi_x| |D_x^{k_2} \psi_x| |D_x^{k_3} u| )
 \end{equation}
and hence by Lemma \ref{dilemma}
$$ \partial_s |D_x^k u| \leq \Delta |D_x^k u| + \sum_{k_1+k_2+k_3=k} O( |D_x^{k_1} \psi_x| |D_x^{k_2} \psi_x| |D_x^{k_3} u| ).$$
Using \eqref{linfty-const} we conclude
$$ \partial_s |D_x^k u| \leq \Delta |D_x^k u| + \sum_{k_3=0}^{k} O( s^{(k_3-k-2)/2} |D_x^{k_3} u| ).$$
Now we prove (the covariant version of) \eqref{u-l2-fixed} for some $s > 0$.  From (the covariant version of) the inductive hypothesis \eqref{u-l2-integ} and the pigeonhole principle we can find $s/2 \leq s_0 \leq s$ such that
$$ \| D_x^k u(s_0) \|_{L^2_x(\R^2)} \lesssim s^{-k/2}$$
so hence by Duhamel's formula \eqref{duh} and \eqref{heat-lp} we have
$$ \| D_x^k u(s) \|_{L^2_x(\R^2)} \lesssim s^{-k/2} + 
\sum_{k_3=0}^{k} s^{(k_3-k-2)/2} \int_{s_0}^s \|D_x^{k_3} u(s')\|_{L^2_x(\R^2)}\ ds'.$$
Using (the covariant version of) the inductive hypothesis \eqref{u-l2-fixed} (for $k_3 < k$) and \eqref{u-l2-integ} (for $k_3 = k$) we obtain 
$$ \| D_x^k u(s) \|_{L^2_x(\R^2)} \lesssim s^{-k/2}$$
which is the covariant form of \eqref{u-l2-fixed}.

Finally, to show \eqref{u-l2-integ}, we return to \eqref{dxu} and conclude that
$$ \partial_s |D_x^k u|^2 \leq \Delta |D_x^k u|^2 - 2 |D_x^{k+1} u|^2
+ \sum_{k_1+k_2+k_3=k-1} O( |D_x^{k_1} \psi_x| |D_x^{k_2} \psi_x| |D_x^{k_3} u| |D_x^k u|)$$
and hence on integrating and using the induction hypothesis
$$ \partial_s \int_{\R^2} |D_x^k u|^2 \leq - 2 \int_{\R^2} |D_x^{k+1} u|^2
+ \sum_{k_3=0}^{k} O\left( s^{(k_3-k-2)/2} \int_{\R^2} |D_x^{k_3} u| |D_x^k u|\right).$$
Writing $E_k(s) := \int_{\R^2} |D_x^k u|^2$, we then obtain \eqref{2m} from the arithmetic mean-geometric mean inequality, and so arguing as in Proposition \ref{corpar} we establish \eqref{u-l2-integ} as required.

The estimate \eqref{u-linfty-fixed} follows from \eqref{u-l2-fixed} and the Gagliardo-Nirenberg inequality \eqref{gag-2}, and \eqref{u-linfty-integ} similarly follows from \eqref{u-l2-integ}, \eqref{gag-2}, and Cauchy-Schwarz. Finally, the estimate \eqref{ulp} follows from \eqref{upoint} and Lemma \eqref{ubang}.
\end{proof}

\section{Proof of Theorem \ref{energy-claim}}\label{energy-sec}

We are now ready to prove Theorem \ref{energy-claim}.

\subsection{Construction of the energy space}

Recall that the space ${\mathcal L}$ in \eqref{ldef} is the Hilbert space of pairs $(\psi_s, \phi_1)$ of measurable functions $\psi_s: \R^+ \times \R^2 \to \R^m$ and $\phi_1: \R^2 \to \R^m$ whose norm
\begin{equation}\label{l-def}
\| (\psi_s, \phi_1) \|_{\mathcal L}^2 := \int_0^\infty \int_{\R^2} |\psi_s|^2\ dx ds + \frac{1}{2} \int_{\R^2} |\phi_1|^2\ dx
\end{equation}
is finite.  As usual we identify functions which agree almost everywhere.  The orthogonal group $SO(m)$ acts on $\R^m$ and thus acts unitarily on ${\mathcal L}$ in the obvious manner, with each rotation matrix $U \in SO(m)$ sending $(\psi_s, \phi_1)$ to $(U \circ \psi_s, U \circ \phi_1)$. This is clearly an isometry.  If we then quotient out by this compact group we obtain a metric space $SO(m)\backslash {\mathcal L}$.

Now let $(\phi_0,\phi_1) \in \S$ be classical initial data, with $\phi_0(\infty)$ equal to $\phi_0(\infty)$ at infinity.  We extend $\phi_0$ to $\R^+ \times \R^2$ by the heat flow, and use Theorem \ref{caloric-thm} to pick a caloric gauge $e$ for $\phi_0$ which equals some arbitrary frame $e(\infty) \in \Frame(T_{\phi_0(\infty)} \H)$ at infinity, giving rise to the differentiated fields $\psi_x, \psi_s$ and connections $A_x$ in the usual manner.  We then define the \emph{nonlinear Littlewood-Paley resolution map} $\iota: \S \to SO(m) \backslash {\mathcal L}$ by the formula
$$ \iota( \phi_0, \phi_1 ) := SO(m) (\psi_s, e^* \phi_1).$$
One easily verifies from \eqref{l2-integ} that $(\psi_s,e^* \phi_1)$ does indeed lie in ${\mathcal L}$.  Note that rotating the frame $e(\infty)$ rotates the fields $(\psi_s, e^* \phi_1)$ by an element of $SO(m)$, and so $\iota$ is well-defined.

We then define the energy space $\Energy$ to be the closure of $\iota(\S)$ in $SO(m) \backslash {\mathcal L}$.

\subsection{Easy verifications}

We can now quickly establish all the claims in Theorem \ref{energy-claim} except for property (iv), which is more delicate and will be treated later.

Property (i) of Theorem \ref{energy-claim} is immmediate by construction.  The first part of property (ii) follows immediately from Remark \ref{rotor}.  To prove the converse claim in (ii), suppose that we had two classical data $\Phi = (\phi_0,\phi_1)$ and $\tilde \Phi = (\tilde \phi_0,\tilde \phi_1)$ which had the same image under $\iota$, thus we have (after applying a rotation in $SO(m)$ if necessary) caloric gauges $e, \tilde e$ with respect to which $\psi_s = \tilde \psi_s$ and $\tilde e^* \tilde \phi^1 = e^* \phi^1$.  By applying a rotation in $SO(m,1)$ we may take $\phi_0(\infty) = \tilde \phi_0(\infty)$ and $e(\infty) = \tilde e(\infty)$.  

From \eqref{psis-def} we have
$$ \partial_s \phi_0 = e \psi_s$$
while from \eqref{phin} we have
$$ (\phi_0^* \nabla)_s e = 0.$$
This gives us a system of ODE with which to recover $\phi_0, e$ from the boundary data $\phi_0(\infty)$, $\tilde \phi_0(\infty)$.  Since $\tilde \phi_0, \tilde e$ has the same data, we thus see from the Picard uniqueness theorem that $\phi_0 = \tilde \phi_0$ and $e = \tilde e$. (Here we need the qualitative decay of $\psi_s$ and $\tilde \psi_s$ from \eqref{phi-2}.)  Since $\tilde e^* \tilde \phi^1 = e^* \phi^1$ we thus conclude $\tilde \phi^1 = \phi^1$, and so $\Phi = \tilde \Phi$ as required.

It is easy to see that the actions \eqref{space-trans-data}, \eqref{time-reverse-data}, \eqref{scaling-data} on $\S$ intertwine with analogous actions on ${\mathcal L}$ or $SO(m) \backslash {\mathcal L}$, which one easily verifies to be isometric, and property (iii) then follows from property (i).

To verify property (v), we use

\begin{lemma}[Energy identity]\label{energy-ident}  For any $\Phi \in \S$ we have
$$ \E( \Phi ) = d_{SO(m) \backslash {\mathcal L}}( \iota(\Phi), 0 )^2.$$
\end{lemma}

\begin{proof} In view of \eqref{energy-def}, \eqref{l-def}, and the unitary nature of the $SO(m)$ action, and it suffices to show that
$$ \frac{1}{2} \int_{\R^2} |\nabla \phi_0(0,\cdot)|_{\phi_0^* h}^2\ dx = \int_0^\infty \int_{\R^2} |\psi_s|^2\ dx ds.$$
The left-hand side can be expressed as
$$ \frac{1}{2} \int_{\R^2} |\psi_x(0,\cdot)|^2\ dx.$$
But from \eqref{psi-evolve} we have
$$ \partial_s |\psi_x|^2 = 2 \psi_i \cdot D_i \psi_s$$
and so by integration by parts
$$ \partial_s \frac{1}{2} \int_{\R^2} |\psi_x(s,\cdot)|^2\ dx = - \int_{\R^2} D_i \psi_i \cdot \psi_s\ dx.$$
By \eqref{ps-frame}, the right-hand side is $\int_{\R^2} |\psi_s|^2\ dx$.  Finally, using \eqref{joke} we see that $\int_{\R^2} |\psi_x(s,\cdot)|^2\ dx$ goes to zero as $s \to \infty$, and the claim follows from the fundamental theorem of calculus.
\end{proof}

If $\Phi \in \Energy$ is such that $\E(\Phi) = 0$, then by Property (i) we can find a sequence $\Phi^{(n)}$ of classical data such that $\iota(\Phi^{(n)})$ converges to $\Phi$.  Assuming Property (iv) for now, this implies that $\E(\Phi^{(n)})$ converges to zero, and thus by Lemma \ref{energy-ident} $\Phi^{(n)}$ converges to the constant data, and Property (v) follows (conditionally on Property (iv)).

\subsection{Continuity of the Gram matrix}

It remains to establish Property (iv), namely that the Gram matrix map $\Gram: \S \to L^1(\R^2 \to \Sym^2(\R^{1+2}))$ extends continuously to $\Energy$.  It suffices to show that if $\Phi^{(n)} = (\phi_0^{(n)}, \phi_1^{(n)}) \in \S$ is a sequence of classical data such that $\iota( \Phi^{(n)} )$ a Cauchy sequence in $SO(m) \backslash {\mathcal L}$, then $\Gram(\Phi^{(n)})$ is a Cauchy sequence in $L^1$.  We may then use Theorem \ref{caloric-thm} to find caloric gauges $e^{(n)}$ (with the attendant fields $\psi_x^{(n)}, \psi_s^{(n)}, A_x^{(n)}$) such that $(\psi_s^{(n)}, \psi^{(n)}_0(0,\cdot))$ is Cauchy in ${\mathcal L}$, where we adopt the convention $\psi^{(n)}_0(0,\cdot) := (e^{(n)})^* \phi_1^{(n)}$.  (The data $\phi_0^{(n)}(\infty)$, $e^{(n)}(\infty)$ can depend on $n$, but this will not concern us.)

Observe that the Gram matrix can be expressed as
$$ \Gram(\Phi^{(n)})_{\alpha \beta} = \psi^{(n)}_\alpha(0,\cdot) \cdot \psi^{(n)}_\beta(0,\cdot).$$
By Cauchy-Schwarz, it thus suffices to show 

\begin{proposition}\label{moveit}  If $(\psi_s^{(n)}, \psi^{(n)}_0(0,\cdot))$ is Cauchy in ${\mathcal L}$, then
$\psi^{(n)}_\alpha(0,\cdot)$ is Cauchy in $L^2_x(\R^2 \to \R^m)$ for each $\alpha=0,1,2$.
\end{proposition}

We now prove the proposition.  For $\alpha = 0$, the claim follows immediately from \eqref{l-def} and the Cauchy nature of $\psi^{(n)}_\alpha(0,\cdot)$, so it suffices to show that $\psi^{(n)}_x(0,\cdot)$ is Cauchy in $L^2$, i.e. we need to show that
\begin{equation}\label{psio}
 \| \psi^{(n)-(n')}_x(0,\cdot) \|_{L^2_x(\R^2)} = o_{n,n' \to \infty}(1)
 \end{equation}
where $o_{n,n' \to \infty}(1)$ denotes an expression which goes to zero as $n,n'$ jointly go to infinity, and we adopt the convention that $\psi^{(n)-(n')}_x$ is short for $\psi^{(n)}_x - \psi^{(n')}_x$ (and similarly for other fields).  On the other hand, from \eqref{l-def} we already know that
\begin{equation}\label{ncauchy}
\int_0^\infty \|\psi_s^{(n)-(n')}(s)  \|_{L^2_x(\R^2)}^2\ ds = o_{n,n' \to \infty}(1).
\end{equation}

Also, as Cauchy sequences are bounded, we see from Lemma \ref{energy-ident} that the $\Phi^{(n)}$ are uniformly bounded.  Thus there exists a finite $E$ such that \eqref{en} holds uniformly in $n$.  We fix this $E$ and allow all implied constants to depend on $E$.  In particular, from \eqref{l2-integ-ord} and \eqref{ax-infty} we have
$$ 
\int_0^\infty s^k \|\partial_x^k \psi_s^{(n)}(s) \|_{L^2_x(\R^2)}^2\ ds \lesssim_k 1$$
for all $k \geq 0$, and in particular
$$ 
\int_0^\infty s^k \|\partial_x^k \psi_s^{(n)-(n')}(s) \|_{L^2_x(\R^2)}^2\ ds \lesssim_k 1.$$
Using the Gagliardo-Nirenberg inequality \eqref{gag-1}, \eqref{ncauchy}, and Cauchy-Schwarz, we thus conclude that
\begin{equation}\label{soak}
\int_0^\infty s^k \|\partial_x^k \psi_s^{(n)-(n')}(s) \|_{L^2_x(\R^2)}^2\ ds = o_{n,n' \to \infty;k}(1)
\end{equation}
for all $k \geq 0$, where the $k$ subscript on the right-hand side means that we allow the rate of decay in the $o()$ notation to depend on $k$.  Another application of Gagliardo-Nirenberg \eqref{gag-2} then gives
\begin{equation}\label{soso}
\int_0^\infty s^{k+1} \|\partial_x^k \psi_s^{(n)-(n')}(s) \|_{L^\infty_x(\R^2)}^2\ ds = o_{n,n' \to \infty;k}(1)
\end{equation}
for all $k \geq 0$.

We can convert these integrated convergence estimates into fixed-time convergence estimates by exploiting some regularity in time as follows.  From \eqref{l2-const-ord}, \eqref{ax-infty}, \eqref{ps-frame} we obtain the estimate
$$ \| \partial_s \partial_x^k \psi_s^{(n)}(s) \|_{L^2_x(\R^2)} \lesssim_k s^{-(k+3)/2} $$
for all $s > 0$ and $k \geq 0$, and thus by the triangle inequality
$$ \| \partial_s \partial_x^k \psi_s^{(n)-(n')}(s)\|_{L^2_x(\R^2)} \lesssim_k s^{-(k+3)/2}.$$
Applying  the fundamental theorem of calculus and Minkowski's inequality we have
$$ \| \partial_x^k \psi_s^{(n)-(n')}(s) \|_{L^2_x(\R^2)} = \| \partial_x^k \psi_s^{(n)-(n')}(s') \|_{L^2_x(\R^2)} + O_k(s^{-(k+3)/2} |s-s'|)$$
whenever $s \leq s' \leq 2s$.  Averaging this, we conclude that
$$\| \partial_x^k \psi_s^{(n)-(n')}(s)\|_{L^2_x(\R^2)} = \frac{1}{\eps s} \int_s^{s+\eps s} \| \partial_x^k \psi_s^{(n)-(n')}(s') \|_{L^2_x(\R^2)}\ ds' + O_k(\eps s^{-(k+1)/2})$$
for any $0 < \eps < 1$.  Applying Cauchy-Schwarz and \eqref{soak} we conclude that
$$\| \partial_x^k \psi_s^{(n)-(n')}(s) \|_{L^2_x(\R^2)} \lesssim_k
(\frac{1}{\eps} o_{n,n' \to \infty;k}(1) + \eps ) s^{-(k+1)/2}.$$
Setting $\eps$ to decay to zero sufficiently slowly in $n,n'$, we conclude that
$$ s^{(k+1)/2} \| \partial_x^k \psi_s^{(n)-(n')}(s)\|_{L^2_x(\R^2)} = o_{n,n' \to \infty;k}(1)$$
and then by Gagliardo-Nirenberg \eqref{gag-2} as before
\begin{equation}\label{nono}
 s^{(k+2)/2} \| \partial_x^k \psi_s^{(n)-(n')}(s)\|_{L^\infty_x(\R^2)} = o_{n,n' \to \infty;k}(1).
 \end{equation}
Note that the decay on the right-hand side is uniform in $s$.

To apply all these estimates to establish \eqref{soak}, we must express $\psi^{(n)-(n')}_x$ in terms of $\psi^{(n)-(n')}_s$.  It is convenient to introduce the scalar functions
\begin{align*}
f(s) &:= s \| \psi^{(n)}_s(s) \|_{L^\infty_x(\R^2)} + s \| \psi^{(n')}_s(s) \|_{L^\infty_x(\R^2)} \\
&\quad +
s^{1/2} \| \psi^{(n)}_x(s) \|_{L^\infty_x(\R^2)} + s^{1/2} \| \psi^{(n')}_x(s) \|_{L^\infty_x(\R^2)} \\
&\quad +
s^{1/2} \| A^{(n)}_x(s) \|_{L^\infty_x(\R^2)} + s^{1/2} \| A^{(n')}_x(s) \|_{L^\infty_x(\R^2)}
\end{align*}
and
$$
 g(s) := s \| \psi^{(n)-(n')}_s \|_{L^\infty_x(\R^2)} +s^{3/2} \| \partial_x \psi^{(n)-(n')}_s \|_{L^\infty_x(\R^2)};$$
observe from Proposition \ref{abound} and Corollary \ref{corbound} that
\begin{equation}\label{sups}
\sup_{s > 0} f(s) + \int_0^\infty \frac{f(s)^2}{s}\ ds = O(1).
\end{equation}
while from \eqref{soso}, \eqref{nono} we have
\begin{equation}\label{sugs}
\sup_{s > 0} g(s) + \int_0^\infty \frac{g(s)^2}{s}\ ds = o_{n,n' \to \infty}(1).
\end{equation}

From \eqref{psi-evolve} we have
\begin{equation}\label{psn}
\partial_s \psi^{(n)-(n')}_x = \partial_x \psi^{(n)-(n')}_s +
O( |A^{(n')}_x| |\psi^{(n)-(n')}_s| + |A^{(n)-(n')}_x| |\psi^{(n)}_s|)
\end{equation}
and thus
\begin{align*}
|\partial_s \psi^{(n)-(n')}_x|
&\lesssim |\partial_x \psi^{(n)-(n')}_s| + |A^{(n')}_x| |\psi^{(n)-(n')}_s|
+ |A^{(n)-(n')}_x| |\psi^{(n)}_s|\\
&\lesssim s^{-3/2} g(s) + s^{-3/2} f(s) g(s) + s^{-1} f(s) |A^{(n)-(n')}_x|\\
&\lesssim s^{-3/2} g(s) + s^{-1} f(s) |A^{(n)-(n')}_x|.
\end{align*}
Similarly, from \eqref{sax} we have
\begin{align*}
|\partial_s A^{(n)-(n')}_x| &\lesssim |\psi^{(n')}_x| |\psi^{(n)-(n')}_s| + |\psi^{(n)-(n')}_x| |\psi^{(n)}_s| \\
&\lesssim s^{-3/2} f(s) g(s) + s^{-1} f(s) |\psi^{(n)-(n')}_x|.
\end{align*}
If we thus introduce another scalar function
$$ h(s) := \|\psi^{(n)-(n')}_x\|_{L^\infty_x(\R^2)} + \|A^{(n)-(n')}_x\|_{L^\infty_x(\R^2)}$$
we thus have
$$ |\partial_s h(s)| \lesssim s^{-3/2} g(s) + s^{-1} f(s) h(s)$$
(interpreted in a distributional sense).  By Gronwall's inequality we thus conclude that
$$
h(s_0) \lesssim \exp\left( \int_{s_0}^{s_1} \frac{f(s)}{s}\ ds \right) h(s_1)
+ \int_{s_0}^{s_1} \exp\left( \int_{s_0}^{s'} \frac{f(s)}{s}\ ds \right)  \frac{g(s')}{(s')^{3/2}}\ ds'$$
for any $0 < s_0 < s_1$.
From \eqref{sups} and Cauchy-Schwarz we have
$$ \int_{s_0}^{s'} \frac{f(s)}{s}\ ds \lesssim \log^{1/2}(s'/s_0),$$
while from Theorem \ref{caloric-thm} we have $h(s_1) = O_{\phi^{(n)}, \phi^{(n')}}( s_1^{-1} )$.  We can thus take limits as $s_1 \to \infty$ and conclude that\footnote{The exponential of the square root of the logarithm also appears in \cite[Section 4, Step 2(d)]{tao:wavemap2}, for much the same reason as it does here.  The key point is that this quantity grows slower than any polynomial.}
$$ h(s_0) \lesssim \int_{s_0}^\infty \exp( O( \log^{1/2}(s'/s_0) ) ) \frac{g(s')}{(s')^{3/2}}\ ds'.$$
Applying \eqref{sugs} and Young's inequality for multiplicative convolution (with Haar measure $ds/s$), we conclude that 
$h(s_0) = s_0^{-1/2} o_{n,n' \to \infty}(1)$ for all $s_0 > 0$ and that
$$ \sup_{s > 0} s^{1/2} h(s) + \int_0^\infty h(s)^2\ ds = o_{n,n' \to \infty}(1).$$
In particular we have
\begin{equation}\label{axis}
 \int_0^\infty \| A^{(n)-(n')}_x(s) \|_{L^\infty_x(\R^2)}^2\ ds = o_{n,n' \to \infty}(1).
\end{equation}
Meanwhile, from \eqref{joke} we know that $\psi^{(n)-(n')}_x(s)$ converges to zero in $L^2_x(\R^2)$ as $s \to \infty$.  Thus, from \eqref{psn}, we have
\begin{align*}
\| \psi^{(n)-(n')}_x(0) \|_{L^2_x(\R^2)}
&\lesssim \limsup_{S \to \infty} \| \int_0^S \partial_x \psi^{(n)-(n')}_s(s)\ ds \|_{L^2_x(\R^2)}\\
&\quad
+ \int_0^\infty \| A^{(n')}_x(s) \|_{L^\infty_x(\R^2)} \| \psi^{(n)-(n')}_s(s) \|_{L^2_x(\R^2)}\\
&\quad
+ \| A^{(n)-(n')}_x(s) \|_{L^\infty_x(\R^2)} \| \psi^{(n)}_s(s) \|_{L^2_x(\R^2)}\ ds.
\end{align*}
Applying \eqref{aint-infty}, \eqref{ncauchy}, \eqref{axis}, \eqref{l2-integ} and Cauchy-Schwarz we thus have
$$ \| \psi^{(n)-(n')}_x(0) \|_{L^2_x(\R^2)}
\lesssim \limsup_{S \to \infty} \| \int_0^S \partial_x \psi^{(n)-(n')}_s(s)\ ds \|_{L^2_x(\R^2)} + o_{n,n' \to \infty}(1)$$
and so to show \eqref{psio}, it suffices to show that
$$ \| \int_0^S \partial_x \psi^{(n)-(n')}_s(s)\ ds \|_{L^2_x(\R^2)}^2 = o_{n,n' \to \infty}(1)$$
uniformly in $S$.  But the left-hand side can be expanded using symmetry and integration by parts as
$$ - 2 \int_{0 \leq s_1 \leq s_2 \leq S} \int_{\R^2} \psi^{(n)-(n')}_s(s_1) \partial_x^2 \psi^{(n)-(n')}_s(s_2)\ ds_1 ds_2.$$
Introducing the function
$$ k(s) := \| \psi^{(n)-(n')}_s \|_{L^2_x(\R^2)} + s \| \partial_x^2 \psi^{(n)-(n')}_s \|_{L^2_x(\R^2)}$$
one can estimate this quantity using Cauchy-Schwarz by
$$ O\left( \int_{0 \leq s_1 \leq s_2 \leq \infty} \frac{1}{s_2} k(s_1) k(s_2)\ ds_1 ds_2 \right)$$
which by Schur's test can be bounded by $O( \int_0^\infty k(s)^2\ ds)^{1/2}$.  But this is $o_{n,n' \to \infty}(1)$ as desired, thanks to \eqref{soak}.  This yields \eqref{psio} as desired, and the proof of Proposition \ref{moveit} and hence Theorem \ref{energy-claim} is now complete.

\begin{remark}  The above analysis in fact shows that $\psi^{(n)}_x$, $A^{(n)}_x$, and $\psi^{(n)}_s$ all converge in smooth topologies to some limits $\psi^{(\infty)}_x$, $A^{(\infty)}_x$, and $\psi^{(\infty)}_s$, which then behave like ``virtual heat flows'' in the sense that identities such as \eqref{zerotor-frame}, \eqref{curv-frame}, \eqref{ps-frame}
continue to hold in the limit.  Because of this, one can meaningfully solve the heat flow equation in the energy class $\Energy$ (up to the usual $SO(m)$ ambiguity), and identify members of this class (up to this ambiguity) with heat flows viewed in an orthonormal frame on $(0,+\infty) \times \R^2$ whose asymptotic energy $\lim_{s \to 0^+} \int_{\R^2} |\psi_x|^2$ is finite (cf. the definition of the classical Hardy space ${\mathcal H}^2(\R)$ via Poisson extension to the upper half-plane). It may in fact be possible to extract a meaningful limiting value of $\phi$ and $e$ at $s=0$ (subject, of course, to the $SO(m,1)$ ambiguity), thus providing an alternate way to interpret the energy space $\Energy$ (cf. the ``Lebesgue perspective'' from Remark \ref{perspectives}).  We will not pursue these matters.
\end{remark}

\section{Proof of Theorem \ref{nondeg}}\label{nondeg-sec}

We are now ready to prove Theorem \ref{nondeg}.  By rotation symmetry we can fix $v = (1,0)$; thus we need to show that any $(\phi_0,\phi_1) \in \Energy$ with $|\phi_1 + \partial_1 \phi_0|_{\phi^* h}^2 \equiv |\partial_2 \phi_0|_{\phi^* h}^2 \equiv 0$ has zero energy.  In view of Theorem \ref{energy-claim} and a limiting argument, it suffices to show

\begin{proposition}  Let $(\phi^{(n)}_0, \phi^{(n)}_1) \in \S$ be a sequence of classical data with $\iota(\phi^{(n)}_0, \phi^{(n)}_1)$ convergent in $\Energy$, and such that $|\phi^{(n)}_1 + \partial_1 \phi^{(n)}_0|_{(\phi^{(n)})^* h}^2 \equiv |\partial_2 \phi^{(n)}_0|_{(\phi^{(n)})^* h}^2$ converge to zero in $L^1_x(\R^2)$.  Then $\E(\phi^{(n)}_0, \phi^{(n)}_1)$ converges to zero.
\end{proposition}

\begin{remark}  It is essential here that $\iota(\phi^{(n)}_0, \phi^{(n)}_1)$ is convergent (or at least precompact) and not merely bounded in the energy space.  To see this, let us take the Euclidean model, in which $\phi^{(n)}_0, \phi^{(n)}_1$ take values in $\R$ rather than $\H$.  If one sets $\phi^{(n)}_0(x_1,x_2) := \frac{1}{\sqrt{n}} \eta( x_1, \frac{x_2}{n} )$ and $\phi^{(n)}_1 := - \partial_1 \phi^{(n)}_0$, one easily verifies that $|\phi^{(n)}_1 + \partial_1 \phi^{(n)}_0|^2 \equiv |\partial_2 \phi^{(n)}_0|^2$ converges to zero in $L^1_x(\R^2)$, but that the energies $\E(\phi^{(n)}_0, \phi^{(n)}_1)$ are bounded.
\end{remark}

We now prove the proposition.  We use Theorem \ref{caloric-thm} to place caloric gauges $e^{(n)}$ on each $\phi^{(n)}$, with attendant derivative fields $\psi^{(n)}_x, \psi^{(n)}_s$ and connection fields $A^{(n)}_x$; by rotating these gauges we may then assume that $(\psi^{(n)}_s, (e^{(n)})^* \phi^{(n)}_1)$ is convergent in $\L$ to some limit $(\psi^{(\infty)}_s, \psi^{(\infty)}_t)$, thus by \eqref{l-def}
\begin{equation}\label{conv}
\int_0^\infty\int_{\R^2} |\psi^{(n)}_s - \psi^{(\infty)}_s|^2\ dx ds + \| (e^{(n)})^* \phi^{(n)}_1 - \psi^{(\infty)}_t \|_{L^2_x(\R^2)} = o_{n \to \infty}(1).
\end{equation}
Our task is to show that $\psi^{(\infty)}_s$ and $\psi^{(\infty)}_t$ vanish.  By hypothesis, we have
\begin{equation}\label{epsin}
\| (e^{(n)})^* \phi^{(n)}_1 + \psi^{(n)}_1(0) \|_{L^2_x(\R^2)} = o_{n \to \infty}(1)
\end{equation}
and
\begin{equation}\label{psi2}
 \| \psi^{(n)}_2(0) \|_{L^2_x(\R^2)} = o_{n \to \infty}(1).
 \end{equation}
It will then suffice to show that $\psi^{(\infty)}_s$ vanishes, since by the energy identity (Lemma \ref{energy-ident}) this shows that $\psi^{(n)}_1(0)$ converges to zero in $L^2_x(\R^2)$, which by \eqref{epsin}, \eqref{conv} yields that $\psi^{(\infty)}_t$ vanishes as required.

Clearly, it will suffice to show that
$$ \int_{1/S}^S \int_{|x| \leq S} |\psi^{(\infty)}_s|^2\ dx ds = 0 $$
for any $S \geq 1$.

Fix $S$.  From \eqref{conv} it suffices to show that
$$ \int_{1/S}^S \int_{|x| \leq S} |\psi^{(n)}_s|^2\ dx ds \lesssim o_{n \to \infty;S}(1)$$
for all $n$, which in turn will follow if we can show that
\begin{equation}\label{jon}
 \int_{|x| \leq S} |\psi^{(n)}_s(s)|^2\ dx ds \lesssim o_{n \to \infty;S}(1)
 \end{equation}
for all $n$ and all $1/S \leq s \leq S$.

Fix $n,s$.  From \eqref{psix-heat}, Lemma \ref{uheat-est}, and \eqref{ax-infty} we have
\begin{equation}\label{dx2}
  \| D_x \psi^{(n)}_2(s) \|_{L^2_x(\R^2)} = o_{n \to \infty;S}(1).
\end{equation}
By \eqref{zerotor-frame} this implies in particular that
$$  \| D_2 \psi^{(n)}_1(s) \|_{L^2_x(\R^2)} = o_{n \to \infty;S}(1)$$
and hence by the diamagnetic inequality (Lemma \ref{dilemma})
$$  \| \partial_2 |\psi^{(n)}_1(s)| \|_{L^2_x(\R^2)} = o_{n \to \infty;S}(1).$$
Meanwhile from \eqref{l2-const-ord} we have
$$  \| \psi^{(n)}_1(s) \|_{L^2_x(\R^2)} = O_S(1).$$
Applying the Poincar\'e inequality
$$ \int_{|x_2| \leq 2S} |f(x_1,x_2)|^2\ dx_2 \lesssim_S ( \int_\R |f(x_1,x_2)|^2\ dx_2)^{1/2}
( \int_\R |\partial_2 f(x_1,x_2)|^2\ dx_2)^{1/2}$$
for all $x_1 \in \R$ (cf. \eqref{gag-2}), integrating in $x_1$, and applying Cauchy-Schwarz and Fubini's theorem, we conclude that
\begin{equation}\label{psilocal}
\| \psi^{(n)}_1(s) \|_{L^2_x(|x| \leq 2S)} = o_{n \to \infty;S}(1).
\end{equation}
Meanwhile, from \eqref{l2-const-ord} we have
$$  \| \partial_x \psi^{(n)}_1(s) \|_{L^2_x(|x| \leq 2S)} = O_S(1)$$
and
$$  \| \partial_x^2 \psi^{(n)}_1(s) \|_{L^2_x(|x| \leq 2S)} = O_S(1).$$
Applying Gagliardo-Nirenberg \eqref{gag-1} (applied to a smooth truncation of $\psi^{(n)}_1$ to the ball of radius $2S$), we conclude that
$$  \| \partial_x \psi^{(n)}_1(s) \|_{L^2_x(|x| \leq S)} = o_{n \to \infty;S}(1)$$
and thus by \eqref{ax-infty}, \eqref{psilocal}, \eqref{D-def}
\begin{equation}\label{dx1}
\| D_x \psi^{(n)}_1(s) \|_{L^2_x(|x| \leq S)} = o_{n \to \infty;S}(1).
\end{equation}
From \eqref{dx2}, \eqref{dx1}, \eqref{ps-frame} we obtain \eqref{jon} as required.  This completes the proof of Theorem \ref{nondeg}.

\section{The heat flow applied to wave maps}\label{heatwave-sec}

In the Euclidean setting, applying the linear heat operator $e^{s\Delta}$ to a rough solution $u$ to the free wave equation $\partial^\alpha\partial_\alpha u$ yields a smooth solution to the free wave equation.  This gives one a means to regularise a rough wave into a smooth one.

We would similarly like to apply the harmonic map heat flow to regularise a rough wave map into a smooth wave map.  Unfortunately, the heat flow and the wave map equation do not quite commute, but fortunately the commutator is sufficiently well behaved that the heat flow regularises the rough wave map into a smooth \emph{approximate} wave map.  

More precisely, let $(\phi,I)$ be a classical wave map on a compact interval $I$, and let $\phi: \R^+ \times I \times \R^2 \to \H$ be its dynamic heat flow extension given by Theorem \ref{dynamic-caloric}.  We let $e$ be a caloric gauge for this wave map given by that theorem, thus giving the usual fields $\psi_x, \psi_t, \psi_s, A_x, A_t$.  We define the \emph{wave-tension field}
\begin{equation}\label{wn-def}
w := D^\alpha \psi_\alpha = - D_t \psi_t + \psi_s,
\end{equation}
thus the wave map equation \eqref{cov} asserts that $w$ vanishes when $s=0$:
\begin{equation}\label{wn0}
w(t,0,\cdot) = 0.
\end{equation}

For $s > 0$, we have the following parabolic evolution equation for $w$.

\begin{lemma}[Parabolic evolution of wave-tension field]\label{w-evolve}  For any classical wave map viewed in a caloric gauge, we have
\begin{equation}\label{wneq}
 \partial_s w = D_i D_i w - (w \wedge \psi_i) \psi_i - 2 (\psi_t \wedge \psi_i) D_t \psi_i 
 \end{equation}
and 
\begin{equation}\label{psw}
 \partial_s |w| \leq \Delta |w| + O( |D_x \psi_t| |\psi_x| |\psi_t| )
 \end{equation}
in the sense of distributions.
\end{lemma}

\begin{proof} 
Using \eqref{ass}, \eqref{zerotor-frame}, \eqref{curv-frame}, \eqref{ps-frame} we have
\begin{align*}
\partial_s D_t \psi_t &= D_t D_s \psi_t - (\psi_s \wedge \psi_t) \psi_t \\
&= D_t D_t \psi_s - (D_i \psi_i \wedge \psi_t) \psi_t \\
&= D_t D_t D_i \psi_i - (D_i \psi_i \wedge \psi_t) \psi_t \\
&= D_i D_t D_t \psi_i - D_t ((\psi_t \wedge \psi_i) \psi_i) - (\psi_t \wedge \psi_i) D_t \psi_i
- (D_i \psi_i \wedge \psi_t) \psi_t \\
&= D_i D_i D_t \psi_t - D_i ((\psi_t \wedge \psi_i) \psi_t) - D_t ((\psi_t \wedge \psi_i) \psi_i) \\
&\quad - (\psi_t \wedge \psi_i) D_t \psi_i - (D_i \psi_i \wedge \psi_t) \psi_t  \\
&= D_i D_i D_t \psi_t - (D_i \psi_t \wedge \psi_i) \psi_t - 3 (\psi_t \wedge \psi_i) D_t \psi_i \\
&quad - (\psi_t \wedge D_i \psi_t) \psi_i - (D_t \psi_t \wedge \psi_i) \psi_i \\
&= D_i D_i D_t \psi_t - 2 (\psi_t \wedge \psi_i) D_t \psi_i - (D_t \psi_t \wedge \psi_i) \psi_i. 
\end{align*}
Combining this with \eqref{psis-heat} we obtain \eqref{wneq}.  The inequality \eqref{psw} then follows from Lemma \ref{dilemma}.
\end{proof}

The presence of the forcing term in \eqref{psw} means that $w$ need not vanish for $s > 0$.  However, observe that no time derivatives appear in that forcing term (other than those implicit in the $\psi_t$ field).  Because of this, it is still possible to obtain reasonable estimates on $w$, which in turn let us control time derivatives of the $\psi$ or $A$ fields (thus essentially providing double time derivative control on $\phi$).  The purpose of this section is to record all the necessary bounds that we shall need.

Throughout this section we assume an energy bound
\begin{equation}\label{energy-bound}
\E( \phi ) \leq E
\end{equation}
on the energy of the wave map.  

We begin with some basic fixed-time estimates.

\begin{lemma}[Basic estimates]\label{basic}  Let $(\phi,I)$ be a classical wave map in the caloric gauge satisfying \eqref{energy-bound} for some $E > 0$.  Then for all $t \in I$, $s > 0$, $k \geq 0$, we have
\begin{align*}
\| \partial_x^k \Psi \|_{L^2_x(\R^2)} &\lesssim_{k,E} s^{-k/2} \\
\| \partial_x^k \Psi \|_{L^\infty_x(\R^2)} &\lesssim_{k,E} s^{-(k+1)/2}
\end{align*}
where $\Psi := (\psi_x, \psi_t, A_x, A_t)$.  Similarly if $\partial_x \Psi$ is replaced by $\psi_s$, if $\partial_x^2$ is replaced by $\partial_s$, and/or $\partial_x$ is replaced by $D_x$.
\end{lemma}

\begin{proof} For the $\psi_x$ and $A_x$ components of $\Psi$, this follows from Proposition \ref{abound} and Corollary \ref{corbound}.  For the $\psi_t$ component, this follows from \eqref{dst} and Lemma \ref{uheat-est}.  For the $A_t$ component, we use \eqref{ast} to write
$$ A_t = \int_s^\infty \psi_s \wedge \psi_t(s')\ ds'$$
and the claims then follow from the bounds just established.
\end{proof}

\begin{remark} One also has some integrated estimates for derivatives of $\Psi$ in the heat-temporal variable $s$ analogous to those in Proposition \ref{abound} and Corollary \ref{corbound}, but we will not need them here.
\end{remark}

Now we control the wave-tension field.

\begin{lemma}[Preliminary wave-tension field bound]\label{wprelim}  Let the notation and assumptions be as in Lemma \ref{basic}.  Then for all $t \in I$ and $s \geq 0$ we have
$$ \| w(s,t) \|_{L^1_x(\R^2)} \lesssim_E 1.$$
\end{lemma}

\begin{proof}  From \eqref{wn0}, \eqref{psw}, \eqref{heat-lp} and Duhamel's formula \eqref{duh} one has
$$ \| w(s,t) \|_{L^1_x(\R^2)} \lesssim \int_0^s \| |D_x \psi_t| |\psi_x| |\psi_t| \|_{L^1_x(\R^2)}(s')\ ds';$$
applying H\"older's inequality, it suffices to show that
\begin{align}
\sup_{s>0} s^{1/2} \| D_x \psi_t \|_{L^2_x(\R^2)} &\lesssim_E 1\label{dx-1} \\
\int_0^\infty s^{-1/2} \| \psi_{t,x} \|_{L^4_x(\R^2)}^2(s)\ ds &\lesssim_E 1 \label{dx-3}.
\end{align}
But \eqref{dx-1} follows Lemma \ref{basic}, while \eqref{dx-3} follows from \eqref{psix-heat}, \eqref{dst}, and Lemma \ref{uheat-est}.
\end{proof}

Parabolic regularity then lets us control spatial derivatives of $w$ also:

\begin{lemma}\label{wreg}  Let the notation and assumptions be as in Lemma \ref{basic}.  Then for all $t \in I$, $s > 0$, and $k \geq 0$ we have
\begin{equation}\label{l1x}
\| \partial_x^k w(s,t) \|_{L^1_x(\R^2)} \lesssim_{E,k} s^{-k/2}
\end{equation}
and similarly
$$ \| \partial_x^k w(s,t) \|_{L^2_x(\R^2)} \lesssim_{E,k} s^{-(k+1)/2}$$
and
\begin{equation}\label{linftyx}
 \| \partial_x^k w(s,t) \|_{L^\infty_x(\R^2)} \lesssim_{E,k} s^{-(k+2)/2}
 \end{equation}
for all $k \geq 0$.
\end{lemma}

\begin{proof} It suffices to prove \eqref{l1x}, as the other two estimates then follow by the Gagliardo-Nirenberg inequality \eqref{gag-6}, \eqref{gag-2}.  For brevity we omit the explicit dependence on the wave-temporal variable $t$, and on the parameters $E,k$.

We establish \eqref{l1x} by induction on $k$.  The case $k=0$ is Lemma \ref{wsmall}, so suppose $k \geq 1$ and the claim has already been proven for smaller $k$.  We use \eqref{D-def} to express \eqref{wneq} in the form
\begin{equation}\label{psweq}
 \partial_s w = \Delta w + \partial_i F_i + G
 \end{equation}
where 
\begin{equation}\label{fdef}
F_i := 2 A_i w
\end{equation}
and $G$ takes the schematic form
\begin{equation}\label{gdef}
 G = \bigO( (\partial_x \Psi + \Psi^2) w + \Psi \partial_x \Psi + \Psi^4 )
\end{equation}
where $\Psi$ was as in Lemma \ref{basic}.
From Duhamel's formula \eqref{duh} and \eqref{heat-lp} we have
\begin{align*}
\| \partial_x^k w(s_1) \|_{L^1_x(\R^2)} &\lesssim (s_1-s_0)^{-1/2} \| \partial_x^{k-1} w(s_0) \|_{L^1_x(\R^2)} \\
&\quad + \int_{s_0}^{s_1} (s_1-s)^{-1/2} \| \partial_x^k F_x(s) \|_{L^1_x(\R^2)}\ ds \\
&\quad + \int_{s_0}^{s_1} \|\partial_x^k G(s)\|_{L^1_x(\R^2)}\ ds
\end{align*}
for all $0 < s_0 < s_1$.  A computation using Lemma \ref{basic} and the induction hypothesis reveals that
$$ \| \partial_x^k F_x(s) \|_{L^1_x(\R^2)} \lesssim_k s^{-1/2} \| \partial_x^k w(s) \|_{L^1_x(\R^2)} + s^{-(k+1)/2}$$
and
$$ \| \partial_x^k G(s) \|_{L^1_x(\R^2)} \lesssim_k s^{-1} \| \partial_x^k w(s) \|_{L^1_x(\R^2)} + s^{-(k+2)/2}$$
and thus
$$ \| \partial_x^k w(s_1) \|_{L^1_x(\R^2)} \lesssim_k (s_1-s_0)^{-1/2} s_0^{-k/2} + s_0^{-1/2}
\int_{s_0}^{s_1} (s_1-s)^{-1/2} \| \partial_x^k w(s) \|_{L^1_x(\R^2)}\ ds$$
whenever $0 < s_0 < s_1 < 2s_0$.  This implies that
$$ \sup_{s_0 \leq s \leq (1+c) s_0} (s-s_0)^{1/2}  \| \partial_x^k w(s) \|_{L^1_x(\R^2)} \lesssim_k s_0^{-k/2} + c^{1/2}
\sup_{s_0 \leq s \leq (1+c) s_0} (s-s_0)^{1/2}  \| \partial_x^k w(s) \|_{L^1_x(\R^2)}$$
for any $0 < c < 1$; setting $c$ sufficiently small we establish \eqref{l1x} as required.  (The finiteness of the norms here follows from Theorem \ref{dynamic-caloric}.)
\end{proof}

We can now get some estimates on time derivatives:

\begin{lemma}[Time derivative estimates]\label{time}  Let the notation and assumptions be as in Lemma \ref{basic}.  For all $t \in I$, $s > 0$, $k \geq 0$, we have
$$ \| \partial^k_x \partial_t \Psi(t) \|_{L^2_x(\R^2)} \lesssim_{E,k} s^{-(k+1)/2}$$
and
$$ \| \partial^k_x \partial_t \Psi(t) \|_{L^\infty_x(\R^2)} \lesssim_{E,k} s^{-(k+2)/2}$$
Similarly if $\partial_x \Psi$ is replaced by $\psi_s$, if $\partial_x^2$ is replaced by $\partial_s$, and/or $\partial_x$ is replaced by $D_x$.
\end{lemma}

\begin{proof}  By the Gagliardo-Nireberg inequality \eqref{gag-2} it suffices to verify the $L^2$ estimate.

To deal with the $\psi_x$ component of $\Psi$, we use 
$$ \partial_t \psi_x = \partial_x \psi_t + A_x \psi_t - A_t \psi_x$$
and the claim then follows from Lemma \ref{basic}.  To deal with the $A_x$ component of $\Psi$, we use \eqref{sax} to write
$$ A_x = \int_s^\infty \psi_s \wedge \psi_x(s')\ ds'$$
and the claim then follows from Lemma \ref{basic} and the estimates already established for $\psi_x$.  To deal with the $\psi_t$ component of $\Psi$, we observe from \eqref{wn-def} that
$$ \partial_t \psi_t = \psi_s - w - A_t \psi_t$$
and the claim then follows from Lemma \ref{basic} and Lemma \ref{wreg}.  Finally, to deal with the $A_t$ component of $\Psi$, we use \eqref{dst} to write
$$ A_t = \int_s^\infty \psi_s \wedge \psi_t(s')\ ds'$$
and the claim then follows from Lemma \ref{basic} and the estimates already established for $\psi_x$ and $\psi_t$.
\end{proof}

We have now established boundedness of the first time derivative of $\Psi$ (and thus, implicitly, on the second time derivative of $\phi$).  However for our applications we need to also establish some \emph{uniform continuity} of this time derivative.  The first step in this process is to establish some decay of $w$ in the limit $s \to 0$ (which one would expect thanks to \eqref{wn0}).  In order to apply this to wave maps in the energy class, we now need to deal with convergent sequences of classical wave maps.  The first result is as follows.

\begin{lemma}\label{wsmall}  Let $(\phi^{(n)},I)$ be a sequence of classical wave maps all obeying \eqref{energy-bound} for a uniform $E > 0$, with the associated fields $\psi^{(n)}_x, \psi^{(n)}_s, \psi^{(n)}_t, A^{(n)}_x, A^{(n)}_t, w^{(n)}$, and such that the sequence $\phi^{(n)}$ is uniformly convergent in $\Energy$.  Then for every $\eps > 0$ there exists $S > 0$ such that
$$ \| w^{(n)}(s,t) \|_{L^1_x(\R^2)} \lesssim \eps$$
for all sufficiently large $n$ (depending on $\eps,S$), all $0 \leq s \leq 1/S$, and all $t \in I$.
\end{lemma}

\begin{proof}  
By arguing as in the proof of Lemma \ref{wprelim} it suffices to establish
\begin{align}
\sup_{0 < s \leq 1/S} s^{1/2} \| D_x \psi_t \|_{L^2_x(\R^2)} &\lesssim 1\label{dx-1b} \\
\int_0^{1/S} s^{-1/2} \| \psi_{t,x} \|_{L^4_x(\R^2)}^2(s)\ ds &\lesssim \eps \label{dx-3b}.
\end{align}
The claim \eqref{dx-1b} follows immediately from Lemma \ref{basic}, so we turn to \eqref{dx-3b}.
From Lemma \ref{moveit} we see that the set $\{ |\psi^{(n)}_x(0,t,\cdot)|: t \in I, n \geq 1 \}$ is precompact in $L^2_x(\R^2)$, and thus by Lemma \ref{ubang} and linearity we see that the integrals
$$ \int_0^\infty s^{-1/2} \| e^{s\Delta} |\psi^{(n)}_x(0,t,\cdot)| \|_{L^4_x(\R^2)}^2\ ds$$
for $t \in I$ are uniformly integrable; in particular, for $S$ large enough we have
\begin{equation}\label{init}
 \int_0^{1/S} s^{-1/2} \| e^{s\Delta} |\psi^{(n)}_x(0,t,\cdot)| \|_{L^4_x(\R^2)}^2\ ds \lesssim \eps
\end{equation}
for all $t \in I$ and all $n$.  Using \eqref{psix-compar} we conclude \eqref{dx-3b} for $\psi^{(n)}_x$.  The claim for $\psi^{(n)}_t$ (for which one also has precompactness), from \eqref{l-def}) is similar.
\end{proof}

For our purposes it is crucial that we can go beyond boundedness properties for $w^{(n)}$, and establish an additional uniform continuity property:

\begin{lemma}\label{kappalem}  Let the notation and assumptions be as in Lemma \ref{wsmall}.  Then for every $\eps > 0$ there exists $\kappa > 0$ such that for all sufficiently large $n$, all $t \in I$, and all $s > 0$ we have
$$ \| w^{(n)}(s,t_1) - w^{(n)}(s,t_2) \|_{L^1_x(\R^2)} \lesssim \eps$$
whenever $t_1,t_2 \in I$ are such that $|t_1-t_2| \leq \kappa$.  
\end{lemma}

\begin{proof}  Let $S$ be as in Lemma \ref{wsmall}. For $s \leq 1/S$ the claim follows from that lemma, so we can assume that $s > 1/S$.  For brevity we omit the $n$ superscripts.  We fix $t_1, t_2$ and write $\delta u := u(t_2) - u(t_1)$ for any quantity $u$ depending on time.  We observe the product rule
\begin{equation}\label{prod}
\delta(uv) = (\delta u) v(t_1) + u(t_2) (\delta v).
\end{equation}

Write $W(s) := \|\delta w(1/S) \|_{L^1_x(\R^2)}$.  Our task is to show that $W(s) \lesssim \eps$ for all $s \geq 1/S$.
From Lemma \ref{wsmall} and the triangle inequality we already have
\begin{equation}\label{wis}
W(1/S) \lesssim \eps.
\end{equation}
Also, from \eqref{psweq} we have the heat equation
$$ \partial_s \delta w = \Delta \delta w + \partial_i \delta F_i + \delta G$$
where $F_i, G$ were defined in \eqref{fdef}, \eqref{gdef}.  Thus for any $0 < s_0 < s_1$ we have from Duhamel's formula \eqref{duh} and \eqref{heat-lp} that
$$ W(s_1) \lesssim W(s_0) 
+ \int_{s_0}^{s_1} (s_1-s)^{-1/2} \| \delta F_x(s) \|_{L^1_x(\R^2)} + \|\delta G(s) \|_{L^1_x(\R^2)}\ ds.$$
From \eqref{fdef}, \eqref{gdef}, \eqref{prod}, and H\"older's inequality we have
$$ \| \delta F_x \|_{L^1_x(\R^2)} \lesssim s^{-1/2} f(s) W(s) + s^{-1/2} g(s)$$
and
$$ \| \delta G \|_{L^1_x(\R^2)} \lesssim s^{-1} f(s) W(s) + s^{-1} g(s)$$
where
\begin{equation}\label{fsdef}
\begin{split}
 f(s) &:= s^{1/2} \| A_x(t_*) \|_{L^\infty_x(\R^2)} + s \| \partial A_x(t_*) \|_{L^\infty_x(\R^2)}
+ s \| A_x(t_*) \|_{L^\infty_x(\R^2)}^2 \\
&\quad + s \| \psi_x(t_*) \|_{L^\infty_x(\R^2)}^2
\end{split}
\end{equation}
and
\begin{equation}\label{gsdef}
\begin{split}
 g(s) &:= s^{1/2} \| \delta A_x(t_*) \|_{L^\infty_x(\R^2)} \| w(t_*) \|_{L^1_x(\R^2)} 
+ s \| \partial_x \delta \Psi(t_*) \|_{L^\infty_x(\R^2)} \| w(t_*) \|_{L^1_x(\R^2)} \\
&\quad + s \| \delta \Psi(t_*) \|_{L^\infty_x(\R^2)} \| \Psi(t_*) \|_{L^\infty_x(\R^2)} \| w(t_*) \|_{L^1_x(\R^2)} \\
&\quad 
+ s \| \delta \Psi(t_*) \|_{L^\infty_x(\R^2)} \| \Psi(t_*) \|_{L^2_x(\R^2)} \| \partial_x \Psi(t_*) \|_{L^2_x(\R^2)} \\
&\quad + s \| \Psi(t_*) \|_{L^2_x(\R^2)}^2 \| \partial_x \delta \Psi(t_*) \|_{L^\infty_x(\R^2)} \\
&\quad + s \| \delta \Psi(t_*) \|_{L^\infty_x(\R^2)} \| \Psi(t_*) \|_{L^2_x(\R^2)}^2 \| \Psi(t_*) \|_{L^\infty_x(\R^2)}
\end{split}
\end{equation}
where $t_*$ is summed over $t_1$ and $t_2$ (with multiple occurrences of $t_*$ being summed separately), and $\Psi$ is as in Lemma \ref{basic}.  Thus we have
$$
W(s_1) \lesssim W(s_0) 
+ \int_{s_0}^{s_1} ((s_1-s)^{-1/2} s^{-1/2} + s^{-1}) (f(s) W(s) + g(s))\ ds.$$
We can estimate
$$ \int_{s_0}^{s_1} ((s_1-s)^{-1/2} s^{-1/2} + s^{-1}) f(s)\ ds \lesssim c^{1/2} \|f\|_{L^\infty([s_0,s_1])} + c^{-1/2} \int_{s_0}^{s_1} f(s)\ \frac{ds}{s}$$
for any $c > 0$, and similarly for $g$, and thus
\begin{align*}
\sup_{s \in [s_0,s_1]} W(s) &\lesssim W(s_0) \\
&\quad +
(c^{1/2} \|f\|_{L^\infty([s_0,s_1])} \\
&\quad + c^{-1/2} \int_{s_0}^{s_1} f(s)\ \frac{ds}{s}) \sup_{s \in [s_0,s_1]} W(s) \\
&\quad + \|g\|_{L^\infty([s_0,s_1])} + \int_{s_0}^{s_1} g(s)\ \frac{ds}{s}.
\end{align*}
Applying \eqref{fsdef}, \eqref{ax-infty}, \eqref{aint-infty}, \eqref{linfty-const-ord}, \eqref{linfty-integ-ord} we have
$$ \|f\|_{L^\infty((0,\infty))} + \int_0^\infty \frac{f(s)}{s}\ ds \lesssim 1$$
and thus if we choose $c$ smaller than an absolute constant, and choose an interval $[s_0,s_1]$ such that $\int_{s_0}^{s_1} f(s) \frac{ds}{s} \leq c$, we conclude that
$$ \sup_{s \in [s_0,s_1]} W(s) \lesssim W(s_0)  
+ \|g\|_{L^\infty([s_0,s_1])} + \int_{s_0}^{s_1} g(s)\ \frac{ds}{s}.$$
Dividing $[1/S,T]$ into $O(1/c)$ intervals $[s_0,s_1]$ of the above form for any $T$, and then letting $T \to \infty$, we thus conclude the Gronwall-type inequality
$$ \sup_{s \geq 1/S} W(s) \lesssim W(1/S) +\|g\|_{L^\infty([1/S,\infty))} + \int_{1/S}^{\infty} g(s)\ \frac{ds}{s}.$$
In view of \eqref{wis}, it thus suffices to show that
$$ g(s) \lesssim \kappa / s^{1/2}$$
since the claim then follows by taking $\kappa$ small enough.

Applying Lemma \ref{basic} and Lemma \ref{wsmall} we have
$$ g(s) \lesssim s \| \partial_x \delta \Psi(t_*) \|_{L^\infty_x(\R^2)} +
+ s^{1/2} \| \delta \Psi(t_*) \|_{L^\infty_x(\R^2)}$$
and the claim now follows from the fundamental theorem of calculus and Lemma \ref{time}.
\end{proof}

\section{All weakly harmonic maps are trivial}\label{weakmap}

\subsection{The travelling case}

We shall shortly prove part (i) of Theorem \ref{selfsim}, which roughly speaking asserts that any travelling wave map in the energy class must be trivial; in particular, every stationary wave map in the energy class must be trivial.  Formally, from \eqref{cov} we expect stationary wave maps in the energy class to be equivalent to (weak) harmonic maps in the energy class, although it turns out to be non-trivial to make this statement rigorous due to the extremely low regularity of such maps.  Nevertheless, it is reasonable to expect that in order to establish Theorem \ref{selfsim}(i), one must first establish a result to the effect that all weakly harmonic maps in the energy class are trivial.

It is already clear from Proposition \ref{Decay} that there are no non-trivial \emph{classical} harmonic maps, since the heat flow on such maps is static and thus clearly does not obey the decay estimates in that proposition.  This argument does not directly yield the desired claim.  Nevertheless, it is possible to use the more quantitative analysis of the heat flow in the caloric gauge from Section \ref{littlewood-sec} to establish what we need.  For technical reasons it is convenient to work in the norm $L^1_\loc$ defined by \eqref{morrey}.
The precise statement we will prove is as follows:

\begin{theorem}[Nonlinear Poincar\'e inequality in the plane]\label{nlpoin} Let $\phi: \R^2 \to \H$ be a smooth map such that $\phi - \phi(\infty)$ is rapidly decreasing for some $\phi(\infty) \in \H$.  Suppose we have the energy bound
$$ \int_{\R^2} | \partial_x \phi |_{\phi^* h}^2\ dx \leq E$$
and the local near-harmonicity property
$$ \| |\eta^{ij} (\phi^* \nabla)_i \partial_j \phi|_{\phi^* h} \|_{L^1_\loc(\R^2)} \leq \eps$$
for some $E > 0$ and $0 < \eps < 1$ and some strictly positive definite constant-coefficient matrix $\eta^{ij}$.  Then we have
$$ \| |\partial_i \phi|_{\phi^* h} \|_{L^1_\loc(\R^2)} \lesssim_{E,\eta} \eps^{1/4}.$$
\end{theorem}

\begin{remark} The exponent $1/4$ here is probably non-optimal, but any expression on the right-hand side which decays to zero as $\eps \to 0$ will suffice for our applications.  The presence of the matrix $\eta^{ij}$ is necessary in order to handle travelling wave maps (as opposed to stationary ones), which correspond to Lorentz contracted harmonic maps.
\end{remark}

\begin{proof}  By a linear change of variables we can take $\eta$ to be the identity matrix, thus
$$ \||(\phi^* \nabla)_i \partial_i \phi|_{\phi^* h} \|_{L^1_\loc(\R^2)} \leq \eps.$$

By a smooth truncation and limiting argument we can reduce to the case in which $\phi$ is equal to $\phi(\infty)$ outside of a compact set (our bounds will not depend on the size of this compact set).

We fix $E$ and allow all implied constants to depend on $E$.  By Theorem \ref{caloric-thm} we may find a caloric gauge\footnote{Note that we apply the caloric gauge \emph{after} we change variables by diagonalising $\eta$.  To put this another way, we are not using the Euclidean caloric gauge here, but instead the caloric gauge associated to the metric $\eta$; similarly when we treat the self-similar case later in this section, we shall use a caloric gauge associated to the hyperbolic metric (though we shall conformally map this back to the Euclidean metric for convenience).  Thus we are in fact using multiple (caloric) gauges to study wave maps in this paper.} $e$ for $\phi$, giving rise to the derivative fields $\psi_x, \psi_s$ and connection fields $A_x$.  By hypothesis and \eqref{ps-frame} we have
$$ \| \psi_s(0) \|_{L^1_\loc(\R^2)} \leq \eps.$$
Meanwhile, by \eqref{psis-compar} and \eqref{sdel} we conclude
\begin{equation}\label{psis-m}
 \| \psi_s \|_{L^1_\loc(\R^2)} \lesssim \eps
\end{equation}
for all $s > 0$.

Now we control derivatives of $\psi_s$.  From \eqref{psis-heat} we have
$$ \partial_s \psi_s = \Delta \psi_s + F$$
for some forcing term $F$ of the form
$$ F := O( |A_x| |\partial_x \psi_s| + |\partial_x A_x| |\psi_s| + |A_x|^2 |\psi_s| + |\psi_x|^2 |\psi_s| ).$$
From Duhamel's formula \eqref{duh} we thus see that
$$ \psi_s(s) = e^{\delta s\Delta} \psi_s((1-\delta)s) + \int_{(1-\delta)s}^s e^{(s-s')\Delta} F(s')\ ds'$$
for any $s > 0$, where $0 < \delta < 1/2$ is a parameter to be chosen later.  We differentiate this to obtain
\begin{equation}\label{psisx}
 \partial_x \psi_s(s) = \partial_x e^{\delta s\Delta} \psi_s((1-\delta)s) + \int_{(1-\delta)s}^s \partial_x e^{(s-s')\Delta} F(s')\ ds'
\end{equation}
From Proposition \ref{abound}, Corollary \ref{corbound}, and H\"older's inequality we have
$$ \| F(s') \|_{L^1_x(\R^2)} \lesssim s^{-1} $$
and hence by \eqref{heat-lp}
\begin{equation}\label{fheat}
\| \partial_x e^{(s-s')\Delta} F(s') \|_{L^1_x(\R^2)} \lesssim (s-s')^{-1/2} s^{-1}.
\end{equation}
Meanwhile, an application of Fubini's theorem and \eqref{psis-m} yields that
$$ \| \partial_x e^{\delta s\Delta} \psi_s((1-\delta)s) \|_{L^1_\loc(\R^2)} \lesssim (\delta s)^{-1/2} \| \psi_s((1-\delta)s) \|_{L^1_\loc(\R^2)} \lesssim (\delta s)^{-1/2} \eps.$$
By Minkowski's inequality, we conclude that
$$ \| \partial_x \psi_s \|_{L^1_\loc(\R^2)} \lesssim (\delta s)^{-1/2} \eps + \delta^{1/2} s^{-1/2}.$$
Optimising this in $\delta$, we obtain
$$ \| \partial_x \psi_s \|_{L^1_\loc(\R^2)} \lesssim \eps^{1/2} s^{-1/2}.$$
Also, from \eqref{psis-m}, \eqref{ax-infty} we have
$$ \| A_x \psi_s \|_{L^1_\loc(\R^2)} \leq \| A_x \|_{L^\infty_x(\R^2)} \| \psi_s \|_{L^1_\loc(\R^2)} \lesssim \eps^{1/2} s^{-1/2}.$$
By \eqref{psi-evolve} we thus have
$$ \| \partial_s \psi_x \|_{L^1_\loc(\R^2)} \lesssim \eps^{1/2} s^{-1/2}$$
and thus by the fundamental theorem of calculus and Minkowski's inequality
$$ \| \psi_x(0) \|_{L^1_\loc(\R^2)} \lesssim \eps^{1/2} s^{1/2} + \| \psi_x(s) \|_{L^1_\loc(\R^2)}$$
for any $s > 0$.  On the other hand, from \eqref{linfty-const-ord} we have 
$$ \| \psi_x(s) \|_{L^1_\loc(\R^2)} \lesssim \| \psi_x(s) \|_{L^\infty_x(\R^2)} \lesssim s^{-1/2}.$$
Optimising in $s$ by setting $s := \eps^{-1/2}$, we obtain the claim.
\end{proof}

\subsection{The self-similar case}\label{self-harm}

Now we turn to the analogous elliptic theory required to rule out self-similar wave maps (Theorem \ref{selfsim}(ii)).  As observed in \cite{shatah-struwe}, if $\phi(t,x) = \phi(x/t)$ is a (classical) self-similar wave map on $(-\infty,0) \times \R^2$, then $\phi: \D \to \H$ can be viewed as a harmonic map from the hyperbolic disk\footnote{One can also identify the hyperbolic disk with the hyperbolic space $\H^2 \subset \R^{1+2}$ by identifying $(x_1,x_2)$ with $\frac{(1,x_1,x_2)}{\sqrt{1-x_1^2-x_2^2}}$.  With this identification, the map $f$ below becomes stereographic projection from the ``south pole'' $(-1,0)$ of the upper unit hyperboloid to the unit disk.} $\D = \{ (x_1,x_2): x_1^2+x_2^2 < 1 \}$ with metric $g$ expressible in either Cartesian $(x_1,x_2)$ or polar $(r,\theta)$ coordinates as
\begin{equation}\label{dsdef}
dg^2 := \frac{dx_1^2 + dx_2^2}{1-x_1^2-x_2^2} + \frac{(x_1 dx_1 + x_2 dx_2)^2}{(1-x_1^2-x_2^2)^2} = \frac{dr^2}{(1-r^2)^2} + \frac{r^2 d\theta^2}{1-r^2}
\end{equation}
to the hyperbolic space $\H$, which vanishes at the boundary.
As is well known, the map 
\begin{equation}\label{fconf}
f: x \mapsto \frac{x}{1+\sqrt{1-|x|^2}}
\end{equation}
or equivalently
$$
f^{-1}: y \mapsto \frac{2y}{1+|y|^2}
$$
is a conformal transformation from the hyperbolic disk $(\D, g)$ to the Euclidean disk $(\D, \eta)$; indeed, a direct computation shows that
$$ f^* \eta(x) = \frac{1-|x|^2}{(1+\sqrt{1-|x|^2})^2} g(x)$$
or equivalently
$$ f_* g(y) = \frac{4}{(1-|y|^2)^2} \eta(y).$$
It is also well known that in two dimensions, harmonic maps remain harmonic under conformal change of coordinates.  Thus the pushforward $f_* \phi := \phi \circ f^{-1}$ of $\phi$ is a smooth harmonic map on the Euclidean disk $(\D,\eta)$ that vanishes on the boundary.  A theorem of Lemaire \cite{lemaire} (which is valid for arbitrary target manifolds) then rules out the existence of such maps; in our case of negatively curved targets $\H$, one can also use Bochner-Weitzenb\"ock type identities (as have been used repeatedly in this paper already) to rule out such maps.

The argument of Lemaire relies quite heavily on smoothness (using unique continuation, for instance) and seems to be difficult to extend to the energy class. Nevertheless, the remaining portions of the above argument work quite well in the energy class as long as one avoids the boundary of the hyperbolic or Euclidean disk.  To conclude the argument, one has to understand (approximate or weak) harmonic maps on a slightly smaller disk $\{ |x| \leq 1 - \eps \}$ than the Euclidean disk, which are small but not completely vanishing on the boundary of that disk.  Our tool for this is as follows.

\begin{theorem}[Nonlinear Poincar\'e inequality in Euclidean disks]\label{nlpoin-disk} Let $\overline{\D_{r_0}} := \{ x \in \R^2: |x| \leq r_0\}$ be the closed disk of some radius $r_0 > 0$, and let $\phi: \overline{\D_{r_0}} \to \H$ be a smooth map.  Suppose we have the energy bound
\begin{equation}\label{dro}
 \int_{\overline{\D_{r_0}}} | \partial_x \phi |_{\phi^* h}^2\ dx \leq E
 \end{equation}
the near-harmonicity property
\begin{equation}\label{dro2}
\int_{\overline{\D_{r_0}}} |(\phi^* \nabla)_i \partial_i \phi|_{\phi^* h}\ dx \leq \eps
\end{equation}
and the boundary condition
\begin{equation}\label{boundary}
 \int_{\partial \D_{r_0}} |x_i \partial_i \phi|_{\phi^* h}\ d\sigma \leq \eps
 \end{equation}
where $d\sigma$ is the uniform probability measure on the circle $\partial \D_{r_0}$.  Then
\begin{equation}\label{eps1}
\int_{\overline{\D_{r_0}}} |\partial_i \phi|_{\phi^* h}\ dx \lesssim_E r_0 \eps^{1/4}.
\end{equation}
\end{theorem}

\begin{proof} The theorem is scale-invariant and so we may normalise $r_0=1$.  By a limiting argument we may assume that $\phi$ is equal to a constant $\phi(0) \in \H$ in a neighbourhood of the origin.  We now suppress the dependence of implied constants on $E$.  We can of course take $\eps < 1$ as the claim is trivial (from Cauchy-Schwarz and the energy bound) otherwise.

The idea is to use the Schwartz reflection trick to replace the near-harmonic map on the disk with a near-harmonic map on the plane, so that our previous result (Theorem \ref{nlpoin}) can be applied.  To avoid boundary issues we will use a smooth version of this reflection trick.

Let $0 < \kappa < 1/2$ be a small parameter (which will eventually be sent to zero), and let $f: \R^+ \to [0,1)$ be a smooth function with $f(r) = r$ for $r < 1-\kappa$, $f(r) = 1/r$ for $r > 1+\kappa$, and $f' = O(1)$, $f'' = O(1/\kappa^2)$ for $1-\kappa \leq r \leq 1+\kappa$.  We then define the map $\tilde \phi: \R^2 \to \H$ in polar coordinates $(r,\theta)$ by
$$ \tilde \phi(r,\theta) := \phi( f(r), \theta ).$$
Since $\phi$ is equal to $\phi(0)$ near the origin, we see that $\tilde \phi$ is equal to $\phi(0)$ outside of a compact set.  The map $(r,\theta) \mapsto (f(r),\theta)$ is conformal outside of the annulus $\{ |r-1| \leq \kappa \}$.  One easily checks that the energy $\int | \partial_x \phi |_{\phi^* h}^2\ dx$ and total tension $\int 
|(\phi^* \nabla)_i \partial_i \phi|_{\phi^* h}\ dx$ are preserved under conformal transformations, and so we conclude that
$$ \int_{|r-1| \geq \kappa} | \partial_x \tilde \phi |_{\tilde \phi^* h}^2\ dx \leq 2E$$
and
$$ \int_{|r-1| \geq \kappa} |(\tilde \phi^* \nabla)_i \partial_i \tilde \phi|_{\tilde \phi^* h}\ dx \leq 2\eps$$
(the factor $2$ coming from the two-to-one nature of $f$).  

Now we consider what happens inside the annulus $\{ |r-1| \leq \kappa \}$.  Since $f'=O(1)$ in this region, and $\phi$ is smooth, we see from the chain rule that $| \partial_x \tilde \phi |_{\tilde \phi^* h} = O_\phi(1)$, and so the total contribution to the energy is at most $O_\phi(\kappa)$.  Now we consider the contribution 
\begin{equation}\label{tension}
\int_{|r-1| \leq \kappa} |(\tilde \phi^* \nabla)_i \partial_i \tilde \phi|_{\tilde \phi^* h}\ dx
\end{equation}
to the total tension.  If we write $u = (u^1,u^2): \R^2 \to \D$ for the map $u: (r,\theta) \mapsto (f(r),\theta)$, then the chain rule and product rule gives
$$ \partial_i \tilde \phi = (\partial_i u^j) \partial_j \phi(u)$$
and
\begin{equation}\label{delta}
 (\tilde \phi^* \nabla)_i \partial_i \tilde \phi = (\Delta u^j) \partial_j \phi(u) + (\partial_i u^j) (\partial_i u^k) (\phi^* \nabla)_k \partial_j \phi(u).
 \end{equation}
Since $f' = O(1)$ on the annulus, we have $\nabla u = O(1)$ on this region too, and so the second term in the right-hand side of \eqref{delta} is $O_\phi(1)$ and thus contributes $O_\phi(\kappa)$ to \eqref{tension}.  Meanwhile, a computation in polar coordinates (exploiting the smoothness of $\phi$) shows that
$$ |(\Delta u^j) \partial_j \phi(u(r,\theta))| \lesssim \frac{1}{\kappa} |\partial_r \phi(1,\theta)| + O_\phi(1)$$
in the annulus, and so by \eqref{boundary} the net contribution of this term to \eqref{delta} is $O(\eps) + O_\phi(\kappa)$.  Putting this all together, we see (if $\kappa$ is small enough depending on $\eps$ and $\phi$) that
the energy bound
$$ \int_{\R^2} | \partial_x \tilde \phi |_{\tilde \phi^* h}^2\ dx \lesssim 1$$
and small total tension
$$ \int_{\R^2} |(\tilde \phi^* \nabla)_i \partial_i \tilde \phi|_{\tilde \phi^* h}\ dx \lesssim \eps.$$
In particular
$$ \| |(\tilde \phi^* \nabla)_i \partial_i \tilde \phi|_{\tilde \phi^* h} \|_{L^1_\loc(\R^2)} \lesssim \eps.$$
We can now apply Theorem \ref{nlpoin} to conclude that
$$ \| |\partial_x \tilde \phi|_{\tilde \phi^* h} \|_{L^1_\loc(\R^2)} \lesssim \eps^{1/4}.$$
which implies that
$$ \int_{|x| \leq 1-\kappa} |\partial_x \phi|_{\phi^* h} \lesssim \eps^{1/4}.$$
Sending $\kappa \to 0$ we obtain the claim.
\end{proof}

We now apply a conformal transformation to establish an analogous claim for subdisks of the hyperbolic disk.

\begin{corollary}[Nonlinear Poincar\'e inequality in hyperbolic disks]\label{nlpoin-hyper} Let $-2 \leq t \leq -1$, and $\phi: \overline{\D_{r_0}} \to \H$ be a smooth map on the closed disk $\overline{\D_{r_0}}$ for some $|t|/2 < r_0 < |t|$.  Suppose we have the hyperbolic energy bound
\begin{equation}\label{hyperg}
 \int_{\overline{\D_{r_0}}} \frac{(t^2-r^2)^{1/2}}{t} | \partial_r \phi |_{\phi^* h}^2 + \frac{t}{r^2(t^2-r^2)^{1/2}} |\partial_\theta \phi|_{\phi^* h}^2\ dx \leq E
 \end{equation}
the small total hyperbolic tension property
\begin{equation}\label{hyperten}
\int_{\overline{\D_{r_0}}} \frac{1}{|t|(t^2-r^2)^{1/2}} |t^2 (\phi^* \nabla)_i \partial_i \phi - x_j x_k (\phi^* \nabla)_j \partial_k \phi - 2 x_j \partial_j \phi |_{\phi^* h}\ dx \leq \eps
\end{equation}
and the hyperbolic boundary condition
\begin{equation}\label{boundary-2}
\frac{(t^2-r_0^2)^{1/2}}{t^2} \int_{\partial \D_{r_0}} |x_i \partial_i \phi|_{\phi^* h}\ d\sigma \leq \eps.
\end{equation}
Then
\begin{equation}\label{epso}
 \int_{\overline{\D_{r_0}}} |\partial_r \phi|_{\phi^* h} + \frac{(t^2-r^2)^{1/2}}{t^2 r} |\partial_\theta \phi|_{\phi^* h}\ dx \lesssim_E \eps^{1/4}.
\end{equation}
\end{corollary}

\begin{proof} We can rescale $t = -1$.  We embed $\overline{\D_{r_0}}$ in the hyperbolic disk $(\D, dg^2)$ with metric \eqref{dsdef}.  The Laplace-Beltrami operator $\Delta_g = \nabla_g^\alpha \partial_\alpha$ for this disk, where $\nabla_g$ is of course the Levi-Civita connection on $T\D$ given by the metric $g$, can be computed as
$$ \Delta_g u = (1-|x|^2) ( \partial_i \partial_i u - x_j x_k \partial_j \partial_k u - 2 x_j \partial_j u )$$
and similarly the hyperbolic tension field $g^{\alpha \beta} (\phi^* \nabla \oplus \nabla_g)_\alpha \partial_\beta \phi$ for a map $\phi: \D \to \H$ can be computed as
$$g^{\alpha \beta} (\phi^* \nabla \oplus \nabla_g)_\alpha \partial_\beta \phi = (1-|x|^2)
( (\phi^* \nabla)_i \partial_i \phi - x_j x_k (\phi^* \nabla)_j \partial_k \phi - 2 x_j \partial_j \phi )$$
where the connection $\phi^* \nabla \oplus \nabla_g$ on $\phi^* T\H \oplus T\D$ is the direct sum of the pullback $\phi^* \nabla$ of the Levi-Civita connection $\nabla$ on $T\H$ given by the metric $h$, and the Levi-Civita connection $\nabla_g$ on $T\D$ given by the metric $g$.  The volume measure $dg$ on the hyperbolic disk $\H$ can also be computed as
$$ dg = (1 - |x|^2)^{-3/2}\ dx$$
and the energy density $g^{\alpha \beta} \langle \partial_\alpha \phi, \partial_\beta \phi \rangle_{\phi^* h}$ can be computed in polar coordinates as
$$ g^{\alpha \beta} \langle \partial_\alpha \phi, \partial_\beta \phi \rangle_{\phi^* h} = (1-r^2)^2 |\partial_r \phi|_{\phi^* h}^2 + \frac{1-r^2}{r^2} |\partial_\theta \phi|_{\phi^* h}^2.$$
Thus the hypotheses \eqref{hyperg}, \eqref{hyperten} can be expressed more geometrically as
$$ \int_{\overline{\D_{r_0}}} g^{\alpha \beta} \langle \partial_\alpha \phi, \partial_\beta \phi \rangle_{\phi^* h}\ dg \leq E$$
and
$$ \int_{\overline{\D_{r_0}}} |g^{\alpha \beta} (\phi^* \nabla \oplus \nabla_g)_\alpha \partial_\beta \phi|_{\phi^* h}\ dg \leq \eps.$$
We now apply the conformal transformation \eqref{fconf} to map the hyperbolic disk $(\D,dg^2)$ to the Euclidean disk $(\D,dx^2)$, which maps $\overline{\D_{r_0}}$ to $\overline{\D_{\tilde r_0}}$ where $\tilde r_0 = \frac{r_0}{1+\sqrt{1-r_0}^2}$ is comparable to $1$.  If we let $\tilde \phi := \phi \circ f^{-1}$ be the pushforward of $\phi$ by $f$, we thus see (from the conformal invariance of energy and total tension) that $\tilde \phi$ obeys the hypotheses \eqref{dro}, \eqref{dro2}.  A direct application of the chain rule also lets one deduce \eqref{boundary} from \eqref{boundary-2}.  By Theorem \ref{nlpoin-disk} we see that $\tilde \phi$ obeys \eqref{eps1}, and a final application of the chain rule and change of variables (working in polar coordinates) gives \eqref{epso}.
\end{proof}

\begin{remark} The various powers of $(1-r^2)$ that appear in the above corollary are somewhat unpleasant.  (The powers of $r$ attached to the angular derivatives arise naturally from polar coordinates and do not cause any divergence.) The weights in \eqref{hyperten} and \eqref{epso} turn out to be irrelevant to our arguments, but the weights in \eqref{hyperg} and \eqref{boundary-2} are more delicate to handle.  The weight of $(1-r_0^2)^{1/2}$ in \eqref{boundary-2} is at the critical level, which allows us (barely) to deduce from finite energy assumptions that the left-hand side of \eqref{boundary-2} goes to zero as $r_0 \to 1$.  The negative power of $(1-r^2)^{1/2}$ in the angular derivative term of \eqref{hyperg}, though, causes more difficulty, as it is then not obvious how to control the hyperbolic energy \eqref{hyperg} by the ordinary energy without introducing factors which blow up as $r_0 \to 1$, which would seriously damage the rest of the argument.  Fortunately, as we shall see, there is a conformality argument based on the Hopf differential (also exploited in the argument of Lemaire \cite{lemaire} mentioned earlier) which allows us to control the angular derivatives by the radial ones in a manner that compensates for this negative power.
\end{remark}

\begin{remark} Our arguments here are rather ``extrinsic'' in nature, relying on the existence of explicit conformal mappings to convert problems on the hyperbolic disk to problems on the Euclidean disk and thence to the Euclidean plane.  Presumably one could also work more ``intrinsically'', for instance using a caloric gauge coming from a harmonic map heat flow associated to the metric of the original domain.  This would arguably be the more natural and geometric way to proceed, but would require generalising all the theory of the caloric gauge developed earlier to more general domains than $\R^2$.  Given the vast literature on general harmonic map heat flows, it is likely that this can be accomplished, but we will not attempt to do so here.
\end{remark}

\section{All travelling wave maps are trivial}\label{travel-sec}

We can now prove part (i) of Theorem \ref{selfsim}.  In principle, this follows from Theorem \ref{nlpoin}, but in order to convert rough travelling wave maps into (Lorentz contracted) rough harmonic maps, we will first have to reguarlise these maps by applying the harmonic map heat flow for a small amount of (heat-temporal) time.  But these regularised maps only solve the wave map equation approximately rather than exactly, and we need the full force of the estimates in Section \ref{heatwave-sec} to control for this effect, in particular exploiting the uniform continuity of second time derivatives of the wave map.

We turn to the details. Using Claim \ref{lwp-claim} (and Theorem \ref{energy-claim}) and a standard limiting argument, it suffices to show the following claim about classical wave maps.

\begin{proposition}\label{journey}  Let $I$ be a compact interval, let $v \in \R^2$ be a velocity with $|v| < 1$, and let $(\phi^{(n)},I)$ be a sequence of classical wave maps with $\iota(\phi^{(n)})$ uniformly convergent in $\Energy$, and such that
\begin{equation}\label{shoup}
 \sup_{t \in I} \int_{\R^2} |\partial_t \phi^{(n)} + v \cdot \partial_x \phi^{(n)}|_{(\phi^{(n)})^* h}^2\ dx = o_{n \to \infty}(1).
\end{equation}
Then $\E(\phi^{(n)})$ converges to zero.  (Note that the energy of a classical wave map is conserved in time.)
\end{proposition}

We now prove the proposition.  We use Theorem \ref{dynamic-caloric} to find a dynamic caloric gauge $e^{(n)}$ for $\phi^{(n)}$, with attendant fields $\psi^{(n)}_x$, $\psi^{(n)}_t$, $\psi^{(n)}_s$, $A^{(n)}_x$, $A^{(n)}_t$ and connection $D^{(n)}_x, D^{(n)}_t$.  Up to a time-dependent rotation\footnote{It is probably possible to show that this rotation does not in fact depend on time and so can be eliminated by rotating the frame $e(\infty)$ at infinity, but we will not need to do so here.}, the fields $(\psi^{(n)}_s(\cdot,t,\cdot), \psi^{(n)}_t(0,t,\cdot))$ converge uniformly in $\L$ to a limit $(\psi^{(\infty)}_s(\cdot,t,\cdot), \psi^{(\infty)}_t(0,t,\cdot))$, so in particular
\begin{equation}\label{conv-2}
\int_0^\infty \| |\psi^{(n)}_s(s,t)| - |\psi^{(\infty)}_s(s,t)| \|_{L^2_x(\R^2)}^2\ ds + \| |\psi^{(n)}_t(0,t)| - |\psi^{(\infty)}_t(0,t)| \|_{L^2_x(\R^2)} = o_{n \to \infty}(1)
\end{equation}
uniformly for $t \in I$.  

Since $\iota(\phi^{(n)})$ is convergent in the energy space, the energy of $\phi^{(n)}$ is bounded by some quantity $0 < E < \infty$ independent of $n$.  We now fix $E$ and allow all implied constants to depend on $E$.  We also allow implied constants to depend on the interval $I$ and the velocity $v$.

Write 
\begin{equation}\label{psiv-def}
\psi_v^{(n)} := \psi_t^{(n)} + v \cdot \psi_x^{(n)}.
\end{equation}
From \eqref{shoup} we see that
\begin{equation}\label{psivn}
 \| \psi_v^{(n)}(0,t) \|_{L^2_x(\R^2)} = o_{n \to \infty}(1)
 \end{equation}
uniformly for $t \in I$.  Meanwhile, from \eqref{psix-heat}, \eqref{dst} we see that $\psi_v$ solves the covariant heat equation \eqref{uheat}.  By Lemma \ref{uheat-est} we conclude that
\begin{equation}\label{psiv}
\| \partial_x^k \psi_v^{(n)}(s,t) \|_{L^2_x(\R^2)} = o_{n \to \infty;k}(s^{-k/2})
\end{equation}
uniformly for $t \in I$ and $s > 0$, for every $k \geq 0$; by the Gagliardo-Nirenberg inequality \eqref{gag-2} we conclude
\begin{equation}\label{psiv-infty}
\| \partial_x^k \psi_v^{(n)}(s,t) \|_{L^\infty_x(\R^2)} = o_{n \to \infty;k}(s^{-(k+1)/2})
\end{equation}
for the same range of $t,s,k$.  From \eqref{sax}, \eqref{ast} we also have
$$ A_v := A_t + v \cdot A_x = \int_s^\infty \psi_s \wedge \psi_v(s')\ ds'$$
and hence by Minkowski's inequality and \eqref{ax-infty}, \eqref{l2-const-ord}, \eqref{linfty-const-ord} we see that $A_v$ obeys the same estimates as $\psi_v$, in the sense that
\begin{equation}\label{av}
\| \partial_x^k A_v^{(n)}(s,t) \|_{L^2_x(\R^2)} = o_{n \to \infty;k}(s^{-k/2})
\end{equation}
and
\begin{equation}\label{av-infty}
\| \partial_x^k A_v^{(n)}(s,t) \|_{L^\infty_x(\R^2)} = o_{n \to \infty;k}(s^{-(k+1)/2})
\end{equation}
uniformly for $t \in I$ and $s > 0$, for every $k \geq 0$.

With these bounds we can now get bounds on the time derivatives of $\psi_v$ in the local $L^1$ norm \eqref{morrey}.

\begin{proposition}\label{Tss-prop}  Let $\eps > 0$ and $s_0 > 0$.  Then for all sufficiently large $n$ (depending on $\eps,s_0$), all $t \in I$, and all $s \geq s_0$ we have
\begin{equation}\label{tss}
\| \partial_t \psi^{(n)}_v(s,t) \|_{L^1_\loc(\R^2)} \lesssim \eps.
\end{equation}
\end{proposition}

\begin{proof}  We fix $s \geq s_0$, and omit the $n$ superscripts and the explicit dependence on the $s$ variable.  From \eqref{psivn} we already have
$$ \| \psi_v(t) \|_{L^1_\loc(\R^2)} \lesssim o_{n \to \infty}(1)$$
and hence by the fundamental theorem of calculus
$$ \| \frac{1}{\kappa} \int_{t_1}^{t_2} \partial_t \psi_v(t)\ dt \|_{L^1_\loc(\R^2)} \lesssim o_{n \to \infty;\kappa}(1)$$
whenever $t_1 < t_2$ lie in $I$ with $|t_2-t_1| = \kappa$, and $\kappa$ is a small parameter (at least as small as the quantity in Lemma \ref{kappalem}) to be chosen later.  Thus, by taking $n$ small enough depending on $\kappa$ and $\eps$, it suffices to show that
$$ \| \delta \partial_t \psi_v \|_{L^1_\loc(\R^2)} \lesssim \eps$$
in the notation of the proof of Lemma \ref{kappalem}, whenever $t_1,t_2 \in I$ is such that $|t_2-t_1| \leq \kappa$.

From \eqref{psiv-def}, \eqref{wn-def}, \eqref{zerotor-frame} we have the identity
\begin{equation}\label{wvv}
\partial_t \psi_v = \psi_s - w + v \cdot D_x \psi_v - v_i v_j D_i \psi_j - A_t \psi_v.
\end{equation}
By Lemma \ref{kappalem} we already have
$$ \| \delta w \|_{L^1_\loc(\R^2)} \leq \| \delta w \|_{L^1_x(\R^2)} \lesssim \eps$$
so by the fundamental theorem of calculus it suffices (by taking $\kappa$ sufficiently small) to show that
$$ \| \partial_t \psi_s \|_{L^1_\loc(\R^2)} + \| \partial_t (D_x \psi_v) \|_{L^1_\loc(\R^2)} + \| \partial_t( D_i \psi_j ) \|_{L^1_\loc(\R^2)} + \| \partial_t (A_t \psi_v ) \|_{L^1_\loc(\R^2)} \lesssim_{s_0} 1.$$
But this follows from Lemma \ref{time} and Lemma \ref{basic}.
\end{proof}

From \eqref{wvv}, \eqref{tss}, \eqref{av-infty}, \eqref{psiv}, Lemma \ref{basic}, Lemma \ref{wsmall}, and the triangle inequality we conclude that for every $\eps > 0$ we have
$$ \| \psi^{(n)}_s - v_i v_j D^{(n)}_i \psi^{(n)}_j(1/S) \|_{L^1_\loc(\R^2)} \lesssim \eps,$$
whenever $S$ is sufficiently large depending on $\eps$ and whenever $n$ is sufficiently large depending on $\eps, S$, thus achieving for the first time a smallness bound that is ``elliptic'' in the sense that it does not involve any time differrentiation.  As $|v| < 1$, the expression in the norm can be written here as $\eta^{ij} D^{(n)}_i \psi^{(n)}_j(1/S)$ for some positive definite $\eta^{ij}$ depending only on $v$.  Applying Theorem \ref{nlpoin} (and the subluminal hypothesis $|v| < 1$, we conclude that
$$ \| \psi^{(n)}_x(1/S) \|_{L^1_\loc(\R^2)} \lesssim \eps^{1/4}$$
By \eqref{psix-delta} and Corollary \ref{compar} we have
$$ |\psi^{(n)}_x(s)| \leq e^{(s-1/S)\Delta} |\psi^{(n)}_x(1/S)|$$
for all $s \geq 1/S$, and thus by \eqref{sdel} we have
$$ \| \psi^{(n)}_x(s) \|_{L^1_\loc(\R^2)} \lesssim \eps^{1/4}.$$
Meanwhile, from Lemma \ref{basic} we have
$$ \| \partial_x^2 \psi^{(n)}_x(s) \|_{L^1_\loc(\R^2)} \lesssim s^{-1}$$
and hence by the Gagliardo-Nirenberg inequality \eqref{gag-1} (applied to localised versions of $\psi^{(n)}_x$)
$$ \| \partial_x \psi^{(n)}_x(s) \|_{L^1_\loc(\R^2)} \lesssim s^{-1/2} \eps^{1/8}$$
and thus by \eqref{ps-frame} and Lemma \ref{basic} 
$$ \| \psi^{(n)}_s(s) \|_{L^1_\loc(\R^2)} \lesssim s^{-1/2} \eps^{1/8}.$$
If we let $J \subset (0,+\infty)$ be any compact interval, we thus see (by choosing $\eps$ sufficiently small, $S$ sufficiently large, and assuming $n$ large enough) that
$$ \int_J \| \psi^{(n)}_s(s) \|_{L^1_\loc(\R^2)}\ ds = o_{n \to \infty;J}(1).$$
In particular, for any compact set $K \subset (0,+\infty) \times \R^2$, we have
$$ \lim_{n \to \infty} \int_K |\psi^{(n)}_s| = 0.$$
On the other hand, recall that $\psi^{(n)}_s$ converges in $L^2(\R^+ \times \R^2)$ (and hence locally in $L^1$) to $\psi^{(\infty)}_s$.  We conclude that $\psi^{(\infty)}_s \equiv 0$, and so (by the above mentioned convergence)
$$ \int_0^\infty \int_{\R^2} |\psi^{(n)}_s|^2\ dx ds = o_{n \to \infty}(1)$$
uniformly for $t \in I$.  Using the energy identity (Lemma \ref{energy-ident}) we conclude that
$$ \int_{\R^2} |\psi^{(n)}_x(0,t,x)|^2\ dx = o_{n \to \infty}(1)$$
and then by \eqref{psivn} and the triangle inequality
$$ \int_{\R^2} |\psi^{(n)}_t(0,t,x)|^2\ dx = o_{n \to \infty}(1).$$
Thus $\E(\phi^{(n)})$ converges to zero as claimed, proving Proposition \ref{journey} and thus Theorem \ref{selfsim}(i).

\section{All self-similar wave maps are trivial}\label{self-sec}

We now begin the proof of Theorem \ref{selfsim}(ii).  In principle, this should be a repetition of the arguments of the previous section, with Corollary \ref{nlpoin-hyper} playing the role of Theorem \ref{nlpoin}.  There is a difficulty however arising from the weight $(1-r^2)^{-1/2}$ in the denominator in the bounded hyperbolic energy hypothesis \eqref{hyperg} which prevents one from immediately verifying that hypothesis.  To resolve this we must first perform an initial manipulation using the conservation of the stress-energy tensor \eqref{stress-def} in order to establish a certain angular derivative decay that will allow us to verify \eqref{hyperg}.  For technical reasons we must also establish some spatial decay of the heat flow of wave maps.  At that point we can repeat the arguments used to prove Theorem \ref{selfsim}(i).

\subsection{Stress-energy conservation, and angular derivative decay}

As discussed in Section \ref{self-harm}, classical self-similar wave maps are (after some conformal transformation) equivalent to harmonic maps $\phi: \D \to \H$ on the Euclidean disk which vanish on the boundary.  It is well known that
two-dimensional smooth harmonic maps on compact domains (such as the disk) must be \emph{conformal}, in the sense that the \emph{Hopf differential}
$$ \Psi(x) := |\partial_1 \phi|_{\phi^* h}^2 - |\partial_2 \phi|_{\phi^* h}^2 - 2 i \langle \partial_1 \phi, \partial_2 \phi \rangle_{\phi^* h}$$
vanished identically (or equivalently, that the derivative map $d\phi$ is angle-preserving).  To verify this conformality, one first observes (using the harmonic map equation, or from conservation of the stress-energy tensor) that $\Psi$ is holomorphic with respect to the standard complex structure of the disk $\D$; since $\Psi$ also vanishes on the boundary of the disk, the claim follows.  

The conformality implies in particular that in polar coordinates, one has $|\partial_r \phi|_{\phi^* h} = \frac{1}{r} |\partial_\theta \phi|_{\phi^* h}$ for harmonic maps on the disk which vanish on the boundary.  (One also has $\partial_r \phi$ orthogonal to $\partial_\theta \phi$, but we will not exploit this.)  Undoing the above-mentioned conformal transformations, this shows that for classical self-similar wave maps $\phi: (-\infty,0) \times \R^2 \to \H$, one has 
\begin{equation}\label{heu}
|\partial_r \phi|_{\phi^* h} = \frac{|t|}{r \sqrt{t^2-r^2}} |\partial_\theta \phi|_{\phi^* h}; \langle \partial_r \phi, \partial_\theta \phi \rangle_{\phi^* h} = 0
\end{equation}
inside the light cone (outside this cone, of course, $\phi$ is constant), and away from the spatial origin.  Applying this at time $t=-1$, one can in principle eliminate the divergence caused by the $(1-r^2)^{-1/2}$ in the denominator of \eqref{hyperg}, by estimating the angular component of the energy by the radial component.

It should probably be possible to establish the analogue of the conformality property \eqref{heu} directly for self-similar wave maps in the energy class (of course, it will trivially follow once one proves Theorem \ref{selfsim}(ii)).  The author was not able to accomplish this, but was instead able to obtain a weak averaged version of \eqref{heu} (with an error term) in this class that suffices for the task of establishing the hypothesis \eqref{hyperg}.  Our manipulations will ultimately be based on the holomorphicity of the Hopf differential, although this fact will be heavily disguised by various conformal transformations and also the presence of a time variable.

The precise statement we shall establish is as follows.

\begin{proposition}[Angular derivative decay]  Assume Claim \ref{lwp-claim}. Let $\phi: (-\infty,0) \to \Energy$ be a self-similar energy class wave map with energy $E$, and let $\eps > 0$.  Then 
\begin{equation}\label{ehu}
\int_{-2}^{-1} \int_{|t|-2\eps \leq |x| \leq |t|-\eps} |\partial_\theta \phi|_{\phi^* h}^2\ dx dt \lesssim \eps E.
\end{equation}
\end{proposition}

Note that \eqref{ehu} would follow immediately from \eqref{heu} and energy conservation.  

\begin{proof}  
The main tool we shall use is the pointwise conservation 
\begin{equation}\label{stress-cons}
\partial^\alpha \T_{\alpha \beta} = 0
\end{equation}
of the stress-energy tensor \eqref{stress-def}.  This law is easily verified for classical wave maps, and extends to energy class wave maps by Claim \ref{lwp-claim} and a limiting argument, so long as we now interpret \eqref{stress-cons} in the sense of distributions.

We now contract \eqref{stress-cons} against the vector field\footnote{The choice of this vector field was obtained by the author by starting with a stress-energy based proof of the holomorphicity of the Hopf differential and then laboriously pulling back via the conformal change of variables to the unit hyperboloid, and then extending to the light cone via the scaling vector field $t\partial_t + r \partial_r$.  Presumably there is a more geometric reason why this particular vector field is relevant; certainly the fact that it is tangent to the hyperboloid (or equivalently, Minkowski-orthogonal to the scaling vector field $t \partial_t + r \partial_r$) is very natural geometrically.  More generally, it seems that the computations become more natural in hyperbolic polar coordinates, though the author eventually decided to use cylindrical coordinates $t,r,\theta$ instead.}  $r^2 \partial_t + tr \partial_r$, where of course $r$ is the spatial radial variable.  More precisely, from \eqref{stress-cons} we easily verify the identity
$$ \partial_t (r^2 \T_{00} + t x_i \T_{0i} ) - \partial_i( r^2 \T_{0i} + t x_j \T_{ij} ) = - x_i \T_{0i} - t \T_{ii}.$$
For classical wave maps one computes
$$ x_i \T_{0i} + t \T_{ii} = \langle \partial_t \phi, t \partial_t \phi + r \partial_r \phi \rangle_{h^* \phi}
$$ 
and hence by Cauchy-Schwarz
$$ |x_i \T_{0i} + t \T_{ii}| \lesssim \T_{00}^{1/2} |t \partial_t \phi + r \partial_r \phi|_{h^* \phi}.
$$ 
By approximating the self-similar wave map $\phi$ by classical wave maps using Claim \ref{lwp-claim} and then using \eqref{selfsimilar-def}, we conclude that
$$ x_i \T_{0i} + t \T_{ii} \equiv 0$$
for self-similar wave maps.  Thus we have the distributional identity
$$ \partial_t (r^2 \T_{00} + t x_i \T_{0i} ) - \partial_i( r^2 \T_{0i} + t x_j \T_{ij} ) = 0.$$
Recalling that the divergence $\partial_i X_i$ of a vector field can be expressed in polar coordinates as
$$ \partial_i X_i = (\partial_r + \frac{1}{r}) X_r + \frac{1}{r^2} \partial_\theta X_\theta$$
(as can be seen for instance by computing $\int_{\R^2} f \partial_i X_i$ for scalar test functions $f$ via integration by parts and moving to polar coordinates), we conclude that 
\begin{equation}\label{sa}
 \partial_t (r^2 \T_{00} + t r \T_{0r} ) - (\partial_r+\frac{1}{r})( r^2 \T_{0r} + t r \T_{rr} ) - \frac{1}{r^2} \partial_\theta( r^2 \T_{0\theta} + t r \T_{\theta r} ) = 0
 \end{equation}

For classical wave maps, a computation shows that
$$ r^2 \T_{00} + t x_i \T_{0i} = -\frac{1}{2} r G + \langle t \partial_t \phi + r \partial_r \phi, \frac{r^2}{2t} \partial_t \phi - \frac{r^3}{2t^2} \partial_r \phi + r \partial_r \phi \rangle_{\phi^* h}$$
where $G$ is the quantity
\begin{equation}\label{A-def}
 G := \frac{r^2(t^2-r^2)}{t^2} |\partial_r \phi|_{\phi^* h}^2 - |\partial_\theta \phi|_{\phi^* h}^2
\end{equation}
and thus on taking limits as before we see for self-similar wave maps that
$$ r^2 \T_{00} + t x_i \T_{0i} = -\frac{1}{2} G $$
in the distributional sense, away from the spatial origin.  Similarly, for classical wave maps we have
$$ r^2 \T_{0r} + t r \T_{rr} = \frac{1}{2} \frac{t G}{r} + \frac{r}{2t} |t \partial_t \phi + r \partial_r \phi|_{\phi^* h}^2$$
and thus for self-similar wave maps we have
$$ r^2 \T_{0r} + t r \T_{rr} = \frac{t G}{2r}.$$
Inserting this back into \eqref{sa} we obtain
$$ \partial_t ( - \frac{1}{2} G ) - (\partial_r + \frac{1}{r} ) \frac{t G}{2r} -
\frac{1}{r^2} \partial_\theta( r^2 \T_{0\theta} + t r \T_{\theta r} ) $$
which rearranges to
$$ (r \partial_t + t \partial_r) G = -\frac{2}{r} \partial_\theta( r^2 \T_{0\theta} + t r \T_{\theta r} )$$
and in particular
$$ (r \partial_t + t \partial_r) \int_0^{2\pi} G(t,r,\theta)\ d\theta = 0$$
in the distributional sense for $r > 0$.  Thus the distribution 
\begin{equation}\label{Fdef}
F(t,r) := \int_0^{2\pi} G(t,r,\theta)\ d\theta
\end{equation}
is a function of $t^2-r^2$ only.  

Now let $\Omega$ be the region
$$ \Omega := \{ (t,r): -2 \leq t \leq -1; |t|-2\eps \leq r \leq |t|-\eps \}$$
and $\Omega'$ in the region
$$ \Omega' := \{ (t,r): -10\sqrt{\eps} \leq t \leq -\sqrt{\eps}/10; |t|/10 \leq r \leq 9|t|/10 \}$$
Observe that for every point $(t,r)$ in $\Omega$ there exists $(t',r') \in \Omega'$ with $t^2-r^2 = (t')^2 - (r')^2$ (and thus $F(t,r) = F(t',r')$); indeed there is an arc in $\Omega'$ of length comparable to $\sqrt{\eps}$ with this property.  Applying the change of variables formula we conclude that
$$ |\int_\Omega F(t,r)\ r dr dt| \lesssim \int_{\Omega'} |F(t,r)|\ dr dt.$$
On the other hand, from \eqref{Fdef}, \eqref{A-def} we have
$$ |F(t,r)| \lesssim |t|^2 \int_0^{2\pi} \T_{00}(t,r,\theta)\ d\theta$$
in the cone $r \leq |t|$, and thus (by polar coordinates and energy conservation) we see that
$$ \int_{\Omega'} |F(t,r)|\ dr dt \lesssim \sqrt{\eps} \int_{-10\sqrt{\eps} \leq t \leq -\sqrt{\eps}/10} \int_{\R^2} \T_{00}\ dx dt \lesssim E \eps$$
and thus by conservation of energy
$$ |\int_\Omega F(t,r)\ r dr dt| \lesssim E \eps.$$
From \eqref{fdef}, \eqref{A-def} we thus conclude that
$$ \int_{-2}^{-1} \int_{|t|-2\eps \leq |x| \leq |t|-\eps} \frac{1}{r} |\partial_\theta \phi|_{\phi^* h}^2\ dx dt
\lesssim
\int_{-2}^{-1} \int_{|t|-2\eps \leq |x| \leq |t|-\eps} \frac{r(t^2-r^2)}{t^2} |\partial_r \phi|_{\phi^* h}^2\ dx dt + E \eps.$$
But on the region of integration, we have $\frac{r(t^2-r^2)}{t^2} = O(\eps)$ and so the integrand is $O(\eps \T_{00})$.  By energy conservation, we obtain the claim.
\end{proof}

This leads to an important corollary which will be needed for us to apply Corollary \ref{nlpoin-hyper} later in the argument.

\begin{corollary}[Bounded hyperbolic energy on average]
Assume Claim \ref{lwp-claim}. Let $\phi: (-\infty,0) \to \Energy$ be a self-similar energy class wave map with energy $E$.  Then we have
\begin{equation}\label{ehu-hyper}
\int_{-2}^{-1} \int_{|x| \leq |t|} \frac{(t^2-r^2)^{1/2}}{t} | \partial_r \phi |_{\phi^* h}^2 + \frac{t}{r^2(t^2-r^2)^{1/2}} |\partial_\theta \phi|_{\phi^* h}^2\ dx dt \lesssim E.
\end{equation}
\end{corollary}

\begin{proof} The contribution of $\partial_r \phi$ is clearly acceptable by energy conservation.  The contribution of $\partial_\theta \phi$ can be dealt with by dyadic partitioning of the disk $\{ x: |x| \leq |t| \}$ and \eqref{ehu} (the key point here being that the exponent of $\eps$ on the right-hand side of \eqref{ehu} is strictly greater than $1/2$).
\end{proof}

\subsection{Preliminary reduction}

For the remainder of this section, we assume Claim \ref{lwp-claim}, and let $\phi: (-\infty,0) \to \Energy$ be a self-similar energy class wave map with energy $E$, so in particular we have the angular decay estimate \eqref{ehu}.  We allow all implied constants to depend on $E$.

Using Theorem \ref{energy-claim} and Claim \ref{lwp-claim},
we can find a sequence $(\phi^{(n)}, [-2,-1])$ be a sequence of classical wave maps with $\iota(\phi^{(n)})$ uniformly convergent in $\Energy$ to $\phi$ on $[-2,-1]$.  We use Theorem \ref{dynamic-caloric} to place each classical wave map $\phi^{(n)}$ in a caloric gauge, creating the usual fields $\psi_x, \psi_t, \psi_s, A_x, A_t$.  From Theorem \ref{energy-claim} we see that
\begin{equation}\label{gramcon}
\Gram^{(n)}(t) \hbox{ converges uniformly to } \Gram(t)
\end{equation}
as $n \to \infty$, uniformly in $t \in [-2,-1]$.  From this and Definition \ref{travel} we thus see that
\begin{equation}\label{toast}
 \int_{|x| \geq |t|} |\psi^{(n)}_x|^2 + |\psi^{(n)}_t|^2\ dx = o_{n \to \infty}(1)
 \end{equation}
and
\begin{equation}\label{inscale}
 \int_{|x| \leq |t|} |t \psi^{(n)}_t + x \cdot \psi^{(n)}_x|^2 = o_{n \to \infty}(1)
\end{equation}
uniformly for all $t \in [-2,-1]$.    Also, from \eqref{ehu} and \eqref{gramcon} we have
\begin{equation}\label{shell}
\int_{-2}^{-1} \int_{|t|-2\eps \leq |x| \leq |t|-\eps} |\psi^{(n)}_\theta|^2\ dx dt \lesssim \eps^{1/2} + o_{n \to \infty; \eps}(1)
\end{equation}
for all $\eps > 0$, where $\psi^{(n)}_\theta := r \sin \theta \psi^{(n)}_1 - r \cos \theta \psi^{(n)}_2$ is the angular component of the derivative field.

Suppose that we were able to show that
\begin{equation}\label{limf}
 \liminf_{n \to \infty} \inf_{t \in [-2,-1]} \int_{|x| \leq |t|-\eps_0} \T_{00}^{(n)}(t,x)\ dx = 0
\end{equation}
for all $\eps_0 > 0$.  By finite speed of propagation we have
$$ \int_{|x| \leq 1-\eps} \T_{00}^{(n)}(-1,x)\ dx \leq \inf_{t \in [-2,-1]} \int_{|x| \leq |t|-\eps} \T_{00}^{(n)}(t,x)\ dx$$
and thus from \eqref{limf} and taking limits (using \eqref{gramcon}) we see that
$$ \int_{|x| \leq 1-\eps_0} \T_{00}(-1,x) = 0$$
for all $\eps_0 > 0$, and thus $\T_{00}(-1) \equiv 0$ on the disk $|x| < 1$.  From Definition \ref{travel} we also have $\T_{00}(-1) \equiv 0$ outside this disk, and so $\phi$ has zero energy, and the claim follows.  Thus it will suffice to show \eqref{limf} for each $\eps > 0$.

\subsection{Spatial decay}

As in Section \ref{travel-sec}, we will need to pass to the regularisations $\phi^{(n)}(t,s,\cdot)$ of the classical wave maps $\phi^{(n)}$ in order to conduct our analysis.  A new technical difficulty arises in the self-similar case from the spatial weight $x$ that appears for instance in \eqref{inscale}.  Morally speaking, the wave maps and their heat extensions are primarily localised to the vicinity of the light cone $\{ (t,x): |x| \leq |t| \}$, and $t$ will be localised between $-2$ and $-1$, and so these weights should cause no difficulty.  However, in practice, bounds such as
\eqref{toast} are not quite strong enough to ensure that expressions such as $\|x \cdot \psi^{(n)}_x\|_{L^2_x(\R^2)}$ are bounded uniformly in $n$, which leads to some technical difficulties.  There are at least two ways to resolve this issue.  One is to improve the properties of the approximating wave maps $\phi^{(n)}$ so that they are supported in some uniformly bounded region, e.g. the region $\{ (t,x): -2 \leq t \leq 1; |x| \leq |t|+1\}$.  This seems to be possible, but requires some technical lemmas on approximation of compactly supported data in the energy class by classical data of slightly larger support which turn out to be remarkably annoying to actually prove.  We will therefore adopt an alternate approach, which is to truncate the weights $x$ at some spatial scale $R = R_n$ which is growing slowly with $n$.

More precisely, since $\T_{00}$ vanishes outside of the light cone by Definition \ref{travel}, we know that
$$ \int_{\R^2} |x|^2 \T_{00}\ dx \lesssim 1$$
for all $-2 \leq t \leq -1$.  By \eqref{gramcon}, we conclude that
$$ \sup_{-2 \leq t \leq -1} \int_{\R^2} \min( |x|^2, R^2 ) \T_{00}^{(n)}\ dx \lesssim 1 + o_{n \to \infty;R}(1)$$
for each $R > 0$.  Thus, $R_n$ increases sufficiently slowly to infinity with $n$, then
$$ \sup_{-2 \leq t \leq -1} \int_{\R^2} \min( |x|^2, R_n^2) \T_{00}^{(n)}\ dx \lesssim 1,$$
for all $n$ or equivalently, as phrased in the caloric gauge,
\begin{equation}\label{suptx}
\sup_{-2 \leq t \leq -1} \int_{\R^2} \min(|x|^2, R_n^2) (|\psi^{(n)}_x|^2 + |\psi^{(n)}_t|^2)(t,0,x)\ dx \lesssim 1.
\end{equation}
We now observe that this energy localisation estimate persists for bounded times under the heat flow.

\begin{lemma}[Energy localisation]\label{spread}  With the notation and assumptions as above, we have
$$ \sup_{-2 \leq t \leq -1} \sup_{0 \leq s \leq 1} \int_{\R^2} \min( |x|^2, R_n^2 ) (|\psi^{(n)}_x|^2 + |\psi^{(n)}_t|^2)(t,s,x)\ dx \lesssim 1$$
for all $n$.
\end{lemma}

\begin{proof}  Fix $n$ and $t$; we omit the explicit dependence on these parameters.  Let $\chi(x)$ be a smooth bump function on $\R^2$ equal to $|x|^2$ when $|x| \leq 1/2$ and equal to $1$ for $|x| \geq 1$.    It will suffice to show that
$$ \int_{\R^2} \chi(x/R_n) |\psi_t(s,x)|^2\ dx \lesssim 1$$
for $0 \leq s \leq 1$, and similarly with $\psi_t$ replaced by $\psi_x$.  We shall just prove this for $\psi_t$, as the claim for $\psi_x$ is similar.  From \eqref{dst} we have 
$$ \partial_s |\psi_t|^2 \leq \Delta |\psi_t|^2$$
and so
$$ \partial_s \int \chi(x/R_n) |\psi_t|^2\ dx
\leq \int \Delta(\chi(x/R_n)) |\psi_t|^2\ dx.$$
A computation shows that $\Delta(\chi(x/R_n)) = O(1)$, and hence by energy conservation
$$ \partial_s \int \chi(x/R_n) |\psi_t|^2\ dx = O(1).$$
The claim now follows from \eqref{suptx} and the fundamental theorem of calculus.
\end{proof}

\subsection{Appproximate self-similarity for the heat flow}

The heat flow equation is invariant under parabolic scaling $(s,x) \mapsto (\lambda^2 s, \lambda x)$, and hence one expects the heat flow for a self-similar wave map to be self-similar under the combined scaling $(t,s,x) \to (\lambda t, \lambda^2 s, \lambda x)$.  Expressed infinitesimally, if $t \psi_t + x \cdot \psi_x = 0$ at $s=0$, then one expects $t \psi_t + x \cdot \psi_x + 2s \cdot \psi_s = 0$ for later values of $s$.  To formalise this intuition for our approximately self-similar wave maps $\phi^{(n)}$, we introduce the quantity
$$ \psi^{(n)}_X := t \psi^{(n)}_t + x \cdot \psi^{(n)}_x + 2s \cdot \psi^{(n)}_s$$

\begin{lemma}[Approximate self-similarity]\label{asss}  We have
$$ \int_{|x| \leq R_n} |\psi^{(n)}_X(t,s,x)|^2\ dx \leq o_{n \to \infty}(1)$$
uniformly for all $-2 \leq t \leq -1$ and $0 \leq s \leq 1$, and for all $n$.
\end{lemma}

\begin{proof}  We fix $t$ and $n$, and omit the explicit dependence on these quantities.  For technical reasons we will work with the truncated expression
$$ \tilde \psi_X := t \psi_t + \eta(x/R_n) x \cdot \psi_x + 2s \cdot \psi_s$$
where $\eta$ is a bump function supported on the ball $\{|x| \leq 2\}$ which equals $1$ on the ball $\{|x| \leq 1\}$.

A computation involving \eqref{dst}, \eqref{psix-heat}, \eqref{psis-heat}, and the Leibniz rule shows that
$$ \partial_s \tilde \psi_X = D_i D_i \tilde \psi_X - (\tilde \psi_X \wedge \psi_i) \psi_i
- 2 \partial_i (\eta(x/R_n) x_j) D_i \psi_j - 2 \Delta(\eta(x/R_n) x) \cdot \psi_x + 2 \psi_s$$
We observe that
\begin{align*}
\partial_i (\eta(x/R_n) x_j) D_i \psi_j - \psi_s &= \eta(x/R_n) D_j \psi_j + O( 1_{|x| \geq R_n} |D_x \psi_x| ) - D_j \psi_j \\
&= O( 1_{|x| \geq R_n} |D_x \psi_x| )
\end{align*}
and
$$ \Delta(\chi(x/R_n) x) \cdot \psi_x  = O( R_n^{-1} |\psi_x| )$$
and thus by Lemma \ref{dilemma}
$$ \partial_s |\tilde \psi_X| \leq \Delta |\tilde \psi_X| + O( 1_{|x| \geq R_n} |D_x \psi_x| )
+ O( R_n^{-1} |\psi_x| )$$
and thus
$$ \partial_s \| \tilde \psi_X(s)\|_{L^2_x(\R^2)} \leq O( \| D_x\psi_x\|_{L^2_x(|x| \geq R_n)} 
+ R_n^{-1} \| \psi_x \|_{L^2_x(\R^2)} ).$$
From Corollary \ref{corbound} we have $\| \psi_x \|_{L^2_x(\R^2)} = O(1)$ and  
$\| \nabla_x^2 \psi_x \|_{L^2_x(\R^2)} = O(s^{-1})$.  Meanwhile, from Lemma \ref{spread} we have $\| \psi_x \|_{L^2(|x| \geq R_n/2)} = o_{n \to \infty}(1)$.  From the Gagliardo-Nirenberg inequality \eqref{gag-1} and a localisation argument we conclude that
$$ \| \nabla_x \psi_x \|_{L^2_x(|x| \geq R_n)} = o_{n \to \infty}(s^{-1/2})$$
and hence by \eqref{ax-infty}
$$ \| D_x\psi_x\|_{L^2_x(|x| \geq R_n)} = o_{n \to \infty}(s^{-1/2}).$$
We thus conclude that
$$ \partial_s \| \tilde \psi_X(s)\|_{L^2_x(\R^2)} \leq o_{n \to \infty}(s^{-1/2}) + O(R_n^{-1}).$$
On the other hand, from \eqref{inscale}, \eqref{suptx} we have
$$\| \tilde \psi_X(0)\|_{L^2_x(\R^2)} = o_{n \to \infty}(1).$$
By the fundamental theorem of calculus (and the fact that $R_n \to \infty$) we thus have
$$\| \tilde \psi_X(s)\|_{L^2_x(\R^2)} = o_{n \to \infty}(1).$$
for all $0 \leq s \leq 1$, and the claim follows.
\end{proof}

We can also obtain similar bounds on higher derivatives:

\begin{corollary}[Approximate self-similarity, II]\label{asss2}  Let $k \geq 0$.  If $R_n$ goes to infinity sufficiently slowly, then we have
$$ \int_{|x| \leq R_n} |\nabla_x^k \psi^{(n)}_X(t,s,x)|^2\ dx \leq o_{n \to \infty}(s^{-k})$$
and
$$ \sup_{|x| \leq R_n} |\nabla_x^k \psi^{(n)}_X(t,s,x)| \leq o_{n \to \infty}(s^{-(k+1)/2})$$
uniformly for all $-2 \leq t \leq -1$ and $0 < s \leq 1$, and for all $n$.
\end{corollary}

\begin{proof}  From Lemma \ref{basic} we have
$$ \int_{|x| \leq R_n} |\nabla_x^k \psi^{(n)}_X(t,s,x)|^2\ dx \lesssim_{k,R_n} s^{-k}$$
for all $k \geq 0$.  Applying Lemma \ref{asss2} and the Gagliardo-Nirenberg inequality \eqref{gag-1} (and smoothly truncating $\psi^{(n)}_X$ to the region $|x| \leq R_n$) we conclude
$$ \int_{|x| \leq R_n/2} |\nabla_x^k \psi^{(n)}_X(t,s,x)|^2\ dx \leq o_{n \to \infty;R_n,k}(s^{-k})$$
and
$$ \sup_{|x| \leq R_n/2} |\nabla_x^k \psi^{(n)}_X(t,s,x)| \leq o_{n \to \infty;R_n,k}(s^{-(k+1)/2})$$
and the claim follows for $R_n$ sufficiently slowly growing.
\end{proof}

\subsection{Elliptic bounds}

As in Section \ref{travel-sec}, we need to show that an elliptic quantity (not containing time derivatives) is small.  The key identity here is
\begin{equation}\label{psinx}
 D^{(n)}_{\overline{X}} \psi^{(n)}_X - \psi^{(n)}_X = - t^2 w^{(n)} + t^2 \psi^{(n)}_s - x_i x_j D^{(n)}_i \psi^{(n)}_j - 2 x_i \psi^{(n)}_i - 4 s x_i D_i \psi^{(n)}_s - 6s \psi^{(n)}_s - 4 s^2 \partial_s \psi^{(n)}_s
\end{equation}
where $w^{(n)}$ is the wave-tension field
$$ D^{(n)}_{\overline{X}} = t D^{(n)}_t - x_i D^{(n)}_i - 2s \partial_s;$$
this identity is easily verified using \eqref{zerotor-frame} and the Leibniz rule $D_\alpha (f \varphi) = (\partial_\alpha f) \varphi + f D_\alpha \varphi$ for scalar fields $f$ and vector fields $\varphi$.

Lemma \ref{asss} already lets us control $\psi^{(n)}_X$, and Corollary \ref{asss2} (and Lemma \ref{basic}) lets us control most of $D^{(n)}_{\overline{X}} \psi^{(n)}_X$ except for the time derivative.  And of course the wave-tension field is controlled by Lemma \ref{wprelim}.  To end up controlling a purely spatial expression, we must thus control the time derivative of $\psi^{(n)}_X$.  This is done by the following analogue to Proposition \ref{Tss-prop}:

\begin{proposition}\label{Tss-self}  Let $\eps > 0$, $R \geq 1$, and $0 < s_0 < 1$.  Then for all sufficiently large $n$ (depending on $\eps,s_0, R$), all $-2 \leq t \leq -1$, and all $s_0 \leq s \leq 1$ we have
\begin{equation}\label{tss-self}
\| \partial_t \psi^{(n)}_X \|_{L^1(|x| \leq R)} \lesssim \eps.
\end{equation}
\end{proposition}

\begin{proof}  
As in the proof of Proposition \ref{Tss-prop}, we fix $s_0 \leq s \leq 1$, and omit the $n$ superscripts and the explicit dependence on the $s$ variable.  From Lemma \ref{asss} we already have
$$ \| \psi_X(t) \|_{L^1_\loc(|x| \leq R)} \lesssim o_{n \to \infty;R}(1).$$
Using the fundamental theorem of calculus as in Proposition \ref{Tss-prop}, it thus suffices to show that
$$ \| \delta \partial_t \psi_X \|_{L^1(|x| \leq R)} \lesssim \eps$$
in the notation of the proof of Lemma \ref{kappalem}, whenever $t_1,t_2 \in [-2,-1]$ is such that $|t_2-t_1| \leq \kappa$ for some sufficiently small $\kappa$ (independent of $n$), and $n$ is sufficiently large depending on $\eps,R,\kappa$.

From \eqref{psinx} we can express $\partial_t \psi_X$ as a linear combination of the expressions
$$ A_t \psi_X, \partial_x \psi_X, A_x \psi_X, \partial_s \psi_X, \psi_X, w, \psi_s, \partial_x \psi_x, A_x \psi_x, \psi_x, \partial_x \psi_s, A_x \psi_s, \partial_s \psi_s$$
where the coefficients depend in a smooth manner on $s, t, x$.  We thus need to control the $L^1(|x| \leq R)$ norm of $\delta$ applied to all of the above expressions.  Lemma \ref{kappalem} lets one deal with the $w$ term (if $\kappa$ is small enough).  All the terms involving $\psi_X$ can be handled by Lemma \ref{asss} or Corollary \ref{asss2}, together with Lemma \ref{basic}.  For all the other terms, it suffices by the fundamental theorem of calculus to obtain a bound of $O_{s_0,R}(1)$ on the $L^1(|x| \leq R)$ norms of the time derivatives of these terms, but this follows from Lemma \ref{time} and Lemma \ref{basic}.
\end{proof}

\begin{corollary}\label{star-cross}  If $\eps > 0$, and $S \geq 1$ is sufficiently large depending on $\eps$, then
$$
\| t^2 \psi^{(n)}_s - x_i x_j D^{(n)}_i \psi^{(n)}_j - 2 x_i \psi^{(n)}_i (1/S) \|_{L^1(|x| \leq 10)} \lesssim \eps$$
for all $-2 \leq t \leq -1$ and all sufficiently large $n$ (depending on $\eps, S$).
\end{corollary}

\begin{proof}  We suppress $n$ and write $s :=1/S$.  We use \eqref{psinx} to expand the expression inside the norm as a linear combination of
$$ \partial_t \psi_X, A_t \psi_X, \partial_x \psi_X, A_x \psi_X, s \partial_s \psi_X, w, D_x \psi_s, s A_x \psi_s, s \psi_s, s^2 \partial_s \psi_s.$$
where the coefficients depend smoothly on $t$ and $x$ (but we retain the $s$ dependence in order to exploit the smallness of $s$).  The $\partial_t \psi_X$ term is acceptable by Proposition \ref{Tss-self}.  All other terms involving $\psi_X$ are acceptable by Lemma \ref{asss} or Corollary \ref{asss2}, together with Lemma \ref{basic}.  The $w$ term is acceptable by Proposition \ref{wsmall} if we take $S$ large enough.    From Lemma \ref{basic} the last two terms
$s \psi_s, s^2 \partial_s \psi_s$ have an $L^2(\R^2)$ norm of $O(s^{1/2})$ and are thus acceptable if $S$ is large enough.  The only remaining term to handle is $s D_x \psi_s$.  It would suffice (by H\"older's inequality) to show that
$$ \| s D_x \psi_s\|_{L^2_x(\R^2)} \lesssim o_{s \to 0}(1)$$
uniformly in $n,t$; note that Lemma \ref{basic} just barely fails to establish this.  For this, we return back to construction of the energy space.  Observe from construction that the set $\{ \phi^{(n)}[t]: n \geq 1, -2 \leq t \leq -1 \}$ is precompact in $\Energy$, and thus by \eqref{l-def} the functions $\{ \psi^{(n)}_s(t): n \geq 1, -2 \leq t \leq -1 \}$ is precompact in the space $L^2(\R^+ \times \R^2, ds dx)$.  In particular, these functions are uniformly square-integrable, and thus by monotone convergence we have
$$ \int_0^{s_0} \int_{\R^2} |\psi_s|^2\ dx ds = o_{s_0 \to 0}(1)$$
uniformly in $n, t$.  But from \eqref{psis-heat}, \eqref{energy-ineq} we have
$$ \partial_s \int_{\R^2} |\psi_s|^2\ dx \leq - 2 \int_{\R^2} |D_x \psi_s|^2\ dx$$
whence we conclude
$$ \int_{s_0/2}^{s_0} \int_{\R^2} |D_x \psi_s|^2\ dx ds \lesssim o_{s_0 \to 0}(s_0)$$
and so by the pigeonhole principle there exists $s = (1 - o_{s_0 \to 0}(1)) s_0$ for each $n$ such that
$$ \int_{\R^2} |D_x \psi_s(s)|^2\ dx \lesssim o_{s_0 \to 0}(1).$$
On the other hand, from Lemma \ref{basic} and the Leibniz rule we see that
$$ \partial_s \int_{\R^2} |D_x \psi_s(s)|^2\ dx = O(s^{-1})$$
and so the claim follows from the fundamental theorem of calculus.
\end{proof}

\subsection{Wrapping up}

We are now almost ready to conclude the proof of \eqref{limf}.  We first need to take the estimate
\eqref{ehu-hyper}, which is taking place at $s=0$, and move it to a slightly larger value of $s$.

\begin{lemma}\label{hyperlem}  If $\eps > 0$, and $S \geq 1$ is sufficiently large depending on $\eps$, then
\begin{equation}\label{ehu-hyper-slop}
\int_{-2}^{-1} \int_{|x| \leq |t|-\eps} \frac{(t^2-r^2)^{1/2}}{t} |\psi^{(n)}_r(t,1/S,x)|^2 + \frac{t}{r^2(t^2-r^2)^{1/2}} |\psi^{(n)}_\theta(t,1/S,x)|^2\ dx dt \lesssim E
\end{equation}
whenever $n$ is sufficiently large depending on $\eps, S$.
\end{lemma}

\begin{proof}  We suppress the index $n$.  Let $\eta(t,x)$ be a cutoff function on $[-2,-1] \times \R^2$ which equals $1$ when $|x| \leq |t|-\eps$ and vanishes when $|x| \geq |t|-\eps/2$.  Let 
$$ F(t,s,x) := \eta(t,x) [\frac{(t^2-r^2)^{1/2}}{t} |\psi_r(t,s,x)|^2 + \frac{t}{r^2(t^2-r^2)^{1/2}} |\psi_\theta(t,s,x)|^2].$$
From \eqref{ehu-hyper} and \eqref{gramcon} we already know that
$$
\int_{-2}^{-1} \int_{\R^2} F(t,0,x)\ dx dt \lesssim E$$
if $n$ is large enough depending on $\eps$.  On the other hand, we can express $F$ in the form
$$ F = a_{ij} \psi_i \cdot \psi_j$$
for some smooth compactly supported positive semi-definite function $a_{ij}(t,x)$ on $[-2,-1] \times \R^2$ depending on $\eps$.  From \eqref{psix-heat} and the positive definite nature of $a$ we see that
$$ \partial_s F \leq D_k D_k F - 2 (\partial_k a_{ij}) \partial_k(\psi_i \cdot \psi_j) - (\Delta a_{ij}) \psi_i \cdot \psi_j$$
(this can be seen for instance by diagonalising $a$ at any given point).  Integrating this, we conclude that
$$ \partial_s \int_{\R^2} F\ dx \leq \int_{\R^2} (\Delta a_{ij}) \psi_i \cdot \psi_j\ dx.$$
Since $a$ is smooth compactly supported, we conclude from energy conservation that
$$ \partial_s \int_{\R^2} F\ dx \leq O_{E,a}(1).$$
The claim now follows from the fundamental theorem of calculus.
\end{proof}

Now let $\eps_0, \delta > 0$ be arbitrary.  The energy densities $\{ \T_{00}[t]: -2 \leq t \leq 1 \}$ form a compact subset of $L^1(\R^2)$, and hence there exists $0 < \eps \leq \eps_0$ (depending on $\delta$) such that
\begin{equation}\label{epso2}
\int_{|t| - 2\eps \leq |x| \leq |t|} \T_{00}(t,x)\ dx \lesssim \delta^2
\end{equation}
for all $-2 \leq t \leq -1$.

Now let $S\geq 1$ be sufficiently large depending on $\eps,\delta$, and let $n$ be sufficiently large depending on $\eps, \delta, S$.  
From \eqref{ehu-hyper-slop} and the pigeonhole principle we can find a time $t = t_n \in [-2,-1]$ such that
\begin{equation}\label{hyper-1}
\int_{|x| \leq |t|-\eps/2} \frac{(t^2-r^2)^{1/2}}{t} |\partial_r \phi^{(n)}(t,1/S,x)|_{(\phi^{(n)})^* h}^2 + \frac{t}{r^2(t^2-r^2)^{1/2}} |\partial_\theta \phi^{(n)}(t,1/S,x)|_{(\phi^{(n)})^* h}^2\ dx dt \lesssim E.
\end{equation}
From Corollary \ref{star-cross} (with $\eps$ replaced by $\delta (\eps')^{1/2}$) we also have
\begin{equation}\label{hyper-2}
\int_{|x| \leq |t|-\eps/2} \frac{1}{|t|(t^2-r^2)^{1/2}} |(t^2 (\phi^{(n)})^* \nabla)_i \partial_i \phi^{(n)} - x_j x_k ((\phi^{(n)})^* \nabla)_j \partial_k \phi^{(n)} - 2 x_j \partial_j \phi^{(n)} |_{(\phi^{(n)})^* h}\ dx \lesssim \delta.
\end{equation}
Meanwhile, from \eqref{epso2} and \eqref{gramcon} we thus have
$$
\int_{|t| - 2\eps \leq |x| \leq |t|} \T^{(n)}_{00}(t,x)\ dx \lesssim \delta^2$$
and in particular
$$
\int_{|t| - 2\eps \leq |x| \leq |t|} |\partial_x \phi^{(n)}(t,0,x)|^2_{(\phi^{(n)})^* h}\ dx \lesssim \delta^2.$$
Arguing as in the proof of Lemma \ref{hyperlem} we conclude
$$
\int_{|t| - \eps \leq |x| \leq |t|-\eps/2} |\partial_x \phi^{(n)}(t,1/S,x)|^2_{(\phi^{(n)})^* h}\ dx \lesssim \delta^2.$$
Thus by the pigeonhole principle we can find $|t|-\eps \leq r_0 \leq |t|-\eps/2$ such that
$$
\int_{|x| = r_0} |\partial_x \phi^{(n)}(t,1/S,x)|^2_{(\phi^{(n)})^* h}\ d\sigma(x) \lesssim \delta^2 / \eps$$
(where $d\sigma$ is uniform probability measure) and hence by H\"older's inequality
$$ \frac{(t^2-r_0^2)^{1/2}}{t^2} \int_{|x|=r_0} |x_i \partial_i \phi^{(n)}|_{(\phi^{(n)})^* h}\ d\sigma \leq \delta.$$
We have now have all the hypotheses for Corollary \ref{nlpoin-hyper}, and conclude that
$$
 \int_{\overline{\D_{r_0}}} |\partial_r \phi^{(n)}|_{(\phi^{(n)})^* h} + \frac{(t^2-r^2)^{1/2}}{t^2 r} |\partial_\theta \phi^{(n)}|_{(\phi^{(n)})^* h}\ dx \lesssim_E \delta^{1/4}.$$
In particular this implies that
$$
 \liminf_{n \to \infty} \inf_{t \in [-2,-1]} \int_{|x| \leq |t|-\eps_0} \T_{00}^{(n)}(t,x)\ dx 
 \lesssim_{\eps_0, E} \delta^{1/4}.$$
Since $\delta$ can be arbitrary, we obtain \eqref{limf} as desired.  The proof of Theorem \ref{selfsim}(ii) is now complete.

\end{document}